\documentclass[11pt]{article}

\usepackage{amsmath, amsfonts, amssymb, amsthm,  graphicx, mathtools, enumerate}

\usepackage[blocks, affil-it]{authblk}

\usepackage[numbers, square]{natbib}
\usepackage[CJKbookmarks=true,
            bookmarksnumbered=true,
			bookmarksopen=true,
			colorlinks=true,
			citecolor=red,
			linkcolor=blue,
			anchorcolor=red,
			urlcolor=blue]{hyperref}
\usepackage[usenames]{color}

\usepackage[letterpaper, left=1.2truein, right=1.2truein, top = 1.2truein, bottom = 1.2truein]{geometry}
\usepackage[ruled, vlined, lined, commentsnumbered]{algorithm2e}

\usepackage{prettyref,soul}

\newtheorem{lemma}{Lemma}[section]
\newtheorem{proposition}{Proposition}[section]
\newtheorem{thm}{Theorem}[section]
\newtheorem{definition}{Definition}[section]

\newtheorem{remark}{Remark}[section]

\newrefformat{eq}{(\ref{#1})}
\newrefformat{chap}{Chapter~\ref{#1}}
\newrefformat{sec}{Section~\ref{#1}}
\newrefformat{algo}{Algorithm~\ref{#1}}
\newrefformat{fig}{Fig.~\ref{#1}}
\newrefformat{tab}{Table~\ref{#1}}
\newrefformat{rmk}{Remark~\ref{#1}}
\newrefformat{clm}{Claim~\ref{#1}}
\newrefformat{def}{Definition~\ref{#1}}
\newrefformat{cor}{Corollary~\ref{#1}}
\newrefformat{lmm}{Lemma~\ref{#1}}
\newrefformat{lemma}{Lemma~\ref{#1}}
\newrefformat{prop}{Proposition~\ref{#1}}
\newrefformat{app}{Appendix~\ref{#1}}
\newrefformat{ex}{Example~\ref{#1}}
\newrefformat{exer}{Exercise~\ref{#1}}
\newrefformat{soln}{Solution~\ref{#1}}
\newrefformat{cond}{Condition~\ref{#1}}

\def\text#1{\mbox{\rm #1}}

\DeclarePairedDelimiter{\ceil}{\lceil}{\rceil}

\newcommand{\argmin}{\mathop{\rm argmin}}

\newcommand{\norm}[1]{\|{#1} \|}

\newcommand{\wt}{\widetilde}
\newcommand{\Norm}[1]{\|{#1} \|}

\newcommand{\fnorm}[1]{\|#1\|_{\rm F}}
\newcommand{\opnorm}[1]{\|#1\|_{\rm op}}

\newcommand{\diag}{\mathop{\text{diag}}}
\newcommand{\supp}{{\rm supp}}

\newcommand{\D}{\mathcal{D}}
\newcommand{\U}{\mathcal{U}}
\newcommand{\TV}{{\sf TV}}

\title{Robust Covariance and Scatter Matrix Estimation under Huber's Contamination Model
}
\author[1]{Mengjie Chen}
\author[2]{Chao Gao}
\author[3]{Zhao Ren}
\affil[1]{
University of Chicago
 \\
mengjiechen@uchicago.edu
}
\affil[2]{
University of Chicago
 \\
chaogao@galton.uchicago.edu
}
\affil[3]{
University of Pittsburgh
 \\
 zren@pitt.edu
}

\begin{document}
\maketitle

\begin{abstract}
Covariance matrix estimation is one of the most important problems in statistics. To accommodate the complexity of modern datasets, it is desired to have estimation procedures that not only can incorporate the structural assumptions of covariance matrices, but are also robust to outliers from arbitrary sources. In this paper, we define a new concept called matrix depth and then propose a robust covariance matrix estimator by maximizing the empirical depth function. The proposed estimator is shown to achieve minimax optimal rate under Huber's $\epsilon$-contamination model for estimating covariance/scatter matrices with various structures including bandedness and sparsity.
 	
\smallskip

\end{abstract}



\section{Introduction}



Covariance matrix estimation is one of the most important problems in statistics. The last decade has witnessed the rapid development of statistical theory for covariance matrix estimation under high dimensional settings. Starting from the seminal works of Bickel and Levina \citep{bickel2008covariance,bickel2008regularized}, covariance matrices with a list of different structures can be estimated with optimal theoretical guarantees. Examples include bandable matrix \citep{cai2010optimal}, sparse matrix \citep{lam2009sparsistency,cai2012optimal}, Toeplitz matrix \citep{ren2013optimal} and spiked matrix \citep{birnbaum13,cai13}. For a recent comprehensive review on this topic, see \cite{cai2014estimating}. However, these works do not take into account  the heavy-tailedness of  data and the possible presence of outliers. All  these methods are based on sample covariance matrix, which is shown to have a $1/(n+1)$ breakdown point \citep{hampel1971general}. This means that even if there exists only one arbitrary outlier in the whole dataset, the statistical performance of the estimator can be totally compromised. In this paper, we attempt to tackle the problems of robust covariance matrix estimation under high dimensional settings.

To be more specific, we consider the distribution $(1-\epsilon)N(0,\Sigma)+\epsilon Q$, where $Q$ is an arbitrary distribution that models the outliers and $\epsilon$ is the proportion of contamination. Given i.i.d. observations $X_1,...,X_n$ from this distribution, there are approximately $n\epsilon$ of them distributed according to $Q$, which can influence the performance of an estimator without robustness property. This setting is called $\epsilon$-contamination model, first proposed in a path-breaking paper by Huber \citep{huber1964robust}. In this paper, Huber proposed a robust location estimator and proved its minimax optimality under the $\epsilon$-contamination model. His work suggests an estimator that is optimal under the $\epsilon$-contamination model must achieve statistical efficiency and resistance to outliers simultaneously. Therefore, we view the $\epsilon$-contamination model as a natural framework to develop theories of robust estimation of covariance matrices. The goal of this paper is to propose an estimator of $\Sigma$ that achieves the minimax rate under Huber's $\epsilon$-contamination model. 

To obtain a robust covariance matrix estimator, we propose a new concept called matrix depth. For a $p$-variate distribution $X\sim\mathbb{P}$, the matrix depth of a positive semi-definite $\Gamma\in\mathbb{R}^{p\times p}$ with respect to $\mathbb{P}$ is defined as
\begin{equation}
\D(\Gamma ,\mathbb{P})=\inf_{||u||=1}\min \left\{ \mathbb{P}%
\{|u^{T}X|^{2}\leq u^{T}\Gamma u\},\mathbb{P}\{|u^{T}X|^{2}\geq u^{T}\Gamma
u\}\right\}. \label{eq:defmatrixdep}
\end{equation}
We will show that for $\mathbb{P}=N(0,\Sigma)$, the deepest matrix is $\beta\Sigma$ for some constant multiplier $\beta>0$. Thus, a natural estimator for $\Sigma$ is $\hat{\Gamma}/\beta$ with $\hat{\Gamma}=\arg\max_{\Gamma\succeq 0}\mathcal{D}(\Gamma,\mathbb{P}_n)$. Here, we use the notation $\mathbb{P}_n$ to denote the empirical distribution.

Our definition of matrix depth is parallel to Tukey's depth function \citep{tukey1975mathematics} for a location parameter. The deepest vector according to Tukey's depth is a natural extension of median in the multivariate setting, and thus can be used as a robust location estimator. Zuo and Serfling \citep{zuo2000general} advocated the notion of statistical depth function that satisfies the four properties in \cite{liu1990notion} and verified that Tukey's depth indeed satisfies all these properties while many other depth functions \citep{liu1990notion,oja1983descriptive,rousseeuw1999regression,vardi2000multivariate} do not. The multivariate median defined by Tukey's depth was shown to have high breakdown point \citep{donoho1982breakdown,donoho1983notion,donoho1992breakdown}. The original proposal of the depth function in \cite{tukey1975mathematics} not only provides a way for robust location estimation, but also gives a general way to summarize multivariate data. For example, the depth function can be used to define an index of scatteredness of data \citep{zuo2000nonparametric}. Based on the concept of data depth, a data peeling procedure has been proposed to estimate the covariance matrix.
Specifically, one may trim the data points according to their depths and use the remaining ones to estimate the covariance \citep{donoho1982breakdown,liu1999multivariate}. One may also estimate the covariance through a weighted average with weights that are functions of depths \citep{zuo2005depth}. Though the notion of Tukey's depth is closely related to covariance matrix estimation, depth functions that are directly defined on positive semi-positive matrices are not well explored in the literature.
The need for such a concept has been mentioned in \cite{serfling2004} based on a general framework of location depth functions by \cite{mizera2002depth,Mizera2004}. A proposal that is close in spirit to ours is \cite{zhang2002some}, which also uses the projection idea in Tukey's location depth.
The matrix depth defined in (\ref{eq:defmatrixdep}) offers another option.
Later, we will also define several variants of the matrix depth that take into account the high dimensional structures such as bandedness and sparsity. Those matrix depth functions are powerful tools for robust estimation of structured covariance matrices.

We apply the proposed robust matrix estimator to the problems of estimating banded covariance matrices, bandable covariance matrices, sparse covariance matrices and sparse principal components. We show that in all these cases, the estimators defined by the matrix depth functions achieve the minimax rates of the corresponding $\epsilon$-contamination models under the operator norm. Therefore, the new estimators enjoy both rate optimality and property of resistance to outliers. Interestingly, the minimax rates have a unified expression. That is, $\mathcal{M}(\epsilon)\asymp \max\left\{\mathcal{M}(0),\omega(\epsilon,\mathcal{F})\right\}$, where $\mathcal{M}(\epsilon)$ is the minimax rate for the probability class of distributions $(1-\epsilon)N(0,\Sigma)+\epsilon Q$ ranging over $\Sigma\in\mathcal{F}$ for some covariance matrix class $\mathcal{F}$ and all probability distributions $Q$. The first part $\mathcal{M}(0)$ is the classical minimax rate without contamination. The second part is determined by the quantity $\omega(\epsilon,\mathcal{F})$ called modulus of continuity. Its definition goes back to the fundamental works of Donoho and Liu \citep{donoho1991geometrizing} and Donoho \citep{donoho1994statistical}. A high level interpretation is that the least favorable contamination distribution $Q$ can be chosen in a way that the parameters within $\omega(\epsilon,\mathcal{F})$ under a given loss cannot be distinguished from each other. We establish this phenomenon rigorously through a general lower bound argument for all $\epsilon$-contamination models.

Besides $\epsilon$-contamination models with Gaussian distributions, we show that our proposed estimators also work for general elliptical distributions. To be specific, the setting $(1-\epsilon)P_{\Gamma}+\epsilon Q$ is also considered, where $\Gamma$ is the scatter matrix under the canonical representation of an elliptical distribution. In fact, a characteristic property of the scatter matrix $\Gamma$ of an elliptical distribution is $\mathcal{D}(\Gamma,P_{\Gamma})=1/2$. This property allows the depth function to combine naturally with the elliptical family. The resulting estimators are also shown to achieve the optimal convergence rates. To this end, we can claim that the proposed estimator by matrix depth have two extra robustness properties besides its rate optimality: resistance to outliers and insensitivity to heavy-tailedness. In fact, there are many works in the literature on scatter matrix estimation for elliptical distributions, including \cite{maronna1976robust,tyler1987distribution} in classical settings and \cite{xue2011optimal,han2013eca,han2013optimal,wegkamp2013adaptive,fan2014page,han2014robust,han2014scale,mitra2014multivariate,xue2014rank} in high dimensional settings. However, it still remains open whether these estimators can achieve the minimax rates of  the $\epsilon$-contamination models.


The $\epsilon$-contamination model is a setting where a successful estimator should achieve a good convergence rate and robustness simultaneously. By considering a population variation of the breakdown point which we term as $\delta$-breakdown point, we show in Section \ref{sec:delta-bp} that for a given estimator that has convergence rate $\delta$ under the $\epsilon$-contamination model, its $\delta$-breakdown point is at least $\epsilon$. This suggests convergence under Huber's $\epsilon$-contamination model is a more general notion of robustness than the breakdown point and it provides a unified way to study statistical convergence rate and robustness jointly.

The main contribution of the paper is the derivation of the minimax rates for robust covariance matrix estimation under Huber's $\epsilon$-contamination model, which can be achieved by optimizing over the proposed matrix depth function. We would like to clarify that, in high-dimensional settings, the proposed estimators based on matrix depth are challenging to compute, hence are mainly of theoretical interest. It is interesting and urgent to investigate in the future whether the minimax rates of covariance matrix estimation under Huber's $\epsilon$-contamination model can be achieved by a provable polynomial-time algorithm. For unstructured covariance matrices under low or moderate dimensions (up to $p=10$), the proposed depth-based estimators can be used in practice. We provide an algorithm and perform some numerical studies in the supplementary material \cite{supplement}. An R package is available on the Github at \url{https://github.com/ChenMengjie/DepthDescent}

The paper is organized as follows. First, we revisit Tukey's location depth in Section \ref{sec:location} and discuss the convergence rate of the associated multivariate median. The matrix depth is introduced in Section \ref{sec:matrix} and we use it as a tool to solve various robust structured covariance matrix estimation problems. In Section \ref{sec:elliptical}, we discuss the relationship between matrix depth and elliptical distributions. Results of covariance matrix estimation are extended to scatter matrix estimation for elliptical distributions. 
Section \ref{sec:lower} presents a general result on minimax lower bound for the $\epsilon$-contamination model. In Section \ref{sec:disc}, we discuss some related topics on robust statistics including the connection between breakdown point and the $\epsilon$-contamination model as well as an extension of our notion of matrix depth function to the setting with non-centered observations. All proofs of the theoretical results are given in Section \ref{sec:proof} and the supplementary materials \cite{supplement}. The supplementary materials \cite{supplement} also include some numerical studies of the proposed estimators for unstructured covariance matrices when the dimension is low or moderate.


We close this section by introducing some notation. Throughout the paper, we assume the covariance or scatter matrix of interest is not a zero matrix. Given an integer $d$, we use $[d]$ to denote the set $\{1,2,...,d\}$. For a vector $u=(u_i)$, $\norm{u}=\sqrt{\sum_iu_i^2}$ denotes the $\ell_2$ norm. For a matrix $A=(A_{ij})$, we use $s_k(A)$ to denote its $k$th singular value. The largest and the smallest singular values are denoted as $s_{\max}(A)$ and $s_{\min}(A)$, respectively. The operator norm of $A$ is denoted by $\opnorm{A}=s_{\max}(A)$ and the Frobenius norm by $\fnorm{A}=\sqrt{\sum_{ij}A_{ij}^2}$. When $A=A^T\in\mathbb{R}^{p\times p}$ is symmetric, $\diag(A)$ means the diagonal matrix whose $(i,i)$th entry is $A_{ii}$. Given a subset $J\subset[p]$, $A_{JJ}$ is an $|J|\times |J|$ submatrix, where $|J|$ means the cardinality of $J$. The set $S^{p-1}=\{u\in\mathbb{R}^p:||u||=1\}$ is the unit sphere in $\mathbb{R}^p$. Given two numbers $a,b\in\mathbb{R}$, we use $a\vee b=\max(a,b)$ and $a\wedge b=\min(a,b)$. For two positive sequences $\{a_n\},\{b_n\}$, $a_n\lesssim b_n$ means $a_n\leq Cb_n$ for some constant $C>0$ independent of $n$, and $a_n\asymp b_n$ means $a_n\lesssim b_n$ and $b_n\lesssim a_n$. Given two probability distributions $\mathbb{P},\mathbb{Q}$, the total variation distance is defined by $\sup_B|\mathbb{P}(B)-\mathbb{Q}(B)|$, and the Kullback-Leibler divergence is defined by $D(\mathbb{P}||\mathbb{Q})=\int \log\frac{d\mathbb{P}}{d\mathbb{Q}}d\mathbb{P}$. Throughout the paper, $C,c$ and their variants denote generic constants that do not depend on $n$. Their values may change from line to line.


\section{Prologue: Robust Location Estimation}
\label{sec:location}

We start by the problem of robust location estimation.
Consider i.i.d. observations $X_1,...,X_n\sim \mathbb{P}_{(\epsilon,\theta,Q)}=(1-\epsilon)P_{\theta}+\epsilon Q$, where $P_{\theta}=N(\theta,I_p)$. The goal is to estimate the location parameter $\theta$ from the contaminated data $\{X_i\}_{i=1}^n$. It is known that the sample average is not robust because of its sensitivity to outliers.
We consider Tukey's median (\cite{tukey1974t6,tukey1975mathematics}, see \cite{tukey1977exploratory} as well) as a robust estimator of the location $\theta$. First, we need to introduce Tukey's depth function. For any $\eta\in\mathbb{R}^p$ and a distribution $\mathbb{P}$ on $\mathbb{R}^p$, the Tukey's depth of $\eta$ with respect to $\mathbb{P}$ is defined as
$$\D(\eta,\mathbb{P})=\inf_{u\in S^{p-1}}\mathbb{P}\{u^TX\leq u^T\eta\},\quad\text{where }X\sim\mathbb{P}.$$
Given i.i.d. observations $\{X_i\}_{i=1}^n$, the Tukey's depth of $\eta$ with respect to the observations $\{X_i\}_{i=1}^n$ is defined as
$$\D(\eta,\{X_i\}_{i=1}^n)=\D(\eta,\mathbb{P}_n)=\min_{u\in S^{p-1}}\frac{1}{n}\sum_{i=1}^n\mathbb{I}\{u^TX_i\leq u^T\eta\},$$
where $\mathbb{P}_n=\frac{1}{n}\sum_{i=1}^n\delta_{X_i}$ is the empirical distribution. Then, Tukey's median is defined to be the deepest point with respect to the observations, i.e.,
\begin{equation}
\hat{\theta}=\arg\max_{\eta\in\mathbb{R}^p}\D(\eta,\{X_i\}_{i=1}^n).\label{eqn:def depth}
\end{equation}
When (\ref{eqn:def depth}) has multiple maxima, $\hat{\theta}$ is understood as any vector that attains  the deepest level.
The convergence rate of $\hat{\theta}$ is stated in the following theorem.

\begin{thm}
\label{thm:location upper}
Consider Tukey's median $\hat{\theta}$.
Assume that $\epsilon<1/5$. Then, there exist absolute constants $C,C_1>0$, such that for any $\delta\in(0,1/2)$ satisfying $C_1\left(\frac{p}{n}+\frac{\log(1/\delta)}{n}\right)<1$, we have
$$\norm{\hat{\theta}-\theta}^2\leq C\left(\left(\frac{p}{n}\vee\epsilon^2\right)+\frac{\log(1/\delta)}{n}\right),$$
with $\mathbb{P}_{(\epsilon,\theta,Q)}$-probability at least $1-2\delta$ uniformly over all $\theta$ and $Q$.
\end{thm}

\begin{remark}
By scrutinizing the proof of Theorem \ref{thm:location upper}, the result can hold for any $\epsilon<1/3-c'$ for an arbitrarily small constant $c'$. The critical threshold $1/3$ has a meaning of the highest breakdown point for Tukey's median \citep{donoho1982breakdown,donoho1992breakdown}. Further discussion on the connection between the breakdown point and the $\epsilon$-contamination model is given in Section \ref{sec:disc}.
\end{remark}

\begin{remark}\label{rmk:nothing}
Theorem \ref{thm:location upper} is valid for identity covariance matrix. For
a more general case $P_{\theta }=N(\theta,\Sigma)$, as long as $s_{\max}(\Sigma )\leq M$ with some absolute constant $M>0,$ the result remains
valid. In addition, the result can also be extended to the class of elliptical distributions considered in Section \ref{sec:elliptical}.
\end{remark}

To the best of our knowledge, Theorem \ref{thm:location upper} is the first result in the literature that gives an error bound for Tukey's median under Huber's $\epsilon$-contamination model.
It says that the convergence rate of Tukey's median is $p/n$ in terms of the squared $\ell_2$ loss when $\epsilon^2\lesssim p/n$. Otherwise, the rate is $\epsilon^2$. Therefore, as long as the number of outliers from $Q$ is at the order of $n\epsilon=O(\sqrt{np})$, the convergence rate of Tukey's median is identical to the case when $\epsilon=0$. The next theorem shows that Tukey's median is optimal under the $\epsilon$-contamination model in a minimax sense.

\begin{thm}
\label{thm:location lower}
There exist some absolute constants $C,c>0$ such that
$$\inf_{\hat{\theta}}\sup_{\theta,Q}\mathbb{P}_{(\epsilon,\theta,Q)}\left\{\Norm{\hat{\theta}-\theta}^2\geq C\left(\frac{p}{n}\vee\epsilon^2\right)\right\}\geq c,$$
for any $\epsilon\in[0,1]$.
\end{thm}

Theorem \ref{thm:location lower} provides a minimax lower bound for the $\epsilon$-contamination model. It implies that as long as $\epsilon^2\gtrsim p/n$, the usual minimax rate $p/n$ for estimating $\theta$ is no longer achievable.  It also justifies the optimality of Tukey's median from a minimax perspective. To summarize, Theorem \ref{thm:location upper} and Theorem \ref{thm:location lower} jointly provide a framework for robust statistics that characterize both rate optimality and resistance to outliers simultaneously.

Another natural robust estimator for location is the componentwise median, defined as $\hat{\theta}=(\hat{\theta}_1,...,\hat{\theta}_p)^T$ with $\hat{\theta}_j=\textsf{Median}(\{X_{ij}\}_{i=1}^n)$.
We show that the componentwise median has an inferior convergence rate via the following proposition. 

\begin{proposition}
\label{pro:location cm}
Consider the componentwise median $\hat{\theta}$. There exist absolute constants $C,c>0$ such that
$$\sup_{\theta,Q}\mathbb{P}_{(\epsilon,\theta,Q)}\left\{\Norm{\hat{\theta}-\theta}^2\geq Cp\left(\frac{1}{n}\vee\epsilon^2\right)\right\}\geq c,$$
for any $\epsilon\in[0,1]$.
\end{proposition}

Obviously, $p\left(n^{-1}\vee\epsilon^2\right)$ is also the upper bound for $\hat{\theta}$ by applying Theorem \ref{thm:location upper} to each coordinate. Since $p/n\vee\epsilon^2\ll p\left(n^{-1}\vee\epsilon^2\right)$ when $\epsilon^2\gtrsim 1/n$, the componentwise median has a slower convergence rate. It achieves the rate $p/n$ only when $n\epsilon=O(\sqrt{n})$. Therefore, to preserve the rate $p/n$, the componentwise median can tolerate at most $O(\sqrt{n})$ number of outliers, whereas Tukey's median can tolerate $O(\sqrt{pn})$.


\section{Robust Covariance Matrix Estimation}\label{sec:matrix}

In this section, we consider estimating covariance matrices under the $\epsilon$-contamination model. The model is represented as $\mathbb{P}_{(\epsilon ,\Sigma ,Q)}=(1-\epsilon
)P_{\Sigma }+\epsilon Q$, where $P_{\Sigma }=N(0,\Sigma)$ and $Q$ is any distribution. Motivated by Tukey's depth function for location parameters, we introduce a new concept called matrix depth. The robust matrix estimator is defined as the deepest covariance matrix with respect to the observations. This estimator achieves minimax optimal rates under the $\epsilon$-contamination model.

\subsection{Matrix Depth}

The main idea of Tukey's median is to project  multivariate data onto all one-dimensional subspaces and obtain the deepest point by evaluating depths in those one-dimensional subspaces. Such an idea can be used to estimate covariance matrices. Formally speaking, for $X\sim N(0,\Sigma)$, the population median of $|u^TX|^2$ is $\beta u^T\Sigma u$ for every $u\in S^{p-1}$ with some absolute constant $\beta$ defined later. Thus, an estimator of $\Sigma$ can be obtained by estimating variance on every direction. 

Inspired by the above idea, we define the matrix depth of a positive semi-definite $\Gamma \in \mathbb{R}%
^{p\times p}$ with respect to a distribution $\mathbb{P}$ as
\begin{equation*}
\D(\Gamma ,\mathbb{P})=\inf_{u\in S^{p-1}}\min \left\{ \mathbb{P}%
\{|u^{T}X|^{2}\leq u^{T}\Gamma u\},\mathbb{P}\{|u^{T}X|^{2}\geq u^{T}\Gamma
u\}\right\},
\end{equation*}%
where $X\sim \mathbb{P}$. To adapt to various structure constraints in high-dimensional settings, it is also helpful to define matrix depth by a
subset of the directions $S^{p-1}$. Given a subset $\mathcal{U}\subset S^{p-1}$%
, the matrix depth of $\Gamma $ with respect to a distribution $\mathbb{P}$
relative to $\mathcal{U}$ is defined as
\begin{equation*}
\D_{\mathcal{U}}(\Gamma ,\mathbb{P})={\inf_{u\in \mathcal{U}}\min \left\{
\mathbb{P}\{|u^{T}X|^{2}\leq u^{T}\Gamma u\},\mathbb{P}\{|u^{T}X|^{2}\geq
u^{T}\Gamma u\}\right\}}, 
\end{equation*}%
where $X\sim \mathbb{P}$. We adopt the notation $\D_{S^{p-1}}(\Gamma ,\mathbb{P})=\D(\Gamma ,\mathbb{P})$, and when $\mathcal{U}$ is a singleton set, we use $\D_{u}(\Gamma ,\mathbb{P})$ instead of $\D_{\{u\}}(\Gamma ,\mathbb{P})$. At the population level, the following proposition shows that the true covariance matrix, multiplied by a scalar, is the deepest positive semi-definite matrix.

\begin{proposition} \label{prop:truth}
Define $\beta$ through the equation
\begin{equation}
\Phi(\sqrt{\beta})=\frac{3}{4},\label{eq:beta}
\end{equation}
where $\Phi$ is the cumulative distribution function of $N(0,1)$.
Then, for any $\mathcal{U}\subset S^{p-1}$, we have $\mathcal{D}_{\mathcal{U}}(\beta\Sigma,P_{\Sigma})=\frac{1}{2}$.
\end{proposition}

Given i.i.d. observations $\{X_{i}\}_{i=1}^{n}$ from $\mathbb{P}$, the
matrix depth of $\Gamma $ with respect to $\{X_{i}\}_{i=1}^{n}$ is defined as
\begin{eqnarray}
\D_{\mathcal{U}}(\Gamma ,\{X_{i}\}_{i=1}^{n})&=&\min_{u\in \mathcal{U}}\min \left\{ \frac{1}{n}\sum_{i=1}^{n}\mathbb{I}
\{|u^{T}X_{i}|^{2}\leq u^{T}\Gamma u\},\right.  \label{def:mat depth empi}\\
&& \qquad \qquad \left. \frac{1}{n}\sum_{i=1}^{n}\mathbb{I}\{|u^{T}X_{i}|^{2}\geq u^{T}\Gamma u\}\right\}. \nonumber
\end{eqnarray}%
{Note that there are only $n+1$ possible values for $\frac{1}{n}\sum_{i=1}^{n}\mathbb{I}
\{|u^{T}X_{i}|^{2}\leq u^{T}\Gamma u\}$, which allows us to use minimum rather than infimum when defining the empirical matrix depth function in (\ref{def:mat depth empi}). We adopt the notation $\D_{S^{p-1}}(\Gamma ,\{X_{i}\}_{i=1}^{n})=\D(\Gamma ,\{X_{i}\}_{i=1}^{n})$.} A general estimator for $\beta\Sigma$ is given by
\begin{equation}
\hat{\Gamma}=\arg\max_{\Gamma\in\mathcal{F}}\mathcal{D}_{\mathcal{U}}(\Gamma,\{X_i\}_{i=1}^n),\label{eq:hatgamma}
\end{equation}
where $\mathcal{F}$ is some matrix class to be specified later. One can either use $\mathcal{F}$ to impose various structure constraints in high-dimensional settings or use it to promote positive-definiteness of the estimator.
The estimator of $\Sigma$ is
\begin{equation}
\hat{\Sigma}=\hat{\Gamma}/\beta ,  \label{eq:estimator of Cov}
\end{equation}
where $\beta$ is defined through (\ref{eq:beta}).

\subsection{General Covariance Matrix}

Consider the
following covariance matrix class with bounded spectra
\begin{equation*}
\mathcal{F}(M)=\left\{ \Sigma =\Sigma ^{T}\in \mathbb{R}^{p\times p}:\Sigma
\succeq 0, s_{\max }(\Sigma )\leq M\right\},
\end{equation*}%
where $\Sigma\succeq 0$ means $\Sigma$ is positive semi-definite and $M>0$ is some absolute constant that does not scale with $p$ or $n$.

To define an estimator,  it is natural to pick $\mathcal{U}= S^{p-1}$. Recall we adopt the notation $\D_{S^{p-1}}(\Gamma ,\{X_{i}\}_{i=1}^{n})=\D(\Gamma ,\{X_{i}\}_{i=1}^{n})$. Define
\begin{equation}
\hat{\Gamma}=\arg \max_{\Gamma \succeq 0}\D(\Gamma
,\{X_{i}\}_{i=1}^{n})\mbox{\rm .}  \label{eq:estimator of Gamma}
\end{equation}%
When (\ref{eq:estimator of Gamma}) has multiple maxima, $\hat{\Gamma}$ is understood as any positive semi-definite
matrix that attains the deepest level.
A final estimator of $\Sigma$ is defined by $\hat{\Sigma}=\hat{\Gamma}/\beta$ as in (\ref{eq:estimator of Cov}).
The error bound of $\hat{\Sigma}$ is stated in
the following theorem.

\begin{thm}\label{thm:matrix upper}
Assume that $\epsilon <1/5$. Then, there exist absolute constants $C,C_1>0$, such that for any $\delta\in(0,1/2)$ satisfying $C_1{\frac{p+\log(1/\delta)}{n}}<1$, we have
\begin{equation*}
\opnorm{\hat{\Sigma}-\Sigma}^2 \leq C\left(
\left(\frac{p}{n}\vee\epsilon^2\right)+\frac{\log(1/\delta)}{n}\right) ,
\end{equation*}%
with $\mathbb{P}_{(\epsilon ,\Sigma ,Q)}$-probability at least $1-2\delta$ uniformly over all $Q$
and $\Sigma \in \mathcal{F}(M)$.
\end{thm}


The convergence rate for the deepest covariance is $(p/n)\vee\epsilon^2$ under the squared operator norm. A matching lower bound is given by the following theorem.

\begin{thm}\label{thm:l1}
There exist some absolute constants $C,c>0$ such that
$$\inf_{\hat{\Sigma}}\sup_{\Sigma\in\mathcal{F}(M)}\sup_Q\mathbb{P}_{(\epsilon,\Sigma,Q)}\left\{\opnorm{\hat{\Sigma}-\Sigma}^2\geq C\left(\frac{p}{n}\vee\epsilon^2\right)\right\}\geq c,$$
for any $\epsilon\in[0,1]$.
\end{thm}

Theorem \ref{thm:matrix upper} and Theorem \ref{thm:l1} show that the minimax rate for estimating a covariance matrix under Huber's $\epsilon$-contamination model is $(p/n)\vee\epsilon^2$. The part $p/n$ is the classical parametric rate \citep{davidson2001local} for estimating a covariance matrix without contamination under the squared spectral norm.

\subsection{Bandable Covariance Matrix}\label{sec:banded}

In many high-dimensional applications such as time series data in finance, the covariates of data are collected in an ordered fashion. This leads to a natural banded estimator of the covariance matrix \citep{bickel2008regularized, cai2010optimal}.
Define the class of covariance matrices with a banded structure  by
$$\mathcal{F}_k=\{\Sigma=(\sigma_{ij}) \succeq 0: \sigma_{ij}=0\text{ if }|i-j|>k\}.$$
Next, we propose a notion of matrix depth function relative to some subset $\mathcal{U}_k\subset S^{p-1}$ defined particularly for the class $\mathcal{F}_k$. For any $l_1,l_2\in[p]$, define $\mathcal{V}_{[l_1,l_2]}=\{u=(u_i)\in S^{p-1}: u_i=0\text{ if }i\notin[l_1,l_2]\}$. Then $\mathcal{V}_{[l_1,l_2]}$ is equivalent to $S^{l_2-l_1}$ on the coordinates $\{l_1,...,l_2\}$.  The depth function is defined relatively to the following subset
$$\mathcal{U}_k=\cup_{l=1}^{p+1-2k}\mathcal{V}_{[l,l+2k-1]}\text{ if }2k\leq p,\quad\text{and}\quad\mathcal{U}_k=\mathcal{V}_{[1,p]}=S^{p-1}\text{ if }2k>p.$$
Then, a robust covariance matrix estimator with banded structure is defined as
\begin{equation}
\hat{\Gamma}=\arg\max_{\Gamma\in\mathcal{F}_k}\mathcal{D}_{\mathcal{U}_k}(\Gamma,\{X_i\}_{i=1}^n).\label{eq:bandgamma}
\end{equation}
An estimator for $\Sigma$ is $\hat{\Sigma}=\hat{\Gamma}/\beta$ as in (\ref{eq:estimator of Cov}).


To study the convergence rate of $\hat{\Sigma}$, we consider the class $\mathcal{F}_k(M)=\mathcal{F}_k\cap\mathcal{F}(M)$. The convergence rate of $\hat{\Sigma}$ under the $\epsilon$-contamination model is stated in the following theorem.

\begin{thm} \label{thm:band}
Assume that $\epsilon <1/5$. Then, there exist absolute constants $C,C_1>0$, such that for any $\delta\in(0,1/2)$ satisfying $C_1\left(\frac{k+\log p}{n}.\right)+\left.\frac{\log(1/\delta)}{n}\right)<1$, we have
\begin{equation*}
\opnorm{\hat{\Sigma}-\Sigma}^2 \leq C\left(
\left(\frac{k+\log p}{n}\vee\epsilon^2\right)+\frac{\log(1/\delta)}{n}\right) ,
\end{equation*}
with $\mathbb{P}_{(\epsilon ,\Sigma ,Q)}$-probability at least $1-2\delta$ uniformly over all $Q$
and $\Sigma \in \mathcal{F}_k(M)$.
\end{thm}

Theorem \ref{thm:band} states that the convergence rate for $\hat{\Sigma}$ under the class $\mathcal{F}_k(M)$ is $\frac{k+\log p}{n}\vee\epsilon^2$. When $\epsilon^2\lesssim \frac{k+\log p}{n}$, this is exactly the minimax rate in \cite{cai2010optimal}. Therefore, Theorem \ref{thm:band} extends the result of \cite{cai2010optimal} to a robust setting. If the rate $\frac{k+\log p}{n}$ is pursued, then the maximum number of outliers that $\hat{\Sigma}$ can tolerate is $O(\sqrt{n(k+\log p)})$.

Besides matrices with exact banded structure, we also consider the following class of bandable matrices, in which the variables $X_i$ and $X_j$ become less correlated for larger $|i-j|$. That is,
\begin{eqnarray*}
\mathcal{F}_{\alpha}(M,M_0,M_{\min})&=&\Big\{\Sigma=(\sigma_{ij})\in\mathcal{F}(M): \max_j\sum_{\{i:|i-j|>k\}}|\sigma_{ij}|\leq M_0k^{-\alpha},\\
&& s_{\min}(\Sigma)>M_{\min}\Big\},
\end{eqnarray*}
where $M_0>0$ and $0<M_{\min}<M$ are some absolute constants that do not scale with $p$ or $n$. This covariance class is mainly motivated by many scientific applications such as climatology and spectroscopy. See, for example, \cite{friston1994analysis} and \cite{visser1995trend}. The parameter $\alpha$ specifies how fast the magnitude of $\sigma_{ij}$ decays to zero along the off-diagonal direction.
\begin{thm} \label{thm:bandable}
Consider the robust banded estimator $\hat{\Sigma}$ in Theorem \ref{thm:band} with $k=\ceil{n^{\frac{1}{2\alpha+1}}}\wedge p$.
In addition, we assume that $\epsilon <1/5$. Then, there exist absolute constants $C,C_1>0$, such that for any $\delta\in(0,1/2)$ satisfying $C_1{\frac{\min(n^{\frac{1}{2\alpha+1}}+\log p,p)+\log(1/\delta)}{n}}<1$, we have
\begin{equation*}
\opnorm{\hat{\Sigma}-\Sigma}^2 \leq C\left(
\left(\min\left\{n^{-\frac{2\alpha}{2\alpha+1}}+\frac{\log p}{n},\frac{p}{n}\right\}\vee\epsilon^2\right)+\frac{\log(1/\delta)}{n}\right),
\end{equation*}
with $\mathbb{P}_{(\epsilon ,\Sigma ,Q)}$-probability at least $1-2\delta$ uniformly over all $\Sigma \in \mathcal{F}_{\alpha}(M,M_0,M_{\min})$ and $Q$.
\end{thm}
\begin{remark}
	Unlike Theorem \ref{thm:band}, in Theorems \ref{thm:bandable} we impose a
	condition $s_{\min }(\Sigma )>M_{\min }$ on the smallest eigenvalue of $%
	\Sigma $ while the minimax rate-optimal result in Cai, Zhang and Zhou (2010) does
	not require such a condition in the uncontaminated setting. The reason we
	consider a slightly smaller parameter space is mainly due to our depth-based
	estimation approach. Indeed, since a bandable matrix $\Sigma $ is not
	necessarily banded, the analysis naturally takes a bias-variance tradeoff
	with the pivotal matrix being $\Sigma _{k}=(\sigma _{ij}\mathbb{I}%
	\{|i-j|\leq k\})$, a banded version of $\Sigma $. Our analysis measures
	the bias via the matrix depth. The condition on $s_{\min }(\Sigma )$ guarantees
	the proper behavior of the depth of $\Sigma _{k}$, which can be well
	controlled solely by the bandwidth $k$.
\end{remark}

To close this section, we show in the following theorem that both rates in Theorem \ref{thm:band} and Theorem \ref{thm:bandable} are minimax optimal under the $\epsilon$-contamination model.
\begin{thm}\label{thm:l2}
Assume $p\leq\exp(\gamma n)$ for some $\gamma>0$.
There exist some absolute constants $C,c>0$ such that
$$\inf_{\hat{\Sigma}}\sup_{\Sigma\in\mathcal{F}_k(M)}\sup_Q\mathbb{P}_{(\epsilon,\Sigma,Q)}\left\{\opnorm{\hat{\Sigma}-\Sigma}^2\geq C\left(\frac{k+\log p}{n}\vee\epsilon^2\right)\right\}\geq c,$$
and
$$\inf_{\hat{\Sigma}}\sup_{\Sigma\in\mathcal{F}_{\alpha}}\sup_Q\mathbb{P}_{(\epsilon,\Sigma,Q)}\left\{\opnorm{\hat{\Sigma}-\Sigma}^2\geq C\left(
\min\left\{n^{-\frac{2\alpha}{2\alpha+1}}+\frac{\log p}{n},\frac{p}{n}\right\}\vee\epsilon^2\right)\right\}\geq c,$$
for any $\epsilon\in[0,1]$, where $\mathcal{F}_{\alpha}=\mathcal{F}_{\alpha}(M,M_0,M_{\min})$.
\end{thm}

Theorem \ref{thm:band}, Theorem \ref{thm:bandable} and Theorem \ref{thm:l2} give minimax rates for the classes of banded and bandable covariance matrices. When $\epsilon=0$, the minimax rates of the two classes are given in \cite{cai2010optimal}. Both rates are achieved by a tapered sample covariance estimator when there is no contamination. In comparison, when $\epsilon>0$, we achieve the minimax rate by incorporating the structural assumption into the definition of the matrix depth function.

\subsection{Sparse Covariance Matrix}\label{sec:sparsem}

We consider sparse covariance matrices in this section.
For a subset of coordinates $S\subset[p]$, define $\mathcal{G}(S)=\{G=(g_{ij})\in\mathbb{R}^{p\times p}: g_{ij}=0\text{ if }i\notin S\text{ or }j\notin S\}$. Define $\mathcal{G}(s)=\cup_{S\subset[p]: |S|\leq s}\mathcal{G}(S)$. Then, the sparse covariance class is
$$\mathcal{F}_s=\left\{\Sigma\succeq 0: \Sigma-\text{diag}(\Sigma)\in\mathcal{G}(s)\right\}.$$
In other words, there are $s$ covariates in a block that are correlated with each other. The remaining covariates are independent from this block and from each other. Such sparsity structure has been extensively studied in the problem of sparse principal component analysis \citep{johnstone09,ma13,vu12,cai2013sparse}, and is different from the notion of degree sparsity studied in \cite{bickel2008covariance,cai2012optimal}. Estimating the whole covariance matrix under such sparsity was considered by \cite{cai13}.

To take advantage of the sparsity structure, we define a subset $\mathcal{U}_s\subset S^{p-1}$ for the matrix depth function. For any $S\subset [p]$, define $\mathcal{V}_{S}=\{u=(u_i)\in S^{p-1}: u_i=0\text{ if }i\notin S\}$. The depth function is defined relatively to the following subset
$$\mathcal{U}_s=\cup_{S\subset[p]:|S|=2s}\mathcal{V}_S.$$
A robust sparse covariance matrix estimator is defined by
\begin{equation}
\hat{\Gamma}=\arg\max_{\Gamma\in\mathcal{F}_s}\mathcal{D}_{\mathcal{U}_s}(\Gamma,\{X_i\}_{i=1}^n). \label{eq:sparsegamma}
\end{equation}
An estimator for $\Sigma$ is $\hat{\Sigma}=\hat{\Gamma}/\beta$ as in (\ref{eq:estimator of Cov}).

The error of $\hat{\Sigma}$ is studied in the class $\mathcal{F}_s(M)=\mathcal{F}_s\cap \mathcal{F}(M)$ under the $\epsilon$-contamination model.
\begin{thm}\label{thm:sparse}
Assume that $\epsilon <1/5$. Then, there exist absolute constants $C,C_1>0$, such that for any $\delta\in(0,1/2)$ satisfying $C_1{\frac{s\log\frac{ep}{s}+\log(1/\delta)}{n}}<1$, we have
\begin{equation*}
\opnorm{\hat{\Sigma}-\Sigma}^2 \leq C\left(
\left(\frac{s\log\frac{ep}{s}}{n}\vee\epsilon^2\right)+\frac{\log(1/\delta)}{n}\right) ,
\end{equation*}
with $\mathbb{P}_{(\epsilon ,\Sigma ,Q)}$-probability at least $1-2\delta$ uniformly over all $Q$
and $\Sigma \in \mathcal{F}_s(M)$.
\end{thm}

The next theorem shows that the upper bound in Theorem \ref{thm:sparse} is optimal under the $\epsilon$-contamination model.
\begin{thm}\label{thm:l3}
There are some absolute constants $C,C_1,c>0$ such that as long as $\frac{s\log\frac{ep}{s}}{n}\leq C_1$ holds, then
$$\inf_{\hat{\Sigma}}\sup_{\Sigma\in\mathcal{F}_s(M)}\sup_Q\mathbb{P}_{(\epsilon,\Sigma,Q)}\left\{\opnorm{\hat{\Sigma}-\Sigma}^2\geq C\left(\frac{s\log\frac{ep}{s}}{n}\vee\epsilon^2\right)\right\}\geq c,$$
for any $\epsilon\in[0,1]$.
\end{thm}

Theorem \ref{thm:sparse} and Theorem \ref{thm:l3} together show that the minimax rate of the covariance matrix class $\mathcal{F}_s(M)$ under the $\epsilon$-contamination model is $\frac{s\log\frac{ep}{s}}{n}\vee\epsilon^2$. When $\epsilon=0$, the rate $\frac{s\log\frac{ep}{s}}{n}$ is obtained by \cite{cai13} for a closely related matrix class that is a subset of $\mathcal{F}_s(M)$.

\subsection{Sparse Principal Component Analysis}

As an application of Theorem \ref{thm:sparse}, we consider sparse principal component analysis. We adopt the spiked covariance model \citep{johnstone09,birnbaum13}. That is,
$$\Sigma=V\Lambda V^T+I_p,$$
where $V\in\mathbb{R}^{p\times r}$ is an orthonormal matrix and $\Lambda$ is a diagonal matrix with elements $\lambda_1\geq\lambda_2\geq...\geq\lambda_r>0$.  When $V$ has $s$ nonzero rows \citep{cai2013sparse,cai13}, $\Sigma$ is in the class $\mathcal{F}_s$. The goal is to estimate the subspace projection matrix $VV^T$. We propose a robust estimator by applying singular value decomposition to $\hat{\Gamma}$ in (\ref{eq:sparsegamma}). That is, define $\hat{V}\in O(p,r)$ to be the matrix whose $l$th column is the $l$th eigenvector of $\hat{\Gamma}$. Then, $\hat{V}\hat{V}^T$ is a robust estimator of $VV^T$.

To study the convergence rate of $\hat{V}$, define the covariance matrix class as
\begin{eqnarray*}
\mathcal{F}_{s,\lambda}(M,r)&=&\left\{\Sigma=V\Lambda V^T+I_p: \lambda\leq\lambda_r\leq...\lambda_1\leq M, V\in O(p,r), \right.\\
&& \left. |\supp(V)|\leq s\right\},
\end{eqnarray*}
where $O(p,r)$ is the class of $p\times r$ orthonormal matrices and $\supp(V)$ is the set of nonzero rows of $V$. The rank $r$ is assumed to be bounded by a constant.

\begin{thm}\label{thm:pca}
Assume that $\epsilon < 1/5$. Then, there exist absolute constants $C,C_1,C_2>0$, such that for any $\delta\in(0,1/2)$ satisfying $C_1\left(\left(\frac{s\log\frac{ep}{s}}{n\lambda^2}\vee\frac{\epsilon^2}{\lambda^2}\right)\right.$ $\left.+\frac{\log(1/\delta)}{n\lambda^2}\right)\leq 1$ and $r\leq C_2$,  we have
\begin{equation*}
\fnorm{\hat{V}\hat{V}-VV^T}^2 \leq C\left(
\left(\frac{s\log\frac{ep}{s}}{n\lambda^2}\vee\frac{\epsilon^2}{\lambda^2}\right)+\frac{\log(1/\delta)}{n\lambda^2}\right) ,
\end{equation*}
with $\mathbb{P}_{(\epsilon ,\Sigma ,Q)}$-probability at least $1-2\delta$  uniformly over all $\Sigma \in \mathcal{F}_{s,\lambda}(M,r)$ and $Q$.
\end{thm}

According to Theorem \ref{thm:pca}, the convergence rate for principal subspace estimation is $\frac{s\log\frac{ep}{s}}{n\lambda^2}\vee \frac{\epsilon^2}{\lambda^2}$. We have the rate $\epsilon^2/\lambda^2$ instead of the usual $\epsilon^2$ to account for the outliers in the previous cases. As shown in the next theorem, the rate $\epsilon^2/\lambda^2$ is in fact optimal for sparse principal component analysis.
\begin{thm} \label{thm:l4}
There exist some absolute constants $C,c,c'>0$ such that
$$\inf_{\hat{\Sigma}}\sup_{\Sigma\in\mathcal{F}_{s,\lambda}(M,r)}\sup_Q\mathbb{P}_{(\epsilon,\Sigma,Q)}\left\{\fnorm{\hat{V}\hat{V}-VV^T}^2\geq C\left(\frac{s\log\frac{ep}{s}}{n\lambda^2}\vee\frac{\epsilon^2}{\lambda^2}\right)\wedge c'\right\}\geq c,$$
for any $\epsilon\in[0,1]$.
\end{thm}

Note that Theorem \ref{thm:pca} and Theorem \ref{thm:pca} imply that the minimax rate of sparse PCA under the $\epsilon$-contamination model is $\frac{s\log\frac{ep}{s}}{n\lambda^2}\vee\frac{\epsilon^2}{\lambda^2}$. When $\epsilon=0$, our minimax rate reduces to the case without contamination, which was previously obtained by \cite{cai2013sparse,cai13}. It is interesting that for this class, the term in the minimax rate that characterizes the influence of contamination is $\frac{\epsilon^2}{\lambda^2}$, compared with $\epsilon^2$ in all the previous theorems. We will explain this curious fact by a unified lower bound argument in Section \ref{sec:lower}.

To close this section, we briefly discuss the case where $M$ in various covariance matrix classes is not necessarily a constant. For unstructured covariance class $\mathcal{F}(M)$ in Theorem \ref{thm:matrix upper}, banded covariance class $\mathcal{F}_k(M)$ in Theorem \ref{thm:band}, sparse covariance class $\mathcal{F}_s(M)$ in Theorem \ref{thm:sparse} and spiked covariance class $\mathcal{F}_{s,\lambda}(M,r)$ in Theorem \ref{thm:pca}, all the upper and lower bounds can be readily extended so that the minimax rates with respect to $\opnorm{\hat{\Sigma}-\Sigma}$ or $\fnorm{\hat{V}\hat{V}-VV^T}$ will include an extra factor of $M$. For the bandable class $\mathcal{F}_{\alpha}(M,M_0,M_{\min})$ in Theorem \ref{thm:band}, we can assume all three values $M$, $M_0$, $M_{\min}$ are at the same order and scale together. For this case, all the upper and lower bounds can also be readily extended so that the minimax rates linearly depend on $M$.


\section{Extension to Elliptical Distributions}\label{sec:elliptical}

In Section \ref{sec:matrix}, we considered estimating the covariance matrix under the Gaussian
distribution $P_{\Sigma }=N(0,\Sigma)$. Though we show that our covariance estimator via matrix depth function is robust to arbitrary outliers, it is not clear whether such property also holds under more general distributions. In real applications, the data may not follow a Gaussian distribution and can have very heavy tails. It is even possible that the distribution may not have finite first or second moment. In this section, we extend the Gaussian setting in Section \ref{sec:matrix} to general elliptical distributions. We show that at the population level, the scatter matrix of an elliptical distribution achieves the maximum of the matrix depth function. This fact motivates us to use the matrix depth estimator (\ref{eq:hatgamma}) in the elliptical distribution setting. Indeed, all error bounds we prove under the Gaussian distribution continue to hold under the elliptical distributions. Therefore, the proposed estimator is also adaptive to the shape of the distribution. 
As is pointed out by a referee, the estimator induced by matrix depth is well-defined even if the underlying distribution is not elliptical. It can be interpreted as a multivariate analogue to the median absolute deviation and can serve as a robust scale estimator of the distribution.

We start by introducing the definition of an elliptical distribution.

\begin{definition}[\cite{fang1990symmetric}]
\label{def:elliptical}
A random vector $X\in\mathbb{R}^p$ follows an elliptical distribution if and only if it has the representation $X=\mu+\xi AU$, where $\mu\in\mathbb{R}^p$ and $A\in\mathbb{R}^{p\times r}$ are model parameters. The random variable $U$ is distributed uniformly on the unit sphere $S^{p-1}$ and $\xi\geq 0$ is a random variable in $\mathbb{R}$ independent of $U$. Letting $\Sigma=AA^T$ and we denote $X\sim EC(\mu,\Sigma,\xi)$.
\end{definition}

For simplicity, we consider the model with $\mu=0$.
We want to remark two points on this definition. First, the representation $EC(0,\Sigma,\xi)$ is not unique. This is because $EC(0,\Sigma,\xi)=EC(0,a^{-2}\Sigma,a\xi)$ for any $a>0$. Secondly, for an elliptical random variable $X\sim EC(0,\Sigma,\xi)$ with $s_{\min}(\Sigma) > 0$, given any unit vector $u\in S^{p-1}$, the distribution of $u^TX/\sqrt{u^T\Sigma u}$ is independent of $u$. In other words, $\Sigma^{-1/2}X$ is spherically symmetric. Motivated by these two points, we define the canonical representation of an elliptical distribution as follows.

\begin{definition}
For a non-degenerate elliptical distribution $EC(0,\Sigma,\xi)$ in the sense that $s_{\min}(\Sigma)>0$, $EC(0,\Gamma,\eta)$ is its canonical representation if and only if $\Gamma=a^{-2}\Sigma$ and $\eta=a\xi$ for some $a>0$, and $P_{\Gamma}\left(\frac{|u^TX|^2}{u^T\Gamma u}\leq 1\right)=\frac{1}{2}$, where $P_{\Gamma}=EC(0,\Gamma,\eta)$. From now on, whenever we use $P_{\Gamma}=EC(0,\Gamma,\eta)$, it always denotes the canonical representation.
\end{definition}

To guarantee the existence and  uniqueness of the canonical representation, we need the following assumption on the marginal distribution. Define the distribution function
\begin{equation}
G(t)=P_{\Gamma}\left(\frac{|u^TX|^2}{u^T\Gamma u}\leq t\right).\label{eq:G(t)}
\end{equation}
Note that $G(t)$ does not depend on the specific direction $u\in S^{p-1}$ used in the definition. We assume that $G(t)$ is continuous at $t=1$ and there exist some $\tau\in (0,1/2)$ and $\alpha,\kappa>0$ such that
\begin{equation}
\inf_{|t|\geq \alpha}|G(1)-G(1+t)|\geq \tau\quad\text{and}\quad\inf_{|t|<\alpha}\frac{|G(1)-G(1+t)|}{|t|}\geq\kappa^{-1/2}.\label{eq:assell}
\end{equation}
Intuitively speaking, we require $G(\cdot)$ to be strictly increasing in a neighborhood of $t=1$.
\begin{proposition}\label{prop:unique}
For an elliptical distribution $EC(0,\Gamma,\eta)$ that satisfies (\ref{eq:assell}), its canonical representation exists and is unique.
\end{proposition}

The existence and uniqueness of the canonical representation of $EC(0,\Gamma,\eta)$ imply that the matrix $\Gamma$ is a well-defined object. We call $\Gamma$ the scatter matrix. The following proposition shows that the scatter matrix $\Gamma$ is actually the deepest one with respect to the matrix depth function.
\begin{proposition}\label{prop:deepell}
For any subset $\mathcal{U}\subset S^{p-1}$, we have $\mathcal{D}_{\mathcal{U}}(\Gamma,P_{\Gamma})=\frac{1}{2}$.
\end{proposition}

When $X\sim EC(0,\Gamma,\eta)$ has a density function, it must have the form $p(x)=f(x^T\Gamma^{-1}x)$ for some univariate function $f(\cdot)$ \citep{fang1990symmetric}.
Examples of elliptical distributions include:
\begin{enumerate}
\item \textit{Multivariate Gaussian.} Density function $p(x)\propto \exp(-\beta x^T\Gamma^{-1}x/2)$, where the constant $\beta$ is defined in (\ref{eq:beta}). Proposition \ref{prop:truth} implies that $\beta^{-1}\Gamma$ is the Gaussian covariance matrix.
\item \textit{Multivariate Laplace.} Density function $p(x)\propto \exp(-\sqrt{\beta x^T\Gamma^{-1}x})$, where the constant $\beta$ is determined through the canonical representation. The covariance matrix has formula $(p+1)\beta^{-1}\Gamma$.
\item \textit{Multivariate $t$.} Density function $p(x)\propto \left(1+\beta x^T\Gamma^{-1}x/d\right)^{-\frac{d+p}{2}}$, where $d$ is the degree of freedom. The constant $\beta$ is determined through the canonical representation. When $d>2$, the covariance matrix is $\frac{d}{d-2}\beta^{-1}\Gamma$. Otherwise, the covariance does not exist.
\item \textit{Multivariate Cauchy.} This is a special case of multivariate $t$ distribution when $d=1$. The density function is $p(x)\propto (1+\beta x^T\Gamma^{-1}x)^{-\frac{p+1}{2}}$.
\end{enumerate}

\begin{proposition}\label{prop:example}
For all the four examples above, $\beta$ is an absolute constant independent of $p$. Moreover, the condition (\ref{eq:assell}) holds with absolute constants $\tau,\alpha,\kappa$ independent of $p$.
\end{proposition}

Let us proceed to  consider estimating the scatter matrix $\Gamma$ under the $\epsilon$-contamination model $\mathbb{P}_{(\epsilon,\Gamma,Q)}=(1-\epsilon)P_{\Gamma}+\epsilon Q$. This requires the estimator to be robust in two senses. First, it should be resistant to the outliers. Second, it should be adaptive to the distribution. Using the property of the scatter matrix spelled out in Proposition \ref{prop:deepell}, we show that the depth-induced estimator (\ref{eq:hatgamma}) enjoys optimal rates of convergence in various settings.

\begin{thm}\label{thm:e1}
Consider the estimator $\hat{\Gamma}$ defined in (\ref{eq:estimator of Gamma}). Assume $\epsilon <\tau/3$  and the distribution $P_{\Gamma}=EC(0,\Gamma,\eta)$ satisfies (\ref{eq:assell}). Then, there exist absolute constants $C,C_1>0$, such that for any $\delta\in(0,1/2)$ satisfying $C_1\frac{p+\log(1/\delta)}{n}<1$, we have
$$
\opnorm{\hat{\Gamma}-\Gamma}^2\leq C\kappa\left(\left(\frac{p}{n}\vee\epsilon^2\right)+\frac{\log(1/\delta)}{n}\right),
$$
with $\mathbb{P}_{(\epsilon,\Gamma,Q)}$-probability at least $1-2\delta$
uniformly over all $Q$ and $\Gamma \in \mathcal{F}(M)$.
\end{thm}

\begin{thm}\label{thm:e2}
Consider the estimator $\hat{\Gamma}$ defined in (\ref{eq:bandgamma}).
Assume $\epsilon <\tau/3$  and the distribution $P_{\Gamma}=EC(0,\Gamma,\eta)$ satisfies (\ref{eq:assell}). Then, there exist absolute constants $C,C_1>0$, such that for any $\delta\in(0,1/2)$ satisfying $C_1\frac{k+\log p+\log(1/\delta)}{n}<1$, we have
\begin{equation*}
\opnorm{\hat{\Gamma}-\Gamma}^2 \leq C\kappa\left(
\left(\frac{k+\log p}{n}\vee\epsilon^2\right)+\frac{\log(1/\delta)}{n}\right) ,
\end{equation*}
with $\mathbb{P}_{(\epsilon ,\Gamma ,Q)}$-probability at least $1-2\delta$ uniformly over all $Q$
and $\Gamma \in \mathcal{F}_k(M)$.
\end{thm}

\begin{thm}\label{thm:e3}
Consider the estimator $\hat{\Gamma}$ defined in (\ref{eq:bandgamma}) with $k=\ceil{n^{\frac{1}{2\alpha+1}}}\wedge p$.
Assume $\epsilon <\tau/3$  and the distribution $P_{\Gamma}=EC(0,\Gamma,\eta)$ satisfies (\ref{eq:assell}). Then, there exist absolute constants $C,C_1>0$, such that for any $\delta\in(0,1/2)$ satisfying $C_1\frac{\min(n^{2\alpha+1}+\log p,p)+\log(1/\delta)}{n}<1$, we have
\begin{equation*}
\opnorm{\hat{\Gamma}-\Gamma}^2 \leq C\kappa\left(
\left(\min\left\{n^{-\frac{2\alpha}{2\alpha+1}}+\frac{\log p}{n},\frac{p}{n}\right\}\vee\epsilon^2\right)+\frac{\log(1/\delta)}{n}\right) ,
\end{equation*}
with $\mathbb{P}_{(\epsilon ,\Gamma ,Q)}$-probability at least $1-2\delta$ uniformly over all $\Gamma \in \mathcal{F}_{\alpha}(M,M_0,M_{\min})$ and $Q$.
\end{thm}

\begin{thm}\label{thm:e4}
Consider the estimator $\hat{\Gamma}$ defined in (\ref{eq:sparsegamma}).
Assume $\epsilon <\tau/3$  and the distribution $P_{\Gamma}=EC(0,\Gamma,\eta)$ satisfies (\ref{eq:assell}). Then, there exist absolute constants $C,C_1>0$, such that for any $\delta\in(0,1/2)$ satisfying $C_1\frac{s\log\frac{ep}{s}+\log(1/\delta)}{n}<1$, we have
\begin{equation*}
\opnorm{\hat{\Gamma}-\Gamma}^2 \leq C\kappa\left(
\left(\frac{s\log\frac{ep}{s}}{n}\vee\epsilon^2\right)+\frac{\log(1/\delta)}{n}\right) ,
\end{equation*}
with $\mathbb{P}_{(\epsilon ,\Gamma ,Q)}$-probability at least $1-2\delta$ uniformly over all $Q$
and $\Gamma \in \mathcal{F}_s(M)$.
\end{thm}

\begin{thm}\label{thm:eca}
Consider $\hat{\Gamma}$ defined in (\ref{eq:sparsegamma}), and
define $\hat{V}\in O(p,r)$ to be the matrix whose $l$th column is the $l$th eigenvector of $\hat{\Gamma}$.
Assume the distribution $P_{\Gamma}=EC(0,\Gamma,\eta)$ satisfies (\ref{eq:assell}). Then, there exist absolute constants $C,C_1,C_2>0$, such that for any $\delta\in(0,1/2)$ satisfying $C_1\kappa\left(\left(\frac{s\log\frac{ep}{s}}{n\lambda^2}\vee\frac{\epsilon^2}{\lambda^2}\right)\right.$ $\left.+\frac{\log(1/\delta)}{n\lambda^2}\right)\leq 1$ and $r\leq C_2$,  we have
\begin{equation*}
\fnorm{\hat{V}\hat{V}-VV^T}^2 \leq {C\kappa}\left(
\left(\frac{s\log\frac{ep}{s}}{n\lambda^2}\vee\frac{\epsilon^2}{\lambda^2}\right)+\frac{\log(1/\delta)}{n\lambda^2}\right) ,
\end{equation*}
with $\mathbb{P}_{(\epsilon ,\Gamma ,Q)}$-probability at least $1-2\delta$ uniformly over all $Q$
and $\Gamma \in \mathcal{F}_{s,\lambda}(M,r)$.
\end{thm}

\begin{remark}
Theorem \ref{thm:eca} requires the scatter matrix $\Gamma$ to belong to $\mathcal{F}_{s,\lambda}(M,r)$, which means that $\Gamma=V\Lambda V^T+I_p$. While the $I_p$ part has a clear meaning for covariance matrix, it may not be a suitable way of modeling the scatter matrix. However, we may consider a more general space which contains $\Gamma=V\Lambda V^T+\sigma^2 I_p$ for some absolute constant $\sigma^2$  bounded in some interval $[M^{-1},M]$. Then, the result of Theorem \ref{thm:eca} still holds.
\end{remark}

\begin{remark}
The problem of finding the leading principal subspace for $EC(0,\Gamma,\eta)$ was coined as elliptical component analysis by \cite{han2013eca}. While \cite{han2013eca} extended sparse principal component analysis to the elliptical distributions, the influence of outliers was not investigated.
In comparison, we show that our estimator is robust to both heavy-tailed distributions and the presence of outliers.
\end{remark}

\begin{remark}
Theorems \ref{thm:e1}-\ref{thm:eca} identify a linear dependence on the number $\kappa$ in the error bounds. This dependence was previously revealed in the literature when $\epsilon=0$ and $p=1$. In this case, our proposed estimator is the median absolute deviation that enjoys asymptotic normality $\sqrt{n}(\hat{\gamma}-\gamma)\leadsto N\left(0,\frac{1}{4|G'(1)|^2}\right)$ (see Example 5.24 in \cite{van2000asymptotic}). Given the fact that $|G(1)-G(1+t)|/|t|\approx |G'(1)|$ when $t$ is small, $\kappa$ plays a similar role as $|G'(1)|^{-2}$.
\end{remark}

To close this section, we remark that the estimators via matrix depth function does not require the knowledge of the exact elliptical distribution. They are adaptive to all $EC(0,\Gamma,\eta)$ that satisfy the condition (\ref{eq:assell}). Since the class of elliptical distributions includes multivariate Gaussian as a special case, the lower bounds in Section \ref{sec:matrix} imply that the convergence rates obtained in this section are optimal.


\section{A General Minimax Lower Bound}\label{sec:lower}

In this section, we provide a general minimax theory for $\epsilon$-contamination model. Given a general statistical experiment $\{P_{\theta}:\theta\in\Theta\}$, recall the notation $\mathbb{P}_{(\epsilon,\theta,Q)}=(1-\epsilon)P_{\theta}+\epsilon Q$. If we denote the minimax rate for the class $\{\mathbb{P}_{(\epsilon,\theta,Q)}:\theta\in\Theta,Q\}$ under some loss function $L(\theta_1,\theta_2)$ by $\mathcal{M}(\epsilon)$, then most rates we obtained in Section \ref{sec:location} and Section \ref{sec:matrix} can be written as $\mathcal{M}(\epsilon)\asymp \mathcal{M}(0)\vee\epsilon^2$. The only exception is $\mathcal{M}(\epsilon)\asymp \mathcal{M}(0)\vee(\epsilon^2/\lambda^2)$ for sparse principal component analysis. Therefore, a natural question is whether we can have a general theory for the $\epsilon$-contamination model that governs those minimax rates. The answer for this question lies in a key quantity called modulus of continuity, whose definition goes back to the seminal works of Dohono and Liu \citep{donoho1991geometrizing} and Donoho \citep{donoho1994statistical}.

The modulus of continuity for the $\epsilon$-contamination model is defined as
\begin{equation}
\omega(\epsilon,\Theta)=\sup\left\{L(\theta_1,\theta_2): \TV(P_{\theta_1},P_{\theta_2})\leq\epsilon/(1-\epsilon);\theta_1,\theta_2\in\Theta\right\}.\label{eq:mod}
\end{equation}
The quantity $\omega(\epsilon,\Theta)$ measures the ability of the loss $L(\theta_1,\theta_2)$ to distinguish two distributions $P_{\theta_1}$ and $P_{\theta_2}$ that are close in total variation at the order of $\epsilon$. A high level interpretation is that two distributions $P_{\theta_1}$ and $P_{\theta_2}$ as close as $\epsilon/(1-\epsilon)$ under total variation distance cannot be distinguished at the presence of arbitrary contamination distribution with proportion $\epsilon$. Thus, an error at the order of $\omega(\epsilon,\Theta)$ cannot be avoided for the loss $L(\cdot,\cdot)$.
A general minimax lower bound depending on the modulus of continuity is stated in the following theorem.

\begin{thm}\label{thm:lower}
Suppose there is some $\mathcal{M}(0)$ such that for $\epsilon=0$
\begin{equation}
\inf_{\hat{\theta}}\sup_{\theta\in\Theta}\sup_Q\mathbb{P}_{(\epsilon,\theta,Q)}\left\{L(\hat{\theta},\theta)\geq\mathcal{M}(\epsilon)\right\}\geq c\label{eq:lower}
\end{equation}
holds. Then for any $\epsilon\in[0,1]$, (\ref{eq:lower}) holds for $\mathcal{M}(\epsilon)\asymp \mathcal{M}(0)\vee\omega(\epsilon,\Theta)$.
\end{thm}

Theorem \ref{thm:lower} shows that the quantity $\omega(\epsilon,\Theta)$ is the price of robustness one has to pay in the minimax rate.
To illustrate this result, let us consider the location model in Section \ref{sec:location} where $P_{\theta}=N(\theta,I_p)$. Since $\norm{\theta_1-\theta_2}^2=2D(P_{\theta_1}||P_{\theta_2})\geq 4\TV(P_{\theta_1},P_{\theta_2})^2$, we have $\omega(\epsilon,\Theta)\gtrsim \epsilon^2$. Besides, it is well known  that $\mathcal{M}(0)\asymp p/n$ for the location model, and thus we obtain the rate $(p/n)\vee\epsilon^2$ as the lower bound, which implies Theorem \ref{thm:location lower}. Similar calculation can also be done for the covariance model. In particular, for sparse principal component analysis, we get $\omega(\epsilon,\Theta)\asymp (\epsilon/\lambda)^2$. The details of derivation are given in the supplementary material \cite{supplement}.


\section{Discussion}\label{sec:disc}

\subsection{Impact of Contamination on Convergence Rates}

For all the problems we consider in this paper, the minimax rate under the $\epsilon$-contamination model has the expression $\mathcal{M}(\epsilon)\asymp\mathcal{M}(0)\vee\omega(\epsilon,\Theta)$. Define
$$\epsilon^*=\sup\left\{\epsilon: \omega(\epsilon,\Theta)\leq \mathcal{M}(0)\right\}.$$
Then, $\epsilon^*$ is the maximal proportion of outliers under which the minimax rate obtained without outliers can still be preserved. Thus,  $n\epsilon^*$ is the maximal expected number of outliers for an optimal procedure to achieve the minimax rate as if there is no contamination.

Compared to the minimax rate, consistency is easier to achieve. Suppose $\mathcal{M}(0)=o(1)$, then the necessary and  sufficient condition for consistency is $\omega(\epsilon,\Theta)=o(1)$. In most cases where $\omega(\epsilon,\Theta)\asymp \epsilon^2$, the condition reduces to $\epsilon=o(1)$, meaning that as long as the expected number of outliers is at a smaller order of $n$, the optimal procedure is consistent under the $\epsilon$-contamination model.

\subsection{Non-centered Observations} \label{sec:disc_center}

In previous sections, we assume that the observations are sampled from a centered distribution. This is essential for the proposed matrix depth method to work. It is important to extend our method to non-centered data in order to make it more practical. 

For the Gaussian case, our inspiration is from the simple fact that $\frac{1}{\sqrt{2}}(X_1-X_2)\sim N(0,\Sigma)$, where $X_1,X_2\sim N(\theta,\Sigma)$ are independent observations with with $\theta\in\mathbb{R}^p$ being an arbitrary mean vector. This motivates the following definition of a U-version empirical matrix depth function. That is
\begin{eqnarray*}
	\bar{\D}_{\U}(\Gamma,\{X_i\}_{i=1}^n) &=& \min_{u\in\U}\min\left\{\frac{1}{{n\choose 2}}\sum_{i<j}\mathbb{I}\{|u^T(X_i-X_j)|^2\leq 2u^T\Gamma u\},\right. \\
	&& \left.\frac{1}{{n\choose 2}}\sum_{i<j}\mathbb{I}\{|u^T(X_i-X_j)|^2\geq 2u^T\Gamma u\}\right\}.
\end{eqnarray*}
Then, a covariance matrix estimator $\hat{\Sigma}$ is defined through (\ref{eq:hatgamma}) and (\ref{eq:estimator of Cov}) with ${\D}_{\U}(\Gamma,\{X_i\}_{i=1}^n)$ replaced by $\bar{\D}_{\U}(\Gamma,\{X_i\}_{i=1}^n)$. A similar pairwise difference trick was used by \cite{dumbgen1998tyler} in a different setting. It turns out that all the non-asymptotic bounds in Section \ref{sec:matrix} continue to hold for this new estimator. Due to limited space, we provide more details in Section A of the supplementary material \cite{supplement}, including the extension to the non-centered elliptical distributions, based on an extension of the concentration inequality for suprema of some empirical process to its corresponding U-process.

\subsection{Connection with $\delta$-Breakdown Point}\label{sec:delta-bp}

The notion of breakdown point \citep{hampel1971general} has been widely used to quantify the influence of outliers for a given estimator. Its relation to the $\epsilon$-contamination model was previously explored through the notion of maximum bias in the context of robust covariance matrix estimation. See, e.g., \cite{zuo2005depth}. In this section, we discuss the connection between a population variation of the breakdown point and Huber's $\epsilon$-contamination model.
Let us start by the definition given in \cite{donoho1982breakdown,donoho1983notion,donoho1992breakdown}. Consider the observations $\{X_i\}_{i=1}^n$ that consist of two parts $\{Y_i\}_{i=1}^{n_1}$ and $\{Z_i\}_{i=1}^{n_2}$ with $n_1+n_2=n$. We view $\{Z_i\}_{i=1}^{n_2}$ as the outliers. Then, a robust estimator $\hat{\theta}(\cdot)$ should not be influenced much by the outliers if the proportion $n_2/(n_1+n_2)$ is small. The breakdown point of $\hat{\theta}$ with respect to $\mathcal{Y}$ is defined as
\begin{equation}
\epsilon(\hat{\theta},\mathcal{Y})=\min\left\{\frac{n_2}{n_1+n_2}: \sup_{\{Y_i\}_{i=1}^{n_1}\in\mathcal{Y}}\sup_{\{Z_i\}_{i=1}^{n_2}}\left\|\hat{\theta}(\{Y_i\}_{i=1}^{n_1})-\hat{\theta}(\{X_i\}_{i=1}^{n})\right\|=\infty\right\},\label{eq:dohonodep}
\end{equation}
where $\norm{\cdot}$ is some norm. In its original form, the {supremum} of $\{Y_i\}_{i=1}^{n_1}$ over $\mathcal{Y}$ does not appear in the definition. However, $\{Y_i\}_{i=1}^{n_1}$ are usually assumed to be in a general position or follow some distribution. Thus, it is natural to apply this modification.
Now let us consider the $\epsilon$-contamination model $\mathbb{P}_{(\epsilon,\theta,Q)}=(1-\epsilon){P}_{\theta}+\epsilon Q$. For i.i.d. observations $X_1,...,X_n\sim \mathbb{P}_{(\epsilon,\theta,Q)}$, it can be decomposed into two parts $\{Y_i\}_{i=1}^{n_1}$ and $\{Z_i\}_{i=1}^{n_2}$, where $n_2\sim \text{Binomial}(n,\epsilon)$ and $n_1=n-n_2$. Conditioning on $n_1$, $Y_1,...,Y_{n_1}\sim P_{\theta}$ and $Z_1,...,Z_{n_2}\sim Q$. Observe that $\frac{n_2}{n_1+n_2}\approx\epsilon$, which means the $\epsilon$ in the contamination model plays a similar role to the ratio $\frac{n_2}{n_1+n_2}$ in (\ref{eq:dohonodep}). Motivated by this fact, we introduce a population variation of (\ref{eq:dohonodep}). Given an estimator $\hat{\theta}$, its $\delta$-breakdown point with respect to some parameter space $\Theta$ is defined as
\begin{equation}
\epsilon(\hat{\theta},\Theta,\delta)=\min\left\{\epsilon: \sup_{\theta\in\Theta}\sup_Q\mathbb{P}_{(\epsilon,\theta,Q)}\left\{L\left(\hat{\theta}(\{Y_i\}_{i=1}^{n_1}),\hat{\theta}(\{X_i\}_{i=1}^{n})\right)>\delta\right\}>c\right\},\label{eq:CGRdep}
\end{equation}
where $L(\cdot,\cdot)$ is some loss function, and $c\in (0,1)$ is some small constant. We may view (\ref{eq:CGRdep}) as the population variation of (\ref{eq:dohonodep}) because $\sup_{Q}$ corresponds to $\sup_{\{Z_i\}_{i=1}^{n_2}}$, $\sup_{\theta\in\Theta}$ corresponds to $\sup_{\{Y_i\}_{i=1}^{n_1}\in\mathcal{Y}}$ and $\epsilon$ corresponds to $\frac{n_2}{n_1+n_2}$. We allow $\delta$ to be a sequence of $n$ instead of $\infty$ because $L(\cdot,\cdot)$ can be a bounded loss such as the one considered in the PCA problem in this paper. When $\delta=\infty$ for an unbounded loss and the bias term dominates the loss, the $\delta$-breakdown point becomes the lower bound of the contamination level $\epsilon$ for which the $\epsilon$-maxbias is infinite. See, e.g., \cite{zuo2005depth}. In general, the $\delta$-breakdown point means the minimal $\epsilon$ such that an estimator $\hat{\theta}$ is influenced at least by the level of $\delta$ under the $\epsilon$-contamination model.
In fact, $\epsilon(\hat{\theta},\Theta,\delta)$ is a quantity directly related to the lower bound of the convergence rate of $\hat{\theta}$ under the $\epsilon$-contamination model. This is rigorously stated in the following theorem.
\begin{thm}\label{thm:niubi}
Assume the loss function is symmetric and satisfies
\begin{eqnarray}
\label{eq:triangle}&& L(\theta_1,\theta_2)\leq A\left(L(\theta_1,\theta_3)+L(\theta_2,\theta_3)\right)\text{ }\forall\theta_1,\theta_2,\theta_3\in\Theta\text{ with some }A>0, \\
\label{eq:notcrazy}&& \sup_{\theta\in\Theta}P_{\theta}^{n}\left\{L(\hat{\theta},\theta)>\frac{1}{2}c_1A^{-1}\delta\right\}\geq\sup_{\theta\in\Theta}P_{\theta}^{n'}\left\{L(\hat{\theta},\theta)>\frac{1}{2}A^{-1}\delta\right\}\text{ }\forall n'\geq \frac{n}{3},
\end{eqnarray}
with some constant $c_1\in (0,1)$.
Then, for $\epsilon=\epsilon(\hat{\theta},\Theta,\delta)<\frac{1}{2}$, we have
$$\sup_{\theta\in\Theta}\sup_Q\mathbb{P}_{(\epsilon,\theta,Q)}\left\{L(\hat{\theta},\theta)>\frac{1}{2}c_1A^{-1}\delta\right\}>\frac{1}{3}c,$$
for some $c>0$ in (\ref{eq:CGRdep}) and sufficiently large $n$.
\end{thm}
Before discussing the implications of Theorem \ref{thm:niubi}, we remark on the assumption (\ref{eq:notcrazy}).
The notation $P_{\theta}^n$ means the estimator $\hat{\theta}(\cdot)$ takes a random argument $\hat{\theta}(\{Y_i\}_{i=1}^n)$ with distribution $Y_1,...,Y_n\sim P_{\theta}$. Thus, the assumption (\ref{eq:notcrazy}) simply means when the sample sizes $n,n'$ are at the same order, the lower bounds remain at the same order. In most cases including all the examples considered in this paper, (\ref{eq:notcrazy}) automatically holds.

A general lower bound based on the notion of $\delta$-breakdown point is provided by Theorem \ref{thm:niubi}. Given an estimator $\hat{\theta}$ and an $\epsilon$-contamination model, the solution $\delta$ to the equation
\begin{equation}
\epsilon(\hat{\theta},\Theta,\delta)=\epsilon\label{eq:breakequation}
\end{equation}
lower bounds its rate of convergence. When $\hat{\theta}$ is a minimax optimal estimator with rate $\mathcal{M}(\epsilon)$, we obtain $\mathcal{M}(\epsilon)\gtrsim\delta$. In other words, the convergence rate $\delta$ under the $\epsilon$-contamination model automatically implies a $\delta$-breakdown point with the same $\epsilon$.

\subsection{A Unified Framework of Robustness and Rate of Convergence}\label{sec:unified-f}

Huber's $\epsilon$-contamination model is very classical in robust statistics, and allows for a deeper investigation than the breakdown point alone. For example, it has been well studied how much bias an estimator wound suffer under the contamination model via the
concept of maxbias in various models, including \cite{zuo2005depth}. In this paper, we demonstrate that Huber's $\epsilon$-contamination model allows a simultaneous joint study of robustness and rate of convergence of an estimator in the minimax sense. There are some important works that studied such properties of robust estimators under $\epsilon$-contamination model. We mention \cite{huber1965robust,huber1973minimax,buja1986huber} among others. However, such results in high-dimensional settings are rarely explored. This is our major reason to develop the minimax rate optimality theory of robust covariance matrix estimation under this framework. We illustrate the importance of this view by re-visiting the componentwise median studied in Section \ref{sec:location}. Without contamination, the componentwise median is a location estimator with minimax rate under Gaussian distribution. It is also robust because of its high breakdown point \citep{donoho1992breakdown}. However, Proposition \ref{pro:location cm} shows that its performance under the presence of contamination is not optimal. In contrast, Tukey's multivariate median shows its advantage over the componentwise median by obtaining optimality under the $\epsilon$-contamination model. This example suggests that the rate optimality and the robustness property of an estimator should be studied together rather than separately.

Recently, Donoho and Montanari \citep{donoho2013high} have studied Huber's M-estimator under the $\epsilon$-contamination model in a regression setting where $p/n$ converges to a constant. They find a critical $\epsilon^*$ that determines the variance breakdown point. The setting of $\epsilon$-contamination model plays a critical role in their work to illustrate both efficiency and robustness of Huber's M-estimator in a unified way.








\section{Proofs of Main Results}\label{sec:proof}

This section provides proofs for the results in Section \ref{sec:matrix}. 

\subsection{Auxiliary Lemmas}

For i.i.d. data $\{X_i\}_{i=1}^n$ from a contaminated distribution $(1-\epsilon)P+\epsilon Q$, it can be written as $\{Y_i\}_{i=1}^{n_1}\cup\{Z_i\}_{i=1}^{n_2}$. Marginally, we have $n_2\sim\text{Binomial}(n,\epsilon)$ and $n_1=n-n_2$. Conditioning on $n_1$ and $n_2$, $\{Y_i\}_{i=1}^{n_1}$ are i.i.d. from $P$ and $\{Z\}_{i=1}^{n_2}$ are i.i.d. from $Q$. The following lemmas control the ratio $n_2/n_1$ and characterize an important property respectively. Their proofs are given in the supplementary material \cite{supplement}.
\begin{lemma} \label{lem:ratio}
Assume $\epsilon<1/5$.
For any $\delta>0$ satisfying $\sqrt{\frac{1}{2n}\log(1/\delta)}<1/5$, we have
\begin{equation}
\frac{n_2}{n_1}\leq \frac{\epsilon}{1-\epsilon}+\frac{25}{12}\sqrt{\frac{1}{2n}\log(1/\delta)},\label{eq:ratio}
\end{equation}
with probability at least $1-\delta$. Moreover, assume $\epsilon^2>1/n$, and then we have
\begin{equation}
\frac{n_2}{n_1}>c'\epsilon, \label{eq:ratio2}
\end{equation}
with probability at least $1/2$ for some constant $c'>0$.
\end{lemma}

\begin{lemma}\label{lem:qiantao}
Consider any parametric family $\{P_{\theta}:\theta\in\Theta\}$. Then
$$\{(1-\epsilon_1)P_{\theta}+\epsilon_1Q:\theta\in\Theta, Q\}\subset \{(1-\epsilon_2)P_{\theta}+\epsilon_2Q:\theta\in\Theta, Q\},$$
holds for any $0\leq\epsilon_1<\epsilon_2\leq 1$.
\end{lemma}

Recall that for any $S\subset [p]$, $\mathcal{V}_{S}=\{u=(u_i)\in S^{p-1}: u_i=0\text{ if }i\notin S\}$. In particular, if $S=\{l_1,\ldots,l_2\}$, then $\mathcal{V}_{S}=\mathcal{V}_{[l_1,l_2]}$ defined in Section \ref{sec:banded}. Moreover, $\mathcal{V}_{S}=S^{p-1}$ if $S=\{1,\ldots,p\}$. Define a subset $IH_{u,t}$ of $\mathbb{R}^{p}$ as $IH_{u,t }=\{y:|u^{T}y|\leq t \}$. Finally, we need the following concentration inequality for  suprema of the empirical process indexed by these subsets $IH_{u,t }$, where $u\in \mathcal{V}_{S}$ and $t\in\mathbb{R}$. Its proof is given in the supplementary material \cite{supplement} by using Dudley's entropy integral \cite{dudley1978central} and VC classes \cite{vapnik1971uniform}.
\begin{lemma}\label{lem:DKW}
For i.i.d. real-valued data $X_1,...,X_n$ from distribution $\mathbb{P}$, we have for any $S\subset [p]$, with probability at least $1-\delta$,
$$\sup_{u\in \mathcal{V}_{S},t \in \mathbb{R} }\left\vert \mathbb{P}(IH_{u,t })-\mathbb{P}%
_{n}(IH_{u,t })\right\vert \leq \sqrt{\frac{1440e\pi}{1-e^{-1}}}\sqrt{\frac{3+2|S|}{n}}+%
\sqrt{\frac{\log (1/\delta )}{2n}},$$where $\mathbb{P}_{n}$ denotes the empirical distribution of $\{X_{i}\}_{i=1}^{n}$.
\end{lemma}

\subsection{Proofs of upper bounds in Section \ref{sec:matrix}}

We first prove the following master theorem.

\begin{thm}\label{thm:general}
For some index subsets $S_1,\ldots,S_m\subset [p]$ with $\max_i|S_i|\leq s$, consider the estimator $\hat{\Sigma}$ defined in (\ref{eq:estimator of Cov}) with $\mathcal{U}=\cup_{i=1}^{m}\mathcal{V}_{S_i}$.
Assume $\epsilon<1/5$. Then, there exist absolute constants $C,C_1>0$, such that for any $\delta\in(0,1/2)$ satisfying $C_1\frac{1+s+\log(m/\delta)}{n}<1$, we have
$$\sup_{u\in\mathcal{U}}\left|u^T\hat{\Sigma}u-u^T\Sigma u\right|\leq C\left(\epsilon+\sqrt{\frac{1+s+\log(m/\delta)}{n}}\right),$$
$\mathbb{P}_{(\epsilon,\Sigma,Q)}$-probability at least $1-2\delta$ uniformly over all $Q$ and $\Sigma\in\mathcal{F}(M)$ with $\beta\Sigma\in\mathcal{F}$, where constant $\beta$ is defined in (\ref{eq:beta}).
\end{thm}
\begin{proof}
By Lemma \ref{lem:ratio}, we decompose the data $\{X_i\}_{i=1}^n=\{Y_i\}_{i=1}^{n_1}\cup\{Z_i\}_{i=1}^{n_2}$. The following analysis is conditioning on the set of $(n_1,n_2)$ that satisfies (\ref{eq:ratio}) with probability at least $1-\delta$. To facilitate the proof, define
\begin{equation*}
\mathcal{D}_u(\Gamma,P_{\Sigma})=\min\left\{P_{\Sigma}(|u^TY|^2\leq u^T\Gamma u), P_{\Sigma}(|u^TY|^2> u^T\Gamma u)\right\},
\end{equation*}
\begin{equation*}
\mathcal{D}_u(\Gamma,\{Y_i\}_{i=1}^{n_1}) = \min\left\{\frac{1}{n_1}\sum_{i=1}^{n_1}\mathbb{I}\{|u^TY_i|^2\leq u^T\Gamma u\}, \frac{1}{n_1}\sum_{i=1}^{n_1}\mathbb{I}\{|u^TY_i|^2> u^T\Gamma u\}\right\},
\end{equation*}
for each $u\in S^{p-1}$. Then, we have $\mathcal{D}_{\mathcal{U}}(\Gamma,P_{\Sigma})=\inf_{u\in\cup_{i=1}^{m}\mathcal{V}_{S_i}}\mathcal{D}_u(\Gamma,P_{\Sigma})$ and $\mathcal{D}_{\mathcal{U}}(\Gamma,\{Y_i\}_{i=1}^{n_1})=\min_{u\in\cup_{i=1}^{m}\mathcal{V}_{S_i}}\mathcal{D}_u(\Gamma,\{Y_i\}_{i=1}^{n_1})$. Observe that
\begin{eqnarray*}
&& \sup_{\Gamma\in\mathcal{F}}\left|\mathcal{D}_{\mathcal{U}}(\Gamma,P_{\Sigma})-\mathcal{D}_{\mathcal{U}}(\Gamma,\{Y_i\}_{i=1}^{n_1})\right| \\
&\leq& 
\sup_{u\in\cup_{i=1}^{m}\mathcal{V}_{S},t\in\mathbb{R}}\left|\frac{1}{n_1}\sum_{i=1}^{n_1}\mathbb{I}\{|u^TY_i|^2\leq t\}-P_{\Sigma}(|u^TY|^2\leq t)\right| \\
&=&
\sup_{u\in\cup_{i=1}^{m} \mathcal{V}_{S},t\in\mathbb{R}}\left\vert P_{\Sigma}(IH_{u,t })-\mathbb{P}%
_{n_1}(IH_{u,t })\right\vert. 
\end{eqnarray*}
Applying Lemma \ref{lem:DKW} and union bound with $\max_i|S_i|\leq s$, we get
\begin{equation}
\sup_{\Gamma\in\mathcal{F}}\left|\mathcal{D}_{\mathcal{U}}(\Gamma,P_{\Sigma})-\mathcal{D}_{\mathcal{U}}(\Gamma,\{Y_i\}_{i=1}^{n_1})\right|\leq \sqrt{\frac{1440e\pi}{1-e^{-1}}}\sqrt{\frac{3+2s}{n_1}}+\sqrt{\frac{\log (m/\delta )}{2n_1}}, \label{eq:dkwU}
\end{equation}
with probability at least $1-\delta$. We lower bound $\mathcal{D}_{\mathcal{U}}(\hat{\Gamma},P_{\Sigma})$ by
\begin{eqnarray}
\label{eq:dU1} && \mathcal{D}_{\mathcal{U}}(\hat{\Gamma},\{Y_i\}_{i=1}^{n_1})-\sqrt{\frac{1440e\pi}{1-e^{-1}}}\sqrt{\frac{3+2s}{n_1}}-\sqrt{\frac{\log (m/\delta )}{2n_1}} \\
\label{eq:bf1} &\geq& \frac{n}{n_1}\mathcal{D}_{\mathcal{U}}(\hat{\Gamma},\{X_i\}_{i=1}^{n})-\frac{n_2}{n_1}-\sqrt{\frac{1440e\pi}{1-e^{-1}}}\sqrt{\frac{3+2s}{n_1}}-\sqrt{\frac{\log (m/\delta )}{2n_1}} \\
\label{eq:defGamma} &\geq&
\frac{n}{n_1}\mathcal{D}_{\mathcal{U}}(\beta\Sigma,\{X_i\}_{i=1}^{n})-\sqrt{\frac{1440e\pi}{1-e^{-1}}}\sqrt{\frac{3+2s}{n_1}}-\sqrt{\frac{\log (m/\delta )}{2n_1}}\\
\label{eq:bf2} &\geq& \mathcal{D}_{\mathcal{U}}(\beta\Sigma,\{Y_i\}_{i=1}^{n_1})-\frac{n_2}{n_1}-\sqrt{\frac{1440e\pi}{1-e^{-1}}}\sqrt{\frac{3+2s}{n_1}}-\sqrt{\frac{\log (m/\delta )}{2n_1}} \\
\label{eq:dU2} &\geq& \mathcal{D}_{\mathcal{U}}(\beta\Sigma,P_{\Sigma})-\frac{n_2}{n_1}-2\sqrt{\frac{1440e\pi}{1-e^{-1}}}\sqrt{\frac{3+2s}{n_1}}-\sqrt{\frac{2\log (m/\delta )}{n_1}}\\
\label{eq:prop} &=& \frac{1}{2}-\frac{n_2}{n_1}-2\sqrt{\frac{1440e\pi}{1-e^{-1}}}\sqrt{\frac{3+2s}{n_1}}-\sqrt{\frac{2\log (m/\delta )}{n_1}}.
\end{eqnarray}
The inequalities (\ref{eq:dU1}) and (\ref{eq:dU2}) are by (\ref{eq:dkwU}). The inequalities (\ref{eq:bf1}) and (\ref{eq:bf2}) are due to the property of depth function that
$$n_1\mathcal{D}_{\mathcal{U}}(\Gamma,\{Y_i\}_{i=1}^{n_1})\geq n\mathcal{D}_{\mathcal{U}}(\Gamma,\{X_i\}_{i=1}^{n})-n_2\geq n_1\mathcal{D}_{\mathcal{U}}(\Gamma,\{Y_i\}_{i=1}^{n_1})-n_2,$$
for any $\Gamma\in\mathcal{F}$. The inequality (\ref{eq:defGamma}) is by the definition of $\hat{\Gamma}$ and that $\beta\Sigma\in\mathcal{F}$. Finally, the equality (\ref{eq:prop}) is due to Proposition \ref{prop:truth}. Now let us use Lemma \ref{lem:ratio} so that the right hand side of (\ref{eq:prop}) can be lower bounded by
$$\frac{1}{2}-\frac{\epsilon}{1-\epsilon}-40\sqrt{\frac{6e\pi}{1-e^{-1}}}\sqrt{\frac{3+2s}{n}}-\frac{7}{2}\sqrt{\frac{\log (m/\delta )}{n}},$$
with probability at least $1-2\delta$. Using the property that $\mathcal{D}_u(\hat{\Gamma},P_{\Sigma}) \geq \mathcal{D}_{\mathcal{U}}(\hat{\Gamma},P_{\Sigma})$ for each $u\in\mathcal{U}$, we have shown that uniformly for all $u\in\mathcal{U}$,
\begin{equation}
\mathcal{D}_u(\hat{\Gamma},P_{\Sigma}) \geq \frac{1}{2}-\frac{\epsilon}{1-\epsilon}-40\sqrt{\frac{6e\pi}{1-e^{-1}}}\sqrt{\frac{3+2s}{n}}-\frac{7}{2}\sqrt{\frac{\log (m/\delta )}{n}}, \label{eq:verydeep}
\end{equation}
with probability at least $1-2\delta$. By Proposition \ref{prop:truth} and the fact that $\frac{1}{2}-\min(x,1-x)=|x-1/2|$ for all $x\in[0,1]$, we get
$$\frac{1}{2}-\mathcal{D}_u(\hat{\Gamma},P_{\Sigma})=2\left|\Phi(\sqrt{\beta})-\Phi\left(\sqrt{\frac{u^T\hat{\Gamma}u}{{u}^T\Sigma u}}\right)\right|.$$
Combining with (\ref{eq:verydeep}), we have
$$\sup_{u\in\mathcal{U}}\left|\Phi(\sqrt{\beta})-\Phi\left(\sqrt{\frac{u^T\hat{\Gamma}u}{{u}^T\Sigma u}}\right)\right|\leq \frac{\epsilon/2}{1-\epsilon}+\sqrt{\frac{2400e\pi}{1-e^{-1}}}\sqrt{\frac{3+2s}{n}}+\frac{7}{4}\sqrt{\frac{\log (m/\delta )}{n}},$$
with probability at least $1-2\delta$. Under the assumption that $\epsilon<1/5$ and $C_1\frac{1+s+\log(m/\delta)}{n}<1$ with some absolute constant $C_1>0$, we have
$$\sup_{u\in\mathcal{U}}\left|\sqrt{\beta}-\sqrt{\frac{u^T\hat{\Gamma}u}{{u}^T\Sigma u}}\right|\leq C_2\left(\epsilon+\sqrt{\frac{1+s+\log(m/\delta)}{n}}\right),$$
for some absolute constant $C_2>0$ with probability at least $1-2\delta$. Finally, due to assumption $\Sigma\in\mathcal{F}(M)$, we obtain that $\sup_{u\in\mathcal{U}}u^T\Sigma u\leq \opnorm{\Sigma}\leq M$, which implies
$$\sup_{u\in\mathcal{U}}\left|u^T\hat{\Gamma}u/\beta-u^T\Sigma u\right|\leq C\left(\epsilon+\sqrt{\frac{1+s+\log(m/\delta)}{n}}\right),$$
with probability at least $1-2\delta$. Thus, the proof is complete.
\end{proof}

\begin{proof}[Proof of Theorem \ref{thm:matrix upper}]	
Since $\mathcal{U}=S^{p-1}$ and $\mathcal{F}$ is taken as the set of all positive semi-definite matrices, the conclusion follows the result of Theorem \ref{thm:general} with $m=1$ and $S_1=[p]$ by noting $\opnorm{\hat{\Sigma}-\Sigma}=\sup_{u\in\mathcal{V}_{S_1}}\left|u^T\hat{\Sigma}u-u^T\Sigma u\right|$.
\end{proof}

\begin{proof}[Proof of Theorem \ref{thm:band}]
Consider the weights $$w_{ij}=k^{-1}\left((2k-|i-j|)_+-(k-|i-j|)_+\right).$$ Since $\hat{\Sigma}-\Sigma=(\hat{\sigma}_{ij}-\sigma_{ij})\in\mathcal{F}_k$, we have $(\hat{\sigma}_{ij}-\sigma_{ij})=((\hat{\sigma}_{ij}-\sigma_{ij})w_{ij})$. This means $\hat{\Sigma}-\Sigma$ can also be viewed as a tapered matrix. Then, Lemma 2 of \cite{cai2010optimal} implies that $\opnorm{\hat{\Sigma}-\Sigma}\leq 3\max_{u\in\mathcal{U}_k}|u^T(\hat{\Sigma}-\Sigma)u|$. Using  the fact that  $\mathcal{U}_k=\cup_{l=1}^{p+1-2k}\mathcal{V}_{[l,l+2k-1]}$ for $2k<p$, the conclusion follows by Theorem \ref{thm:general} with $m=p+1-2k$, $S_i=[i,i+2k-1]$ for $i=1,\ldots,m$ and $s=2k$. The result holds trivially according to Theorem \ref{thm:general} when $2k>p$ since $\mathcal{U}_k=S^{p-1}$.
\end{proof}

\begin{proof}[Proof of Theorem \ref{thm:bandable}]
The main argument of the proof is due to a bias-variance tradeoff. For $\Sigma=(\sigma_{ij})\in\mathcal{F}_{\alpha}(M,M_0,M_{\min})$, define $\Sigma_k=(\sigma_{ij}\mathbb{I}\{|i-j|\leq k\})$. Then
\begin{eqnarray*}
&& |\mathcal{D}_{\mathcal{U}_k}(\beta\Sigma,P_{\Sigma})-\mathcal{D}_{\mathcal{U}_k}(\beta\Sigma_k,P_{\Sigma})| \leq \max_{u\in\mathcal{U}_k}|\mathcal{D}_u(\beta\Sigma,P_{\Sigma})-\mathcal{D}_u(\beta\Sigma_k,P_{\Sigma})|\\
&\leq& 2\max_{u\in\mathcal{U}_k}\left|\Phi(\sqrt{\beta})-\Phi\left(\sqrt{\frac{\beta u^T\Sigma_ku}{u^T\Sigma u}}\right)\right| \leq \sqrt{\frac{2\beta}{\pi}}\max_{u\in\mathcal{U}_k}\left|1-\sqrt{\frac{u^T\Sigma_ku}{u^T\Sigma u}}\right|\\
&\leq& \sqrt{\frac{2\beta}{\pi}}\max_{u\in\mathcal{U}_k}\left|\frac{u^T(\Sigma_k-\Sigma)u}{u^T\Sigma u}\right|\leq \sqrt{\frac{2\beta}{\pi}}M^{-1}_{\min}\opnorm{\Sigma_k-\Sigma}.
\end{eqnarray*}
Recall $\mathcal{U}_k=\cup_{l=1}^{p+1-2k}\mathcal{V}_{[l,l+2k-1]}$ when $2k\leq p$ and $\mathcal{U}_k=S^{p-1}$ when $2k>p$, where $k=\ceil{n^{\frac{1}{2\alpha+1}}}\wedge p$. Using the bias bound above and the fact that $\beta\Sigma_k\in \mathcal{F}_k$, and modifying the arguments (\ref{eq:dU1})-(\ref{eq:verydeep}) in the proof of Theorem \ref{thm:general} with $\mathcal{U}=\mathcal{U}_k$, $m=\max(p+1-2k,1)$ and $s=(2k)\wedge p$, we obtain
\begin{eqnarray*}
	\mathcal{D}_u(\hat{\Gamma},P_{\Sigma}) &\geq& \frac{1}{2}-\frac{\epsilon}{1-\epsilon}-40\sqrt{\frac{6e\pi}{1-e^{-1}}}\sqrt{\frac{3+4k}{n}} \\
	&& -\frac{7}{2}\sqrt{\frac{\log (m/\delta )}{n}}-\sqrt{\frac{2\beta}{\pi}}M^{-1}_{\min}\opnorm{\Sigma_k-\Sigma},
\end{eqnarray*}
uniformly for all $u\in\mathcal{U}_k$ with probability at least $1-2\delta$. Repeating the corresponding subsequent argument in the proof of Theorem \ref{thm:general}, we have
$$\sup_{u\in\mathcal{U}_k}\left|u^T\hat{\Gamma}u/\beta-u^T\Sigma u\right|\leq C_1M\left(\epsilon+\sqrt{\frac{k+\log(m/\delta)}{n}}+M^{-1}_{\min}\opnorm{\Sigma_k-\Sigma}\right).$$
A triangle inequality implies
$$\sup_{u\in\mathcal{U}_k}\left|u^T\hat{\Sigma}u-u^T\Sigma_k u\right|\leq C_2\left(\epsilon+\sqrt{\frac{k+\log(m/\delta)}{n}}+\opnorm{\Sigma_k-\Sigma}\right).$$
By the argument in the proof of Theorem \ref{thm:band} and triangle inequality, we get
\begin{eqnarray*}
\opnorm{\hat{\Sigma}-\Sigma_k}&\leq& C_3\left(\epsilon+\sqrt{\frac{k+\log(m/\delta)}{n}}+\opnorm{\Sigma_k-\Sigma}\right),\\
\opnorm{\hat{\Sigma}-\Sigma}&\leq& C\left(\epsilon+\sqrt{\frac{k+\log(m/\delta)}{n}}+\opnorm{\Sigma_k-\Sigma}\right).
\end{eqnarray*}
A bias argument in \cite{cai2010optimal} implies that $\opnorm{\Sigma_k-\Sigma}\leq C_4k^{-\alpha}$. The proof is complete by observing that $k=\ceil{n^{\frac{1}{2\alpha+1}}}\wedge p$ and $m=\max(p+1-2k,1)$.
\end{proof}

\begin{proof}[Proof of Theorem \ref{thm:sparse}]
Note that $\hat{\Sigma}-\Sigma\in\mathcal{F}_{2s}$, and thus $\opnorm{\hat{\Sigma}-\Sigma}=\max_{|S|=2s}\opnorm{(\hat{\Sigma}-\Sigma)_{SS}}=\sup_{u\in\mathcal{U}_s}|u^T(\hat{\Sigma}-\Sigma)u|$. We denote all subsets of $[p]$ with cardinality $2s$ as $S_1,\ldots,S_m$, where $m={p\choose 2s}\leq \exp\left(2s\log\frac{ep}{s}\right)$. The proof is complete by applying Theorem \ref{thm:general} with these subsets $S_1,\ldots,S_m$, noting that $\mathcal{U}_s=\cup_{i=1}^{m}S_i$.
\end{proof}

\begin{proof}[Proof of Theorem \ref{thm:pca}]
Since $\mathcal{F}_{s,\lambda}(M,r)\subset\mathcal{F}_s(M+1)$, the result of Theorem \ref{thm:sparse} applies and we get
$$\opnorm{\hat{\Gamma}/\beta-\Sigma}^2\leq C\left(
\frac{s\log\frac{ep}{s}}{n}\vee\epsilon^2+\frac{\log(1/\delta)}{n}\right),$$
with probability at least $1-2\delta$. Weyl's inequality implies $|s_{r+1}(\hat{\Gamma}/\beta)-1|\leq \opnorm{\hat{\Gamma}/\beta-\Sigma}$. Under the assumption that the rate is bounded by a small constant, we have $s_r(\Sigma)-s_{r+1}(\hat{\Gamma}/\beta)>c\lambda$ for some constant $c>0$.
By Davis-Kahan theorem \citep{davis70}, we have $\fnorm{\hat{V}\hat{V}^T-VV^T}\leq C'\opnorm{\hat{\Gamma}/\beta-\Sigma}/\lambda$, and the proof is complete.
\end{proof}

\section*{Acknowledgement}

Fang Han suggested the U-statistics idea behind Section \ref{sec:disc_center}. Johannes Schmidt-Hieber suggested a weaker assumption for the main theorems. The authors are grateful to the extensive reviews from an associate editor and two referees. Their comments greatly improved the paper.
The authors also thank Andrew Barron, John Hartigan, David Pollard and Harrison Zhou for valuable comments in a YPNG seminar at Yale.

\bibliographystyle{plainnat}
\bibliography{Robust}

\newcommand{\SortNoop}[1]{}
\begin{thebibliography}{84}
\providecommand{\natexlab}[1]{#1}
\providecommand{\url}[1]{\texttt{#1}}
\expandafter\ifx\csname urlstyle\endcsname\relax
  \providecommand{\doi}[1]{doi: #1}\else
  \providecommand{\doi}{doi: \begingroup \urlstyle{rm}\Url}\fi

\bibitem[Bartlett and Mendelson(2003)]{bartlett2003rademacher}
Peter~L Bartlett and Shahar Mendelson.
\newblock Rademacher and {G}aussian complexities: Risk bounds and structural
  results.
\newblock \emph{The Journal of Machine Learning Research}, 3\penalty0
  (3):\penalty0 463--482, 2003.

\bibitem[Bickel and Levina(2008{\natexlab{a}})]{bickel2008covariance}
Peter~J Bickel and Elizaveta Levina.
\newblock Covariance regularization by thresholding.
\newblock \emph{The Annals of Statistics}, 36\penalty0 (6):\penalty0
  2577--2604, 2008{\natexlab{a}}.

\bibitem[Bickel and Levina(2008{\natexlab{b}})]{bickel2008regularized}
Peter~J Bickel and Elizaveta Levina.
\newblock Regularized estimation of large covariance matrices.
\newblock \emph{The Annals of Statistics}, 36\penalty0 (1):\penalty0 199--227,
  2008{\natexlab{b}}.

\bibitem[Birnbaum et~al.(2013)Birnbaum, Johnstone, Nadler, and
  Paul]{birnbaum13}
Aharon Birnbaum, Iain~M Johnstone, Boaz Nadler, and Debashis Paul.
\newblock Minimax bounds for sparse {PCA} with noisy high-dimensional data.
\newblock \emph{The Annals of Statistics}, 41\penalty0 (3):\penalty0
  1055--1084, 2013.

\bibitem[Buja(1986)]{buja1986huber}
Andreas Buja.
\newblock On the {H}uber-{S}trassen theorem.
\newblock \emph{Probability Theory and Related Fields}, 73\penalty0
  (1):\penalty0 149--152, 1986.

\bibitem[Butler et~al.(1993)Butler, Davies, and Jhun]{butler1993asymptotics}
RW~Butler, PL~Davies, and M~Jhun.
\newblock Asymptotics for the minimum covariance determinant estimator.
\newblock \emph{The Annals of Statistics}, 21\penalty0 (3):\penalty0
  1385--1400, 1993.

\bibitem[Cai and Zhou(2012)]{cai2012optimal}
T~Tony Cai and Harrison~H Zhou.
\newblock Optimal rates of convergence for sparse covariance matrix estimation.
\newblock \emph{The Annals of Statistics}, 40\penalty0 (5):\penalty0
  2389--2420, 2012.

\bibitem[Cai et~al.(2010)Cai, Zhang, and Zhou]{cai2010optimal}
T~Tony Cai, Cun-Hui Zhang, and Harrison~H Zhou.
\newblock Optimal rates of convergence for covariance matrix estimation.
\newblock \emph{The Annals of Statistics}, 38\penalty0 (4):\penalty0
  2118--2144, 2010.

\bibitem[Cai et~al.(2013{\natexlab{a}})Cai, Ma, and Wu]{cai2013sparse}
T~Tony Cai, Zongming Ma, and Yihong Wu.
\newblock Sparse {PCA}: Optimal rates and adaptive estimation.
\newblock \emph{The Annals of Statistics}, 41\penalty0 (6):\penalty0
  3074--3110, 2013{\natexlab{a}}.

\bibitem[Cai et~al.(2013{\natexlab{b}})Cai, Ren, and Zhou]{ren2013optimal}
T~Tony Cai, Zhao Ren, and Harrison~H Zhou.
\newblock Optimal rates of convergence for estimating {T}oeplitz covariance
  matrices.
\newblock \emph{Probability Theory and Related Fields}, 156\penalty0
  (1-2):\penalty0 101--143, 2013{\natexlab{b}}.

\bibitem[Cai et~al.(2015)Cai, Ma, and Wu]{cai13}
T~Tony Cai, Zongming Ma, and Yihong Wu.
\newblock Optimal estimation and rank detection for sparse spiked covariance
  matrices.
\newblock \emph{Probability Theory and Related Fields}, 161\penalty0
  (3-4):\penalty0 781--815, 2015.

\bibitem[Cai et~al.(2016)Cai, Ren, and Zhou]{cai2014estimating}
T~Tony Cai, Zhao Ren, and Harrison~H Zhou.
\newblock Estimating structured high-dimensional covariance and precision
  matrices: {O}ptimal rates and adaptive estimation.
\newblock \emph{Electronic Journal of Statistics}, 10\penalty0 (1):\penalty0
  1--59, 2016.

\bibitem[Chan(2004)]{chan2004optimal}
Timothy~M Chan.
\newblock An optimal randomized algorithm for maximum {T}ukey depth.
\newblock In \emph{Proceedings of the fifteenth annual ACM-SIAM symposium on
  Discrete algorithms}, pages 430--436. Society for Industrial and Applied
  Mathematics, 2004.

\bibitem[Chen et~al.(2017)Chen, Gao, and Ren]{supplement}
Mengjie Chen, Chao Gao, and Zhao Ren.
\newblock Supplement to ``{R}obust covariance and scatter matrix estimation
  under {H}uber's contamination mode".
\newblock 2017.

\bibitem[Croux and Haesbroeck(1999)]{croux1999influence}
Christophe Croux and Gentiane Haesbroeck.
\newblock Influence function and efficiency of the minimum covariance
  determinant scatter matrix estimator.
\newblock \emph{Journal of Multivariate Analysis}, 71\penalty0 (2):\penalty0
  161--190, 1999.

\bibitem[Davidson and Szarek(2001)]{davidson2001local}
Kenneth~R Davidson and Stanislaw~J Szarek.
\newblock Local operator theory, random matrices and {B}anach spaces.
\newblock \emph{Handbook of the geometry of Banach spaces}, 1\penalty0
  (317-366):\penalty0 131, 2001.

\bibitem[Davies(1992)]{davies1992asymptotics}
Laurie Davies.
\newblock The asymptotics of {R}ousseeuw's minimum volume ellipsoid estimator.
\newblock \emph{The Annals of Statistics}, 20\penalty0 (4):\penalty0
  1828--1843, 1992.

\bibitem[Davis and Kahan(1970)]{davis70}
Chandler Davis and William~Morton Kahan.
\newblock The rotation of eigenvectors by a perturbation. {III}.
\newblock \emph{SIAM Journal on Numerical Analysis}, 7\penalty0 (1):\penalty0
  1--46, 1970.

\bibitem[Devroye and Lugosi(2012)]{devroye2012combinatorial}
Luc Devroye and G{\'a}bor Lugosi.
\newblock \emph{Combinatorial Methods in Density Estimation}.
\newblock Springer Science \& Business Media, 2012.

\bibitem[Donoho(1982)]{donoho1982breakdown}
David~L Donoho.
\newblock Breakdown properties of multivariate location estimators.
\newblock Technical report, Harvard University, Boston. URL http://www-stat.
  stanford. edu/\~{} donoho/Reports/Oldies/BPMLE. pdf, 1982.

\bibitem[Donoho(1994)]{donoho1994statistical}
David~L Donoho.
\newblock Statistical estimation and optimal recovery.
\newblock \emph{The Annals of Statistics}, 22\penalty0 (1):\penalty0 238--270,
  1994.

\bibitem[Donoho and Gasko(1992)]{donoho1992breakdown}
David~L Donoho and Miriam Gasko.
\newblock Breakdown properties of location estimates based on halfspace depth
  and projected outlyingness.
\newblock \emph{The Annals of Statistics}, 20\penalty0 (4):\penalty0
  1803--1827, 1992.

\bibitem[Donoho and Huber(1983)]{donoho1983notion}
David~L Donoho and Peter~J Huber.
\newblock The notion of breakdown point.
\newblock \emph{A Festschrift for Erich L. Lehmann}, pages 157--184, 1983.

\bibitem[Donoho and Liu(1991)]{donoho1991geometrizing}
David~L Donoho and Richard~C Liu.
\newblock Geometrizing rates of convergence, {III}.
\newblock \emph{The Annals of Statistics}, 19\penalty0 (2):\penalty0 668--701,
  1991.

\bibitem[Donoho and Montanari(2015)]{donoho2013high}
David~L Donoho and Andrea Montanari.
\newblock Variance breakdown of {H}uber ({M})-estimators: $n/p\rightarrow m\in
  (1,\infty) $.
\newblock \emph{arXiv preprint arXiv:1503.02106}, 2015.

\bibitem[Dudley(1978)]{dudley1978central}
Richard~M Dudley.
\newblock Central limit theorems for empirical measures.
\newblock \emph{The Annals of Probability}, pages 899--929, 1978.

\bibitem[D{\"u}mbgen(1998)]{dumbgen1998tyler}
Lutz D{\"u}mbgen.
\newblock On {T}yler's {M}-functional of scatter in high dimension.
\newblock \emph{Annals of the Institute of Statistical Mathematics},
  50\penalty0 (3):\penalty0 471--491, 1998.

\bibitem[D{\"u}mbgen and Tyler(2005)]{dumbgen2005breakdown}
Lutz D{\"u}mbgen and David~E Tyler.
\newblock On the breakdown properties of some multivariate {M}-functionals.
\newblock \emph{Scandinavian Journal of Statistics}, 32\penalty0 (2):\penalty0
  247--264, 2005.

\bibitem[Fan et~al.(2014)Fan, Han, and Liu]{fan2014page}
Jianqing Fan, Fang Han, and Han Liu.
\newblock {PAGE}: Robust pattern guided estimation of large covariance matrix.
\newblock Technical report, Princeton University, 2014.

\bibitem[Fang et~al.(1990)Fang, Kotz, and Ng]{fang1990symmetric}
Kai-Tai Fang, Samuel Kotz, and Kai~Wang Ng.
\newblock \emph{Symmetric Multivariate and Related Distributions}.
\newblock Chapman and Hall, 1990.

\bibitem[Friston et~al.(1994)Friston, Jezzard, and Turner]{friston1994analysis}
Karl~J Friston, Peter Jezzard, and Robert Turner.
\newblock Analysis of functional {MRI} time-series.
\newblock \emph{Human Brain Mapping}, 1\penalty0 (2):\penalty0 153--171, 1994.

\bibitem[Hampel(1971)]{hampel1971general}
Frank~R Hampel.
\newblock A general qualitative definition of robustness.
\newblock \emph{The Annals of Mathematical Statistics}, 42\penalty0
  (6):\penalty0 1887--1896, 1971.

\bibitem[Han and Liu(2013)]{han2013optimal}
Fang Han and Han Liu.
\newblock Optimal rates of convergence for latent generalized correlation
  matrix estimation in transelliptical distribution.
\newblock \emph{arXiv preprint arXiv:1305.6916}, 2013.

\bibitem[Han and Liu(2014)]{han2014scale}
Fang Han and Han Liu.
\newblock Scale-invariant sparse {PCA} on high-dimensional meta-elliptical
  data.
\newblock \emph{Journal of the American Statistical Association}, 109\penalty0
  (505):\penalty0 275--287, 2014.

\bibitem[Han and Liu(2016)]{han2013eca}
Fang Han and Han Liu.
\newblock {ECA}: High dimensional elliptical component analysis in
  non-{G}aussian distributions.
\newblock \emph{Journal of the American Statistical Association}, \penalty0
  (just-accepted), 2016.

\bibitem[Han et~al.(2014)Han, Lu, and Liu]{han2014robust}
Fang Han, Junwei Lu, and Han Liu.
\newblock Robust scatter matrix estimation for high dimensional distributions
  with heavy tails.
\newblock Technical report, Princeton Univ., 2014.

\bibitem[Haussler(1995)]{haussler1995sphere}
David Haussler.
\newblock Sphere packing numbers for subsets of the {B}oolean n-cube with
  bounded {V}apnik-{C}hervonenkis dimension.
\newblock \emph{Journal of Combinatorial Theory, Series A}, 69\penalty0
  (2):\penalty0 217--232, 1995.

\bibitem[Huber(1964)]{huber1964robust}
Peter~J Huber.
\newblock Robust estimation of a location parameter.
\newblock \emph{The Annals of Mathematical Statistics}, 35\penalty0
  (1):\penalty0 73--101, 1964.

\bibitem[Huber(1965)]{huber1965robust}
Peter~J Huber.
\newblock A robust version of the probability ratio test.
\newblock \emph{The Annals of Mathematical Statistics}, 36\penalty0
  (6):\penalty0 1753--1758, 1965.

\bibitem[Huber and Strassen(1973)]{huber1973minimax}
Peter~J Huber and Volker Strassen.
\newblock Minimax tests and the {N}eyman-{P}earson lemma for capacities.
\newblock \emph{The Annals of Statistics}, pages 251--263, 1973.

\bibitem[Johnstone and Lu(2009)]{johnstone09}
Iain~M Johnstone and Arthur~Yu Lu.
\newblock On consistency and sparsity for principal components analysis in high
  dimensions.
\newblock \emph{Journal of the American Statistical Association}, 104\penalty0
  (486):\penalty0 682--693, 2009.

\bibitem[Kendall(1938)]{kendall1938new}
Maurice~G Kendall.
\newblock A new measure of rank correlation.
\newblock \emph{Biometrika}, 30\penalty0 (1-2):\penalty0 81--93, 1938.

\bibitem[Kruskal(1958)]{kruskal1958ordinal}
William~H Kruskal.
\newblock Ordinal measures of association.
\newblock \emph{Journal of the American Statistical Association}, 53\penalty0
  (284):\penalty0 814--861, 1958.

\bibitem[Lam and Fan(2009)]{lam2009sparsistency}
Clifford Lam and Jianqing Fan.
\newblock Sparsistency and rates of convergence in large covariance matrix
  estimation.
\newblock \emph{The Annals of Statistics}, 37\penalty0 (6B):\penalty0
  4254--4278, 2009.

\bibitem[Leroy and Rousseeuw(1987)]{leroy1987robust}
Annick~M Leroy and Peter~J Rousseeuw.
\newblock \emph{Robust Regression and Outlier Detection}.
\newblock John Wiley \& Sons, 1987.

\bibitem[Liu(1990)]{liu1990notion}
Regina~Y Liu.
\newblock On a notion of data depth based on random simplices.
\newblock \emph{The Annals of Statistics}, 18\penalty0 (1):\penalty0 405--414,
  1990.

\bibitem[Liu et~al.(1999)Liu, Parelius, and Singh]{liu1999multivariate}
Regina~Y Liu, Jesse~M Parelius, and Kesar Singh.
\newblock Multivariate analysis by data depth: descriptive statistics, graphics
  and inference,(with discussion and a rejoinder by {L}iu and {S}ingh).
\newblock \emph{The Annals of Statistics}, 27\penalty0 (3):\penalty0 783--858,
  1999.

\bibitem[Ma(2013)]{ma13}
Zongming Ma.
\newblock Sparse principal component analysis and iterative thresholding.
\newblock \emph{The Annals of Statistics}, 41\penalty0 (2):\penalty0 772--801,
  2013.

\bibitem[Ma and Wu(2013)]{ma2013volume}
Zongming Ma and Yihong Wu.
\newblock Volume ratio, sparsity, and minimaxity under unitarily invariant
  norms.
\newblock In \emph{Information Theory Proceedings (ISIT), 2013 IEEE
  International Symposium on}, pages 1027--1031. IEEE, 2013.

\bibitem[Maronna(1976)]{maronna1976robust}
Ricardo~Antonio Maronna.
\newblock Robust {M}-estimators of multivariate location and scatter.
\newblock \emph{The Annals of Statistics}, 4\penalty0 (1):\penalty0 51--67,
  1976.

\bibitem[Massart(1990)]{massart1990tight}
Pascal Massart.
\newblock The tight constant in the {D}voretzky-{K}iefer-{W}olfowitz
  inequality.
\newblock \emph{The Annals of Probability}, 18\penalty0 (3):\penalty0
  1269--1283, 1990.

\bibitem[McDiarmid(1989)]{mcdiarmid1989method}
Colin McDiarmid.
\newblock On the method of bounded differences.
\newblock \emph{Surveys in combinatorics}, 141\penalty0 (1):\penalty0 148--188,
  1989.

\bibitem[Mitra and Zhang(2014)]{mitra2014multivariate}
Ritwik Mitra and Cun-Hui Zhang.
\newblock Multivariate analysis of nonparametric estimates of large correlation
  matrices.
\newblock \emph{arXiv preprint arXiv:1403.6195}, 2014.

\bibitem[Mizera(2002)]{mizera2002depth}
Ivan Mizera.
\newblock On depth and deep points: a calculus.
\newblock \emph{The Annals of Statistics}, 30\penalty0 (6):\penalty0
  1681--1736, 2002.

\bibitem[Mizera and M\"{u}ller(2004)]{Mizera2004}
Ivan Mizera and Christine~H M\"{u}ller.
\newblock Location-scale depth.
\newblock \emph{Journal of the American Statistical Association}, 99\penalty0
  (468):\penalty0 949--966, 2004.

\bibitem[Nordhausen et~al.(2007)Nordhausen, Sirki{\"a}, Oja, and
  Tyler]{nordhausen2007icsnp}
K~Nordhausen, S~Sirki{\"a}, H~Oja, and DE~Tyler.
\newblock {ICSNP}: Tools for multivariate nonparametrics.
\newblock \emph{R Package Version}, pages 1--0, 2007.

\bibitem[Oja(1983)]{oja1983descriptive}
Hannu Oja.
\newblock Descriptive statistics for multivariate distributions.
\newblock \emph{Statistics \& Probability Letters}, 1\penalty0 (6):\penalty0
  327--332, 1983.

\bibitem[Pe{\~n}a and Prieto(2001)]{pena2001multivariate}
Daniel Pe{\~n}a and Francisco~J Prieto.
\newblock Multivariate outlier detection and robust covariance matrix
  estimation.
\newblock \emph{Technometrics}, 43\penalty0 (3):\penalty0 286--310, 2001.

\bibitem[Pollard(2012)]{pollard2012convergence}
David Pollard.
\newblock \emph{Convergence of stochastic processes}.
\newblock Springer Science \& Business Media, 2012.

\bibitem[Ripley(2011)]{ripley2011mass}
Brian Ripley.
\newblock {MASS}: support functions and datasets for {V}enables and
  {R}ipley’s {MASS}.
\newblock \emph{R Package Version}, pages 7--3, 2011.

\bibitem[Rousseeuw and Hubert(1999)]{rousseeuw1999regression}
Peter~J Rousseeuw and Mia Hubert.
\newblock Regression depth.
\newblock \emph{Journal of the American Statistical Association}, 94\penalty0
  (446):\penalty0 388--402, 1999.

\bibitem[Rousseeuw and Ruts(1998)]{rousseeuw1998constructing}
Peter~J Rousseeuw and Ida Ruts.
\newblock Constructing the bivariate {T}ukey median.
\newblock \emph{Statistica Sinica}, 8\penalty0 (3):\penalty0 827--839, 1998.

\bibitem[Serfling(2004)]{serfling2004}
Robert Serfling.
\newblock Some perspectives on location and scale depth functions.
\newblock \emph{Journal of the American Statistical Association}, 99\penalty0
  (468):\penalty0 970--973, 2004.

\bibitem[Steele(1978)]{steele1978existence}
J~Michael Steele.
\newblock Existence of submatrices with all possible columns.
\newblock \emph{Journal of Combinatorial Theory, Series A}, 24\penalty0
  (1):\penalty0 84--88, 1978.

\bibitem[Struyf and Rousseeuw(2000)]{struyf2000high}
Anja Struyf and Peter~J Rousseeuw.
\newblock High-dimensional computation of the deepest location.
\newblock \emph{Computational Statistics \& Data Analysis}, 34\penalty0
  (4):\penalty0 415--426, 2000.

\bibitem[Tukey(1974)]{tukey1974t6}
John~W Tukey.
\newblock T6: Order statistics, in mimeographed notes for {S}tatistics 411.
\newblock \emph{Department of Statistics, Princeton University}, 1974.

\bibitem[Tukey(1975)]{tukey1975mathematics}
John~W Tukey.
\newblock Mathematics and the picturing of data.
\newblock In \emph{Proceedings of the International Congress of
  Mathematicians}, volume~2, pages 523--531, 1975.

\bibitem[Tukey(1977)]{tukey1977exploratory}
John~W Tukey.
\newblock Exploratory data analysis.
\newblock \emph{Addison-Wesley Series in Behavioral Science: Quantitative
  Methods, Reading, Mass.}, 1, 1977.

\bibitem[Tyler(1987)]{tyler1987distribution}
David~E Tyler.
\newblock A distribution-free {M}-estimator of multivariate scatter.
\newblock \emph{The Annals of Statistics}, 15\penalty0 (1):\penalty0 234--251,
  1987.

\bibitem[van~der Vaart(2000)]{van2000asymptotic}
Aad~W van~der Vaart.
\newblock \emph{Asymptotic Statistics}.
\newblock Cambridge University Press, 2000.

\bibitem[Vapnik and Chervonenkis(1971)]{vapnik1971uniform}
Vladimir~N Vapnik and A~Ya Chervonenkis.
\newblock On the uniform convergence of relative frequencies of events to their
  probabilities.
\newblock \emph{Theory of Probability \& Its Applications}, 16\penalty0
  (2):\penalty0 264--280, 1971.

\bibitem[Vardi and Zhang(2000)]{vardi2000multivariate}
Yehuda Vardi and Cun-Hui Zhang.
\newblock The multivariate $\ell_1$-median and associated data depth.
\newblock \emph{Proceedings of the National Academy of Sciences}, 97\penalty0
  (4):\penalty0 1423--1426, 2000.

\bibitem[Vershynin(2010)]{vershynin2010introduction}
Roman Vershynin.
\newblock Introduction to the non-asymptotic analysis of random matrices.
\newblock \emph{arXiv preprint arXiv:1011.3027}, 2010.

\bibitem[Visser and Molenaar(1995)]{visser1995trend}
H~Visser and J~Molenaar.
\newblock Trend estimation and regression analysis in climatological time
  series: an application of structural time series models and the {K}alman
  filter.
\newblock \emph{Journal of Climate}, 8\penalty0 (5):\penalty0 969--979, 1995.

\bibitem[Vu and Lei(2013)]{vu12}
Vincent~Q Vu and Jing Lei.
\newblock Minimax sparse principal subspace estimation in high dimensions.
\newblock \emph{The Annals of Statistics}, 41\penalty0 (6):\penalty0
  2905--2947, 2013.

\bibitem[Wegkamp and Zhao(2016)]{wegkamp2013adaptive}
Marten Wegkamp and Yue Zhao.
\newblock Adaptive estimation of the copula correlation matrix for
  semiparametric elliptical copulas.
\newblock \emph{Bernoulli}, 22\penalty0 (2):\penalty0 1184--1226, 2016.

\bibitem[Xue and Zou(2013)]{xue2011optimal}
Lingzhou Xue and Hui Zou.
\newblock Optimal estimation of sparse correlation matrices of semiparametric
  {G}aussian copulas.
\newblock \emph{Statistics and Its Interface}, 7:\penalty0 201--209, 2013.

\bibitem[Xue and Zou(2014)]{xue2014rank}
Lingzhou Xue and Hui Zou.
\newblock Rank-based tapering estimation of bandable correlation matrices.
\newblock \emph{Statistica Sinica}, 24:\penalty0 83--100, 2014.

\bibitem[Yu(1997)]{yu1997assouad}
Bin Yu.
\newblock Assouad, {F}ano, and {L}e {C}am.
\newblock In \emph{Festschrift for Lucien Le Cam}, pages 423--435. Springer,
  1997.

\bibitem[Zhang(2002)]{zhang2002some}
Jian Zhang.
\newblock Some extensions of {T}ukey's depth function.
\newblock \emph{Journal of Multivariate Analysis}, 82\penalty0 (1):\penalty0
  134--165, 2002.

\bibitem[Zhang et~al.(2016)Zhang, Cheng, and Singer]{zhang2016marvcenko}
Teng Zhang, Xiuyuan Cheng, and Amit Singer.
\newblock Mar{\v{c}}enko--{P}astur law for {T}yler’s {M}-estimator.
\newblock \emph{Journal of Multivariate Analysis}, 149:\penalty0 114--123,
  2016.

\bibitem[Zuo and Cui(2005)]{zuo2005depth}
Yijun Zuo and Hengjian Cui.
\newblock Depth weighted scatter estimators.
\newblock \emph{The Annals of Statistics}, 33\penalty0 (1):\penalty0 381--413,
  2005.

\bibitem[Zuo and Serfling(2000{\natexlab{a}})]{zuo2000general}
Yijun Zuo and Robert Serfling.
\newblock General notions of statistical depth function.
\newblock \emph{The Annals of Statistics}, 28\penalty0 (2):\penalty0 461--482,
  2000{\natexlab{a}}.

\bibitem[Zuo and Serfling(2000{\natexlab{b}})]{zuo2000nonparametric}
Yijun Zuo and Robert Serfling.
\newblock Nonparametric notions of multivariate ``scatter measure" and ``more
  scattered" based on statistical depth functions.
\newblock \emph{Journal of Multivariate Analysis}, 75\penalty0 (1):\penalty0
  62--78, 2000{\natexlab{b}}.

\end{thebibliography}


\newpage

\begin{center}
	{\Large Supplement to ``Robust Covariance and Scatter Matrix Estimation under Huber's Contamination Model''}\\
	~\\
	Mengjie Chen, Chao Gao and Zhao Ren\\
	~\\
	University of Chicago, University of Chicago and University of Pittsburgh
\end{center}

\appendix

\setcounter{page}{1}

\renewcommand{\theequation}{A.\arabic{equation}}
\setcounter{equation}{0}

\section{Non-centered Observations}\label{sec:center}

In Section \ref{sec:disc_center}, we briefly discussed the extension of our method from centered distribution to non-centered data in order to make it more practical. We  provide more details in this section. Recall for the Gaussian case, we proposed the following definition of a U-version empirical matrix depth function. 
\begin{eqnarray*}
\bar{\D}_{\U}(\Gamma,\{X_i\}_{i=1}^n) &=& \min_{u\in\U}\min\left\{\frac{1}{{n\choose 2}}\sum_{i<j}\mathbb{I}\{|u^T(X_i-X_j)|^2\leq 2u^T\Gamma u\},\right. \\
&& \left.\frac{1}{{n\choose 2}}\sum_{i<j}\mathbb{I}\{|u^T(X_i-X_j)|^2\geq 2u^T\Gamma u\}\right\}.
\end{eqnarray*}
Then, a covariance matrix estimator $\hat{\Sigma}$ is defined through (\ref{eq:hatgamma}) and (\ref{eq:estimator of Cov}) with ${\D}_{\U}(\Gamma,\{X_i\}_{i=1}^n)$ replaced by $\bar{\D}_{\U}(\Gamma,\{X_i\}_{i=1}^n)$. A similar pairwise difference trick was used by \cite{dumbgen1998tyler} in a different setting. The performance of  the proposed estimator is guaranteed by the following theorem.

\begin{thm}\label{thm:U}
Theorems \ref{thm:matrix upper}, \ref{thm:band}, \ref{thm:bandable}, \ref{thm:sparse} and \ref{thm:pca} continue to hold for the proposed estimator.
\end{thm}

It is interesting that all the non-asymptotic bounds in Section \ref{sec:matrix} continue to hold for the estimator defined by the new depth function. The new estimator can be understood as applying the empirical matrix depth function to the data pairs $\left\{\frac{1}{\sqrt{2}}(X_i-X_j)\right\}_{i<j}$. The main argument in the proof is an extension of the concentration inequality for suprema of some empirical process established in Lemma \ref{lem:DKW} to its corresponding U-process (see Lemma \ref{lem:UDKW}). Moreover, since the expected number of contaminated observations is $n\epsilon$ in $\{X_i\}_{i=1}^n$, the number of contaminated pairs in $\left\{\frac{1}{\sqrt{2}}(X_i-X_j)\right\}_{i<j}$ is about $n^2(\epsilon-\epsilon^2/2)$. Therefore, the contamination proportion of the pairs $\left\{\frac{1}{\sqrt{2}}(X_i-X_j)\right\}_{i<j}$ is still of order $\epsilon$, which leads to the same optimal rate dependence on $\epsilon$ in Theorem \ref{thm:U}.

Next, we discuss non-centered elliptical distributions. For an elliptical distribution $EC(\theta,\Gamma,\eta)$ with an arbitrary location vector $\theta\in\mathbb{R}^p$, the goal is to estimate the scatter matrix $\Gamma$ at the presence of an unknown $\theta$. For independent $X_1,X_2\sim EC(\theta,\Gamma,\eta)$, since the characteristic function of $X_1$ has the form $e^{\sqrt{-1}t^T\theta}\phi(t^T\Gamma t)$ with some univariate function $\phi(\cdot)$, the characteristic function of $X_1-X_2$ is $[\phi(t^T\Gamma t)]^2$. This implies that $X_1-X_2$ is distributed by a centered elliptical distribution with a scatter matrix proportional to $\Gamma$. In other words, the distribution of $X_1-X_2$ is $EC(0,\wt{\Gamma},\wt{\eta})$ in a canonical form, where $(\wt{\Gamma},\wt{\eta})$ is determined by $(\Gamma,\eta)$ and $\wt{\Gamma}$ is $\Gamma$ multiplied by a constant factor. Since $\wt{\Gamma}$ carries the same information of the shape of the elliptical distribution as $\Gamma$, we propose to estimate $\wt{\Gamma}$ with non-centered observations. Define a U-version of the empirical matrix depth function as
\begin{eqnarray*}
\wt{\D}_{\U}(\Gamma,\{X_i\}_{i=1}^n) &=& \min_{u\in\U}\min\left\{\frac{1}{{n\choose 2}}\sum_{i<j}\mathbb{I}\{|u^T(X_i-X_j)|^2\leq u^T\Gamma u\},\right. \\
&& \left.\frac{1}{{n\choose 2}}\sum_{i<j}\mathbb{I}\{|u^T(X_i-X_j)|^2\geq u^T\Gamma u\}\right\}.
\end{eqnarray*}
The estimator for $\wt{\Gamma}$ is $\hat{\Gamma}$ defined through (\ref{eq:hatgamma}) with ${\D}_{\U}(\Gamma,\{X_i\}_{i=1}^n)$ replaced by $\wt{\D}_{\U}(\Gamma,\{X_i\}_{i=1}^n)$.

Before stating the result on convergence rates, we introduce an analogous assumption to (\ref{eq:assell}). Define
$$\wt{G}(t)=P_{\wt{\Gamma}}\left(\frac{|u^TY|^2}{u^T\wt{\Gamma}u}\leq t\right),$$
where $Y=X_1-X_2\sim P_{\wt{\Gamma}}=EC(0,\wt{\Gamma},\wt{\eta})$. Then, we assume that $\wt{G}(t)$ is continuous at $t=1$ and there exist some $\tau\in(0,1/2)$ and $\alpha,\kappa>0$ such that
\begin{equation}
\inf_{|t|\geq \alpha}|\wt{G}(1)-\wt{G}(1+t)|\geq \tau\quad\text{and}\quad\inf_{|t|<\alpha}\frac{|\wt{G}(1)-\wt{G}(1+t)|}{|t|}\geq\kappa^{-1/2}.\label{eq:assellwt}
\end{equation}
While (\ref{eq:assell}) is a direct assumption on the distribution of $X_1$, (\ref{eq:assellwt}) is an assumption stated on the distribution of $X_1-X_2$, which may not be as easy to check in practice. Using the simple fact that the characteristic function of $X_1-X_2$ is $[\phi(t^T\Gamma t)]^2$, we show that (\ref{eq:assellwt}) continues to hold for the examples that we have mentioned in Section \ref{sec:elliptical}.

\begin{proposition}\label{prop:exampleU}
For all the four elliptical distributions listed as examples in Section \ref{sec:elliptical}, the condition (\ref{eq:assellwt}) holds with absolute constants $\tau,\alpha,\kappa$ independent of $p$.
\end{proposition}

Now we are ready to extend the results in Section \ref{sec:elliptical} to non-centered observations.

\begin{thm}\label{thm:Ue}
Consider the estimator $\hat{\Gamma}$ defined in this section for non-centered observations. The non-asymptotic bounds in Theorems \ref{thm:e1}-\ref{thm:eca} continue to hold for $\opnorm{\hat{\Gamma}-\wt{\Gamma}}^2$ under the same assumptions except that (\ref{eq:assell}) is replaced by (\ref{eq:assellwt}).
\end{thm}

\renewcommand{\theequation}{B.\arabic{equation}}
\setcounter{equation}{0}

\section{Numerical Studies}\label{sec:num}

\subsection{An Algorithm}\label{sec:algorithm}


The computation of a depth-based estimator is hard. Some partial successes have been made in computing the deepest location estimator when the dimension is low. Exact computation of the bivariate Tukey's median was investigated by \cite{rousseeuw1998constructing}. It was reasoned by \cite{chan2004optimal} that the optimal time complexity is $O(n^{p-1})$, where $n$ is the sample size and $p$ is the dimension.
Later, \cite{struyf2000high} developed an approximate algorithm for computing Tukey's median in higher dimensions. The algorithm we propose for computing the deepest matrix estimator (\ref{eq:estimator of Gamma}) is an adaptation of the location algorithm in \cite{struyf2000high} to the matrix case.

The core idea of our algorithm is to start with some robust scatter matrix estimator that is easy to compute, and then iteratively evaluate its depth on {some direction subset $\bar{\mathcal{U}}$ of finite cardinality in $\mathcal{U}$ to approximate its matrix depth function and move towards the directions that can improve the depth.} For any unit vector $u$, recall the notation $\mathcal{D}_u(\Gamma,\{X_i\}_{i=1}^n)$ introduced in Section \ref{sec:matrix}.
Then, the depth function can be written as $\mathcal{D}_{\mathcal{U}}(\Gamma,\{X_i\}_{i=1}^n)=\min_{u\in\mathcal{U}}\mathcal{D}_u(\Gamma,\{X_i\}_{i=1}^n)$. {We denote the approximate matrix depth with respect to direction set $\bar{\mathcal{U}}\subset\mathcal{U}$ as $\mathcal{D}_{\bar{\mathcal{U}}}(\Gamma,\{X_i\}_{i=1}^n)=\min_{u\in\bar{\mathcal{U}}}\mathcal{D}_u(\Gamma,\{X_i\}_{i=1}^n)$.} The general structure of our proposed algorithm is stated in Algorithm \ref{algo:depth}.

\begin{algorithm}[!htb]
 \SetAlgoLined
   \caption{An outline to compute the deepest matrix estimator}\label{algo:depth}
 \KwIn{Data $\{X_i\}_{i=1}^n$, \\
~~~~~~~~~~~direction set $\bar{\mathcal{U}}\subset\mathcal{U}$,  \\
~~~~~~~~~~~number of iterations $T$, \\
~~~~~~~~~~~initial estimator $\Gamma_{(1)}$.}
 \KwOut{The deepest matrix $\hat{\Gamma}=\Gamma_{(T)}$.}
 \smallskip
   \For{$t=1$ \KwTo $T$} {
    \nl     Uniformly select $u_{(t)}$ from the set $\argmin_{u\in\bar{\mathcal{U}}}\mathcal{D}_u(\Gamma_{(t)},\{X_i\}_{i=1}^n)$\;
     \nl    Set $\delta_{(t)}$ to be the number with the smallest $|\delta_{(t)}|$ such that
          $$\mathcal{D}_{u_{(t)}}(\Gamma_{(t)}+\delta_{(t)}u_{(t)}u_{(t)}^T,\{X_i\}_{i=1}^n)> \mathcal{D}_{u_{(t)}}(\Gamma_{(t)},\{X_i\}_{i=1}^n);$$
      \nl    \If {$\mathcal{D}_{\bar{\mathcal{U}}}(\Gamma_{(t)}+\delta_{(t)}u_{(t)}u_{(t)}^T,\{X_i\}_{i=1}^n)> \mathcal{D}_{\bar{\mathcal{U}}}(\Gamma_{(t)},\{X_i\}_{i=1}^n)$} {Set $\Gamma_{(t+1)}=\Gamma_{(t)}+\delta_{(t)}u_{(t)}u_{(t)}^T$;}\Else{Set $\Gamma_{(t+1)}=\Gamma_{(t)}$.}
     } 
\end{algorithm}

Note that Algorithm \ref{algo:depth} is an outline. Real implementations require some engineering modifications in the spirit of \cite{struyf2000high}.
Specifically, when there are multiple directions achieving the smallest matrix depth, the direction to move towards is randomly selected from all tied directions. In the implementation, we record the five latest iterations whenever a decision has been made among ties and save the unselected directions as back-up. If the depth does not improve after 20 iterations, the algorithm will trace back and select directions from the back-up list. This will effectively prevent the search from being stuck in the current direction. {For the below simulation studies for unstructured matrices with dimension $p=10$, we specify the direction set $\bar{\mathcal{U}}$ to be $10^5$ uniform draws from the unit sphere $S_{p-1}$, which turns out to work well in practice.} The number of iterations is picked to be sufficiently large so that the depth hardly improves anymore after that.

{To further justify the validity of our approximate algorithm which takes direction set $\bar{\mathcal{U}}$ rather than $\mathcal{U}=S^{p-1}$ for unstructured covariance matrices, we show in the following proposition that with appropriate choice of $\bar{\mathcal{U}}\subset\mathcal{U}$ of finite cardinality, the deepest matrix with respect to $\bar{\mathcal{U}}$  also enjoys the the minimax rate optimality as the one does in Theorem \ref{thm:matrix upper}. Let $\bar{\mathcal{U}}$ be a $(1/2)$-net of the unit sphere $S^{p-1}$. This means for any $u\in S^{p-1}$, there exists a $u'\in \bar{\mathcal{U}}$ such that $\norm{u-u'}\leq 1/2$. According to \cite{vershynin2010introduction}, such $\bar{\mathcal{U}}$ can be picked with cardinality bounded by $5^p$. Define $\hat{\bar{\Gamma}}=\arg \max_{\Gamma \succeq 0}\D_{\bar{\mathcal{U}}}(\Gamma,\{X_{i}\}_{i=1}^{n})$. Then the approximate estimator of $\Sigma$ is defined by $\hat{\bar{\Sigma}}=\hat{\bar{\Gamma}}/\beta$ as in (\ref{eq:estimator of Cov}). 	

\begin{proposition}\label{thm:matrix upper_app}
	Assume that $\epsilon <1/5$. Then, there exist absolute constants $C,C_1>0$, such that for any $\delta\in(0,1/2)$ satisfying $C_1{\frac{p+\log(1/\delta)}{n}}<1$, we have
	\begin{equation*}
	\opnorm{\hat{\bar{\Sigma}}-\Sigma}^2 \leq C\left(
	\left(\frac{p}{n}\vee\epsilon^2\right)+\frac{\log(1/\delta)}{n}\right) ,
	\end{equation*}%
	with $\mathbb{P}_{(\epsilon ,\Sigma ,Q)}$-probability at least $1-2\delta$ uniformly over all $Q$ and $\Sigma \in \mathcal{F}(M)$.
\end{proposition}
	
\begin{remark}
Similar minimax rate optimality results can be established for banded, bandable and sparse covariance matrix estimation as well as sparse PCA problem stated in Theorems \ref{thm:band}, \ref{thm:bandable}, \ref{thm:sparse} and \ref{thm:pca}. Specifically, with $\mathcal{V}_{[l,l+2k-1]}$ replaced by $\bar{\mathcal{V}}_{[l,l+2k-1]}$, a $(1/2)$-net of $\mathcal{V}_{[l,l+2k-1]}$ in the estimator defined in (\ref{eq:bandgamma}), Theorems \ref{thm:band}-\ref{thm:bandable} still hold, and with $\mathcal{V}_{S}$ replaced by $\bar{\mathcal{V}}_{S}$, a $(1/2)$-net of $\mathcal{V}_{S}$ in the estimator defined in (\ref{eq:sparsegamma}), Theorems \ref{thm:sparse} and \ref{thm:pca} still hold. Similar results can also be established for corresponding scatter matrix estimation problems considered in Section \ref{sec:elliptical}. We omit the details here. 
\end{remark}	
}

\subsection{Simulation Results}

Before presenting the simulation results, we introduce some other robust covariance/scatter matrix estimators that we will compare with. The definitions of these robust matrix estimator are all up to some scaling factor. First we introduce Tyler's M-estimator \citep{tyler1987distribution}, defined as a solution of
$$\sum_{i=1}^n\frac{X_iX_i^T}{X_i^T\Sigma^{-1}X_i}=c\Sigma,\quad\text{for some }c>0.$$
Note that it is a special case of Maronna's M-estimator \citep{maronna1976robust}. Properties of Tyler's M-estimator have been studied by \cite{dumbgen1998tyler,dumbgen2005breakdown,zhang2002some,zhang2016marvcenko}. The second one is the scaled Kendall's tau estimator. The Kendall's tau correlation coefficient \citep{kendall1938new} between the $j$th and $k$th variables is defined as
$$\hat{\tau}_{jk}=\frac{2}{n(n-1)}\sum_{i<i'}\text{sign}\left((X_i-X_{i'})_j(X_i-X_{i'})_k\right).$$
Then, $\hat{K}=(\hat{K}_{jk})$ with $\hat{K}_{jk}=\sin\left(\frac{\pi}{2}\hat{\tau}_{jk}\right)$ is an estimator of the correlation matrix \citep{kruskal1958ordinal,han2014scale}. To obtain an estimator for the covariance/scatter matrix, define a diagonal matrix $\hat{S}$ with diagonal entries $\hat{S}_{jj}=\textsf{Median}(\{X_{ij}^2\}_{i=1}^n)$. Then, the scaled Kendall's tau estimator for the scatter matrix is
$$\hat{S}^{1/2}\hat{K}\hat{S}^{1/2}.$$
Thirdly, we introduce the minimum volume ellipsoid estimator (MVE) by \cite{leroy1987robust}. It finds the ellipsoid covering at least $n/2$ points of $\{X_i\}_{i=1}^n$ and then use the shape of the ellipsoid as the covariance matrix estimator. Finally, a related estimator is called the minimum covariance determinant estimator (MCD) \citep{leroy1987robust}. It finds $n/2$ points of $\{X_i\}_{i=1}^n$ for which the determinant of the sample covariance is minimal. The sample covariance of the selected $n/2$ points is used as an estimator. Properties of MVE and MCD have been studied by \cite{davies1992asymptotics,croux1999influence,butler1993asymptotics}. All these four estimators can be computed in \textsc{R}. Tyler's M-estimator can be computed by the package \textsf{ICSNP} \citep{nordhausen2007icsnp}. Kendall's tau correlation coefficient is included in the basic \textsc{R} package \textsf{stats}. The MVE and MCD can be computed by the package \textsf{MASS} \citep{ripley2011mass}. For comparison of performances, we rescale all the estimators by some constant factors so that all of them are targeted at the population scatter matrix of a canonical elliptical distribution.

The experiments cover the following five scenarios. The first three scenarios consider a degenerate contamination distribution while the last two consider some non-degenerate contaminations, which are motivated by the scenarios considered in \cite{pena2001multivariate}. 
\paragraph{AR covariance with Gaussian distribution + degenerate contamination}
Consider a covariance matrix with an autoregressive structure. That is, $\Sigma=(\sigma_{ij})$ with $\sigma_{ij}=4\left(\frac{1}{2}\right)^{|i-j|}$. The data is generated by $(1-\epsilon)N(0,\Sigma)+\epsilon Q$ where $Q((10,...,10)^T)=1$. 
\paragraph{AR covariance with $t$-distribution + degenerate contamination}
For the same autoregressive covariance matrix $\Sigma=(\sigma_{ij})$ with $\sigma_{ij}=4\left(\frac{1}{2}\right)^{|i-j|}$, we generate the data from $(1-\epsilon)T_{\Sigma}+\epsilon Q$, where $T_{\Sigma}$ is a $t$ distribution with degrees of freedom $5$ and $Q((10,...,10)^T)=1$. The density function of $T_{\Sigma}$ is proportional to $(1+x^T\Sigma^{-1}x/5)^{-\frac{p+5}{2}}$. 
\paragraph{Wishart covariance with Gaussian distribution + degenerate contamination}
We first generate a matrix $S$ from the Wishart distribution $\mathcal{W}_p(10I_p,100)$, and then let $\Sigma=S/100$. The data is generated by $(1-\epsilon)N(0,\Sigma)+\epsilon Q$ where $Q((10,...,10)^T)=1$.
\paragraph{AR covariance with Gaussian distribution + uniform  contamination}
 $\Sigma=(\sigma_{ij})$ with $\sigma_{ij}=4\left(\frac{1}{2}\right)^{|i-j|}$. The data is generated by $(1-\epsilon)N(0,\Sigma)+\epsilon Q$ where each coordinate of $X\sim Q$ follows a uniform distribution $U[-5, 5]$ independently.
\paragraph{Wishart covariance with Gaussian distribution +  Gaussian  contamination}
 We first generate a matrix $S$ from the Wishart distribution $\mathcal{W}_p(10I_p,100)$, and then let $\Sigma=S/100$. The data is generated by $(1-\epsilon)N(0,\Sigma)+\epsilon Q$ where $Q$ is $N((10,...,10)^T, I)$.

\smallskip
\smallskip

The above experiments cover the cases $n\in\{200,300,500,1000,1500,2000\}$, $p=10$ and $\epsilon\in\{0,0.02,0.04,0.06,0.08,0.1\}$. For each configuration, we measure the error by the average operator norm over $100$ independent experiments. The results are plotted in Figures \ref{fig:1}-\ref{fig:5}.  Scenarios 3 and 5 also cover $\epsilon\in\{0.12,0.14,0.16,0.18,0.2\}$ for a more complete investigation of the behavior of MCD (see Figure \ref{fig:3}). Moreover, for the case $\epsilon=0$, we also demonstrate the statistical efficiency of these robust estimators in Tables \ref{tab:1}-\ref{tab:5} by comparing the errors with those of the MLEs.

\begin{figure}[tbp]
\centering
\caption{Simulation results for Scenario 1.}
\includegraphics[width=5in]{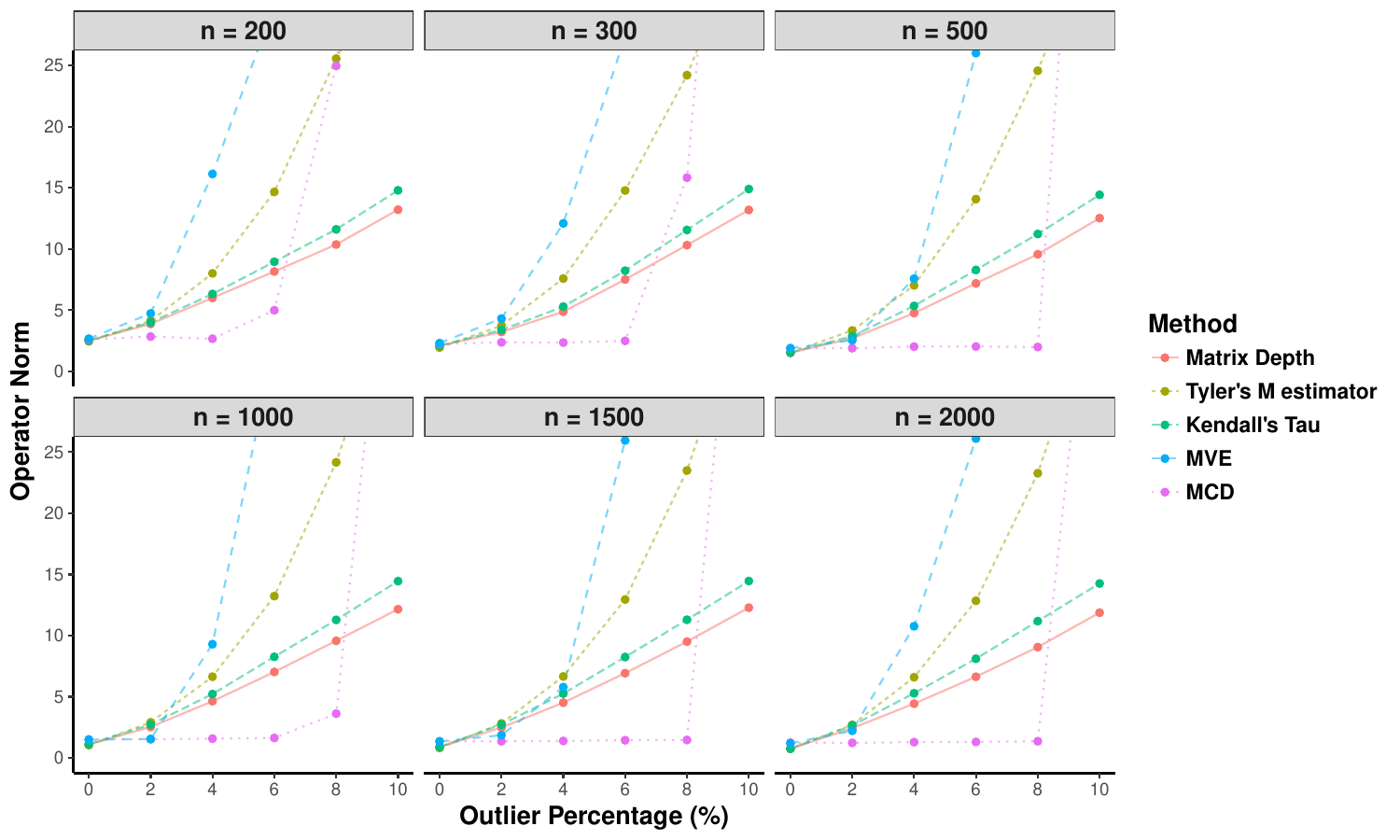}
\label{fig:1}
\end{figure}

\begin{table}[tbp]
\caption{Estimation errors for Scenario 1 when $\epsilon=0$.}
\label{tab:1}\centering
{\footnotesize \centering
\begin{tabular}{cccccccc} \hline
$n$ & $\epsilon$ & MLE & Matrix depth & Tyler's M  & Kendall's tau & MVE & MCD  \\ \hline
200 & 0  & 1.95 & 2.5 & 2.45 & 2.52 & 2.67 & 2.61 \\
300 & 0  & 1.68 & 2.08 & 1.94 & 2.11 & 2.31 & 2.25\\
500 & 0  & 1.38 & 1.51 & 1.53 & 1.53 & 1.89 & 1.89\\
1000 & 0 & 0.98& 1.1 & 1.05 & 1.12 & 1.5 & 1.51\\
1500 & 0 & 0.76 & 0.89 & 0.82 & 0.9 & 1.35 & 1.36\\
2000 & 0 & 0.68 & 0.75 & 0.75 & 0.76 & 1.19 & 1.25\\ \hline
\end{tabular}
}
\end{table}

\begin{figure}[tbp]
\centering
\caption{Simulation results for Scenario 2.}
\includegraphics[width=5in]{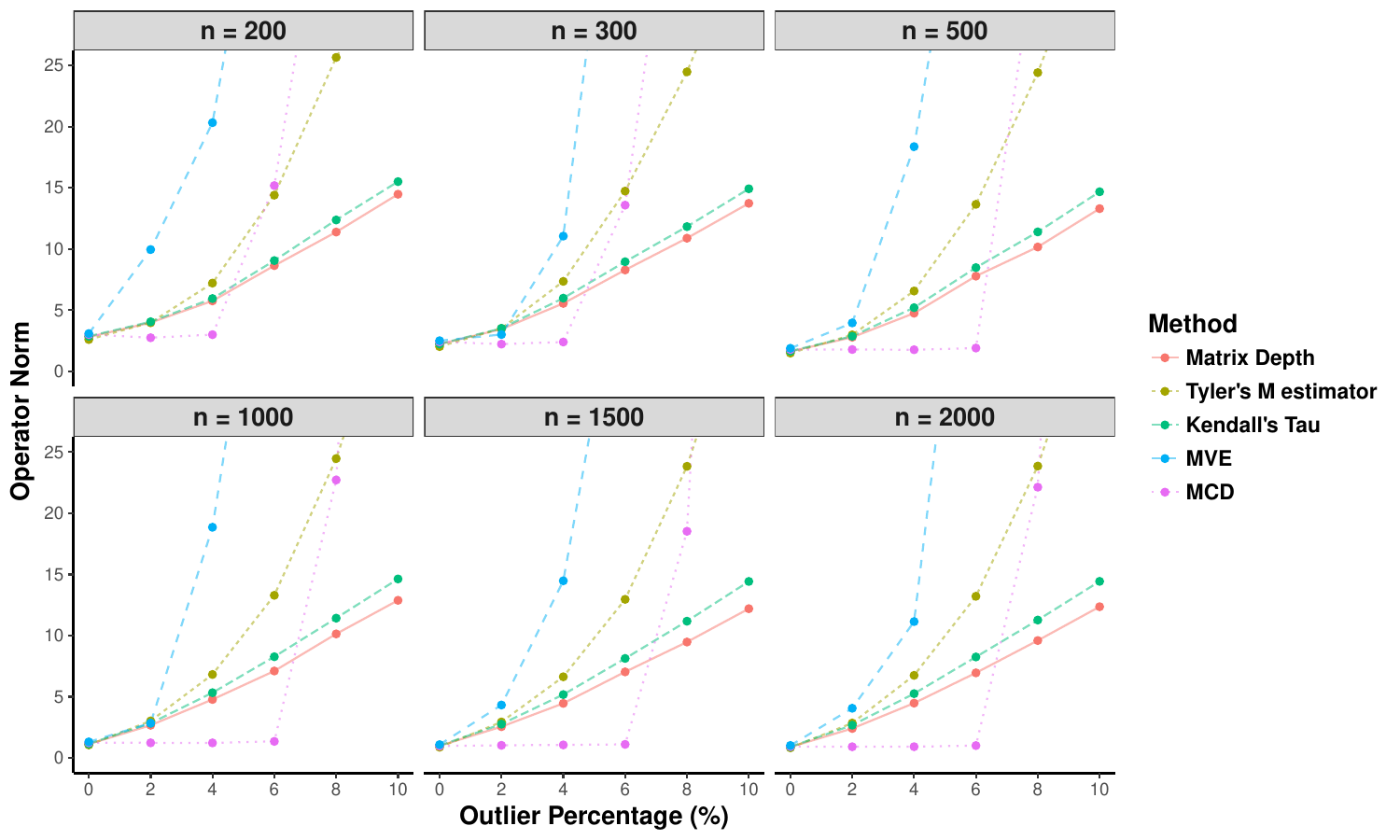}
\label{fig:2}
\end{figure}

\begin{table}[tbp]
\caption{Estimation errors for Scenario 2 when $\epsilon=0$.}
\label{tab:2}\centering
{\footnotesize \centering
\begin{tabular}{cccccccc} \hline
$n$ & $\epsilon$ & MLE & Matrix depth & Tyler's M  & Kendall's tau & MVE & MCD \\ \hline
200 & 0 & 2.51 & 2.8 & 2.6 & 2.87 & 3.09 & 3.01 \\
300 & 0 & 1.98 & 2.22 & 2.02 & 2.24 & 2.5 & 2.41 \\
500 & 0 & 1.45 & 1.62 & 1.49 & 1.65 & 1.88 & 1.78 \\
1000 & 0 & 1 & 1.13 & 1.06 & 1.15 & 1.31 & 1.25 \\
1500 & 0 & 0.84 & 0.99 & 0.88 & 1.01 & 1.09 & 0.98 \\
2000 & 0 & 0.77 & 0.89 & 0.83 & 0.91 & 1.02 & 0.94  \\ \hline
\end{tabular}
}
\end{table} 

\begin{figure}[tbp]
\centering
\caption{Simulation results for Scenario 3.}
\includegraphics[width=5in]{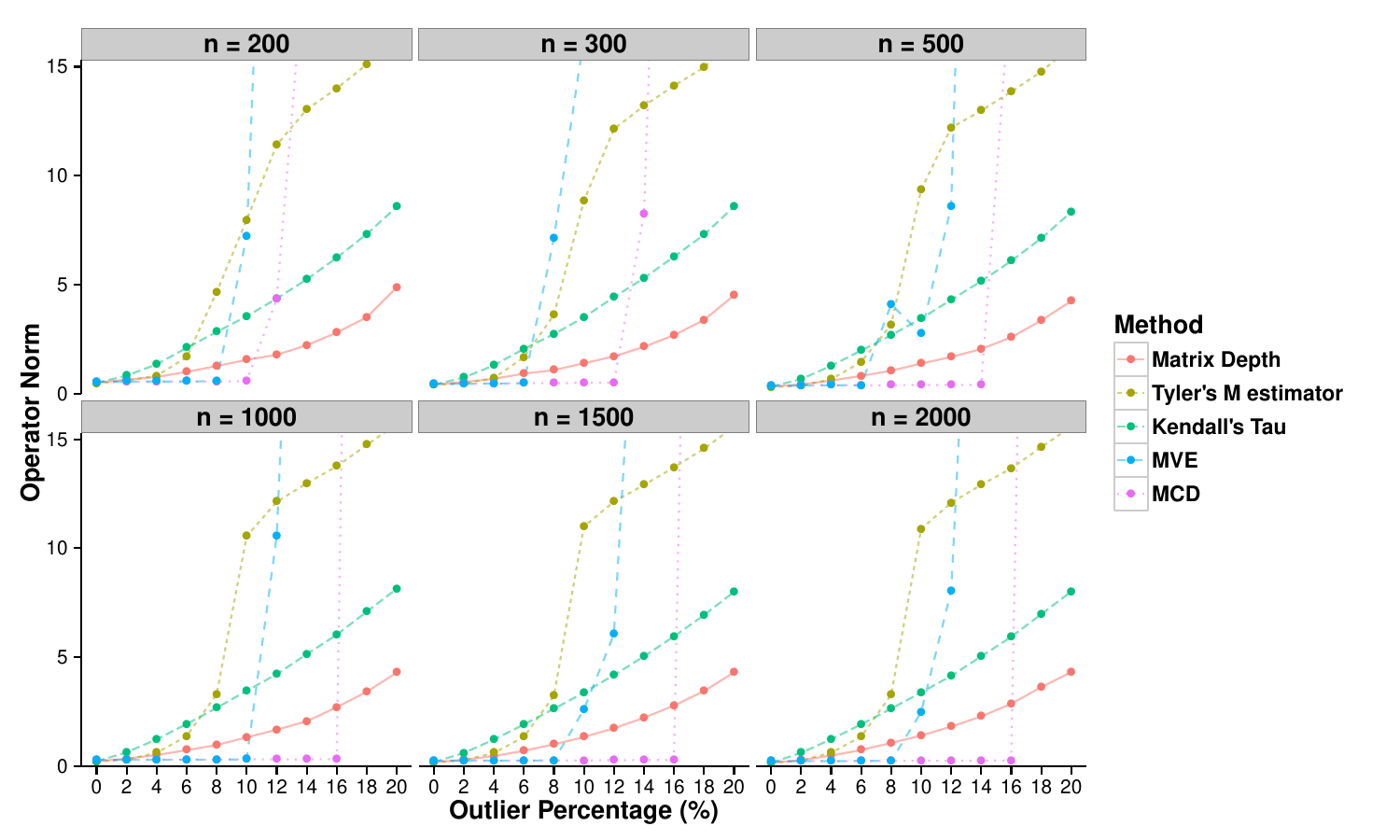}
\label{fig:3}
\end{figure}

\begin{table}[tbp]
\caption{Estimation errors for Scenario 3 when $\epsilon=0$.}
\label{tab:3}\centering
{\footnotesize \centering
\begin{tabular}{cccccccc} \hline
$n$ & $\epsilon$ & MLE & Matrix depth & Tyler's M  & Kendall's tau & MVE & MCD \\ \hline
200 & 0 & 0.4& 0.49 & 0.46 & 0.49 & 0.81 & 0.76\\
300 & 0 & 0.34  & 0.41 & 0.39 & 0.41 & 0.69 & 0.64\\
500 & 0 & 0.4 & 0.31 & 0.3 & 0.32 & 0.53 & 0.49\\
1000 & 0 & 0.19 & 0.22 & 0.21 & 0.22 & 0.37 & 0.35\\
1500 & 0& 0.15 & 0.17 & 0.17 & 0.18 & 0.3 & 0.28\\
2000 & 0 & 0.13 & 0.15 & 0.15 & 0.15 & 0.26 & 0.25 \\ \hline
\end{tabular}
}
\end{table}

\begin{figure}[tbp]
\centering
\caption{Simulation results for Scenario 4.}
\includegraphics[width=5in]{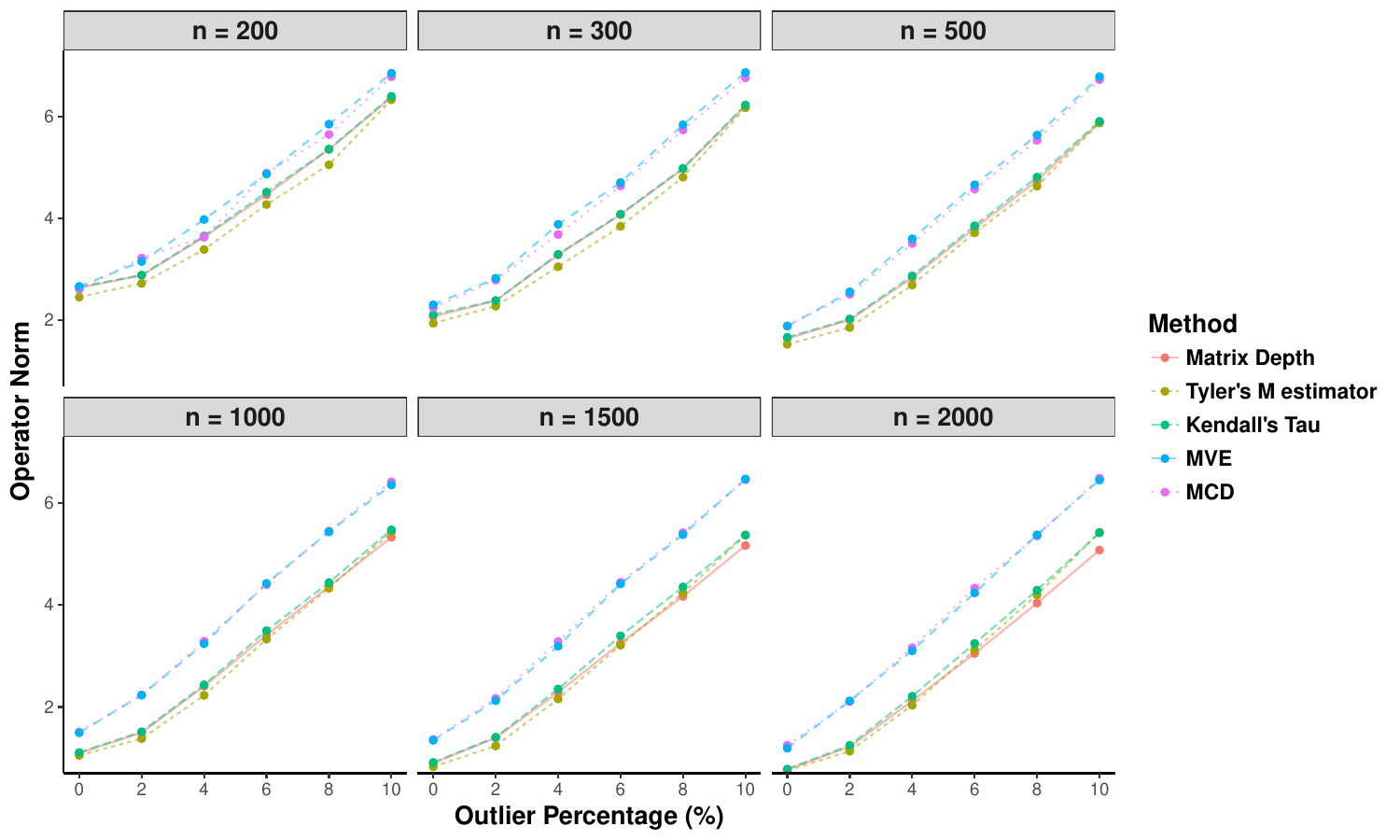}
\label{fig:4}
\end{figure}

\begin{table}[tbp]
\caption{Estimation errors for Scenario 4 when $\epsilon=0$.}
\label{tab:3}\centering
{\footnotesize \centering
\begin{tabular}{cccccccc} \hline
$n$ & $\epsilon$ & MLE & Matrix depth & Tyler's M  & Kendall's tau & MVE & MCD \\ \hline
200 & 0 & 1.95 & 2.64 & 2.45 & 2.66 & 2.67 & 2.61 \\
300 & 0 & 1.68 & 2.07 & 1.94 & 2.11 & 2.31 & 2.25 \\
500 & 0 & 1.38 & 1.65 & 1.53 & 1.67 & 1.89 & 1.89 \\
1000 & 0 & 0.98 & 1.1 & 1.05 & 1.11 & 1.5 & 1.51 \\
1500 & 0 & 0.76 & 0.9 & 0.83 & 0.91 & 1.35 & 1.36 \\
2000 & 0 & 0.68 & 0.77 & 0.75 & 0.78 & 1.19 & 1.25 \\ \hline
\end{tabular}
}
\end{table}

\begin{figure}[tbp]
\centering
\caption{Simulation results for Scenario 5.}
\includegraphics[width=5in]{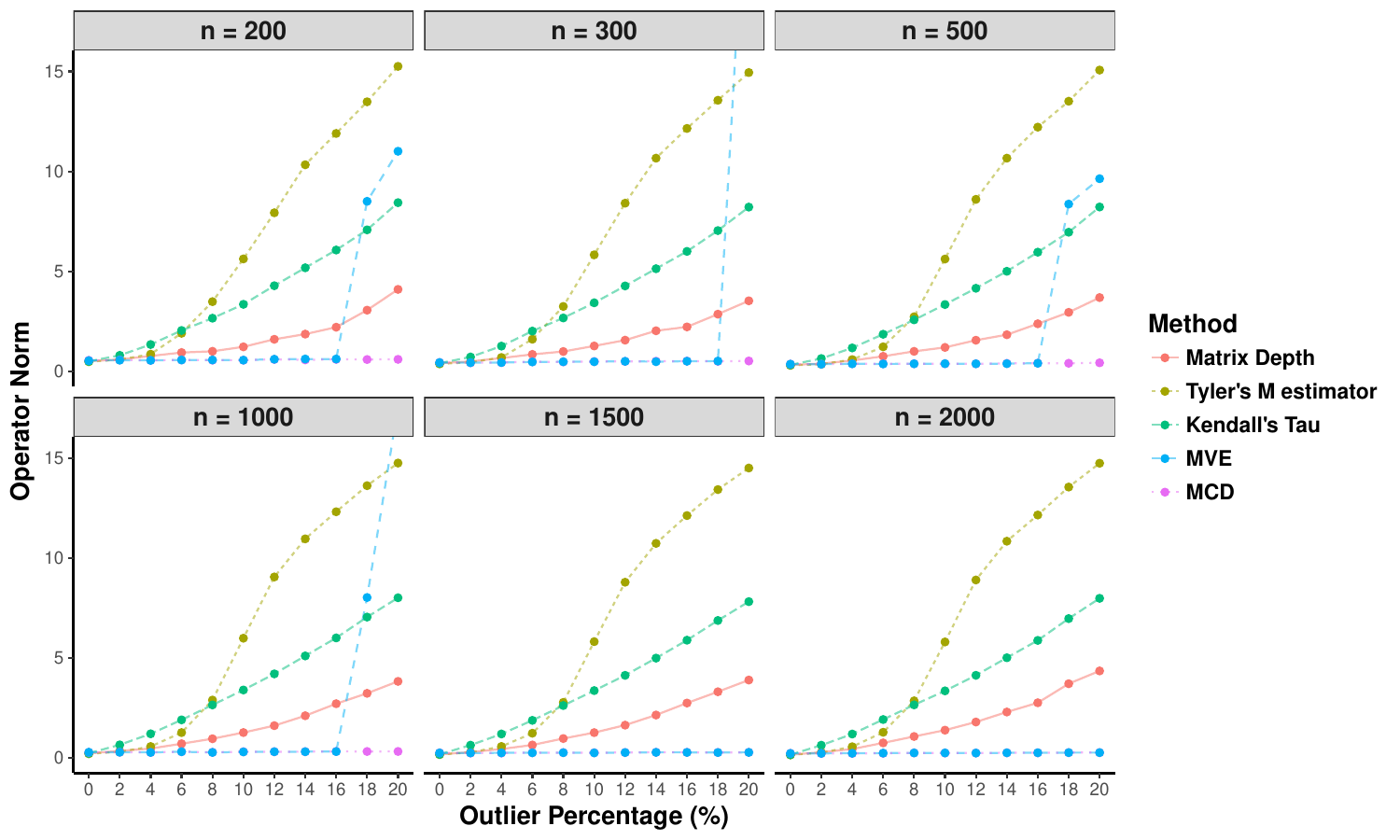}
\label{fig:5}
\end{figure}

\begin{table}[tbp]
\caption{Estimation errors for Scenario 5 when $\epsilon=0$.}
\label{tab:5}\centering
{\footnotesize \centering
\begin{tabular}{cccccccc} \hline
$n$ & $\epsilon$ & MLE & Matrix depth & Tyler's M  & Kendall's tau & MVE & MCD \\ \hline
200 & 0 & 0.4& 0.51 & 0.48 & 0.5 & 0.53 & 0.55 \\
300 & 0 & 0.34 & 0.4 & 0.37 & 0.4 & 0.44 & 0.44 \\
500 & 0 & 0.27 & 0.32 & 0.3 & 0.32 & 0.36 & 0.36 \\
1000 & 0 & 0.19 & 0.22 & 0.21 & 0.24 & 0.29 & 0.28 \\
1500 & 0 & 0.15 & 0.18 & 0.17 & 0.18 & 0.25 & 0.26 \\
2000 & 0 &  0.13 & 0.16 & 0.15 & 0.16 & 0.23 & 0.23 \\ \hline
\end{tabular}
}
\end{table}

Figures \ref{fig:1}-\ref{fig:5} show different behaviors for the five robust estimators under the $\epsilon$-contamination model. In general, matrix depth, Tyler's M and Kendall's tau show more or less similar patterns as $\epsilon$ increases, and MVE is similar to MCD. In all five scenarios, both MVE and MCD are not stable to the contamination proportion $\epsilon$. The errors of both estimators rise abruptly after a certain threshold $\epsilon$, even though MCD is very competitive when $\epsilon$ is small. On the other hand, the increase of errors of matrix depth, Tyler's M and Kendall's tau are more gradual. Among these three estimators, the matrix depth estimator demonstrates the best error behavior against contamination.

Compared to the first two scenario, the points where the errors of MVE and MCD rise abruptly appear later in Scenario 3. There are five cases where MVE is very stable until $\epsilon\geq 0.08$. The errors of MCD are stable for $\epsilon\leq 0.16$, but explode after that. For the other three estimators, the matrix depth estimator shows a more significant advantage over Tyler's M and Kendall's tau than the first two scenarios.

Scenario 4 demonstrates an interesting grouping of the five estimators under uniform contamination. The errors of MVE and MCD are almost identical. The errors of matrix depth, Tyler's M and Kendall's tau are also very similar, and are significantly better that those of MVE and MCD.

In contrast, Scenario 5 shows a different conclusion. The error of MVE is always very small. MCD shows a similar behavior but its error starts to explode when $\epsilon$ passes some threshold. Tyler's M is not favored by this scenario, and matrix depth still demonstrates its advantage over Kendall's tau.

The case $\epsilon=0$ is particularly interesting, where we can compare the statistical efficiency of the five estimators when there is no contamination. The results are reported in Tables \ref{tab:1}-\ref{tab:5}, benchmarked by the MLEs of Gaussian and $t$ distributions. All the five estimators show similar efficiency loss compared with the MLE. The errors of matrix depth, Tyler's M and Kendall's tau are better than those of MVE and MCD, especially when $n$ is large.

\begin{figure}[tbp]
\centering
\caption{Behaviors of errors when $\epsilon$ varies.}
\includegraphics[width=5in]{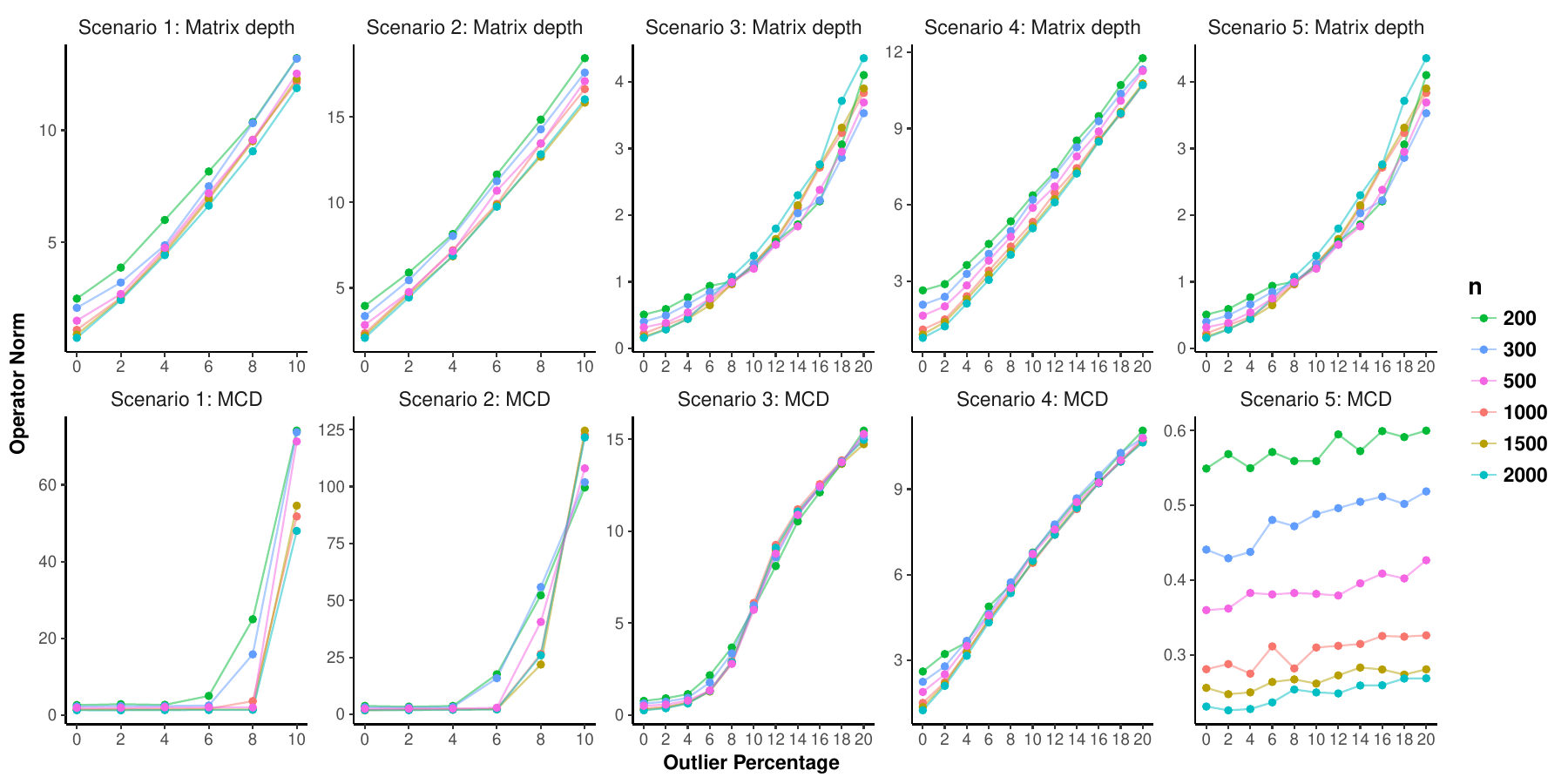}
\label{fig:6}
\end{figure}

Finally, we summarize the behaviors of the matrix depth estimator and the MCD in all the three scenarios in Figure \ref{fig:6}. These two estimators are representative of the five mentioned above. Note that more experiments for MCD are performed in Scenarios 3-5 to get a more complete view of the property of the estimator. According to {Theorem \ref{thm:matrix upper} and the more relevant Proposition \ref{thm:matrix upper_app}}, the convergence rate of the matrix depth estimator under the operator norm is $\sqrt{\frac{p}{n}}\vee\epsilon\asymp \sqrt{\frac{p}{n}}+\epsilon$. This rate is minimax optimal according to Theorem \ref{thm:l1}. Figure \ref{fig:6} clearly shows an approximately linear dependence on $\epsilon$, which is well predicted by our theory. On the other hand, the simulation results of MCD do not reflect a linear dependence on $\epsilon$, which is an evidence of sub-optimality.

\renewcommand{\theequation}{C.\arabic{equation}}
\setcounter{equation}{0}

\section{Additional Proofs in Section 2}

\begin{proof}[Proof of Theorem \ref{thm:location upper}]
Since our estimator is affine invariant, without loss of generality we consider the case where $\theta=0$.
By Lemma \ref{lem:ratio}, we decompose the data $\{X_i\}_{i=1}^n=\{Y_i\}_{i=1}^{n_1}\cup\{Z_i\}_{i=1}^{n_2}$. The following analysis is conditioning on the set of $(n_1,n_2)$ that satisfies (\ref{eq:ratio}). Define half space $H_{u,\eta }=\{y:u^{T}y\leq u^{T}\eta \}$.
Recall that the Tukey's depth of $\eta $ with respect to $%
P_{\theta }$ and its empirical counterpart are
\begin{eqnarray*}
\mathcal{D}(\eta ,P_{\theta }) &=&\inf_{u\in S^{p-1}}P_{\theta }(H_{u,\eta
})=\inf_{u\in S^{p-1}}P_{\theta }\{u^{T}Y\leq u^{T}\eta \}, \\
\mathcal{D}(\eta ,\{Y_{i}\}_{i=1}^{n_1}) &=&\inf_{u\in S^{p-1}}\mathbb{P}%
_{n_1}(H_{u,\eta })=\min_{u\in S^{p-1}}\frac{1}{n_1}\sum_{i=1}^{n_1}%
\mathbb{I}\{u^{T}Y_{i}\leq u^{T}\eta \},
\end{eqnarray*}%
where $\mathbb{P}_{n_1}$ denotes
the empirical distribution of $\{Y_{i}\}_{i=1}^{n_1}$. The class of set
functions $\{\mathbb{I}_{H_{u,\eta }}:u\in S^{p-1},\eta \in \mathbb{R}^{p}\}$ consists of
all half spaces in $\mathbb{R}^{p}$ and hence has VC dimension $p+1$ \cite{vapnik1971uniform}. Then following a similar analysis of Lemma \ref{lem:DKW} (or alternatively, see standard empirical processes theory, for instance, \cite[Theorems~5, 6]{bartlett2003rademacher}), we can obtain
that for any $\delta >0$, with probability at least $1-\delta $, we have
\begin{equation*}
\sup_{u,\eta }\left\vert P_{\theta }(H_{u,\eta })-\mathbb{P}%
_{n_1}(H_{u,\eta })\right\vert \leq \sqrt{\frac{1440e\pi}{1-e^{-1}}}\sqrt{\frac{p+1}{n_1}}+\sqrt{\frac{\log (1/\delta )}{2n_1}}.
\end{equation*}%
As an immediate consequence, we have with probability at least $1-\delta $,
\begin{equation}
\sup_{\eta }\left\vert \mathcal{D}(\eta ,P_{\theta })-\mathcal{D}(\eta
,\{Y_{i}\}_{i=1}^{n_1})\right\vert \leq \sqrt{\frac{1440e\pi}{1-e^{-1}}}\sqrt{\frac{p+1}{n_1}}+\sqrt{\frac{\log (1/\delta )}{2n_1}}.  \label{eqn:fact1}
\end{equation}
We lower bound $\mathcal{D}(\hat{\theta},P_{\theta})$ by
\begin{eqnarray}
&&\mathcal{D}(\hat{\theta},P_{\theta}) \nonumber \\
\label{eq:vc1}  &\geq& \mathcal{D}(\hat{\theta},\{Y_i\}_{i=1}^{n_1}) - \sqrt{\frac{1440e\pi}{1-e^{-1}}}\sqrt{\frac{p+1}{n_1}}-\sqrt{\frac{\log (1/\delta )}{2n_1}} \\
\label{eq:dfact1} &\geq& \frac{n}{n_1}\mathcal{D}(\hat{\theta},\{X_i\}_{i=1}^n)-\frac{n_2}{n_1}-\sqrt{\frac{1440e\pi}{1-e^{-1}}}\sqrt{\frac{p+1}{n_1}}-\sqrt{\frac{\log (1/\delta )}{2n_1}} \\
\label{eq:deftukey} &\geq& \frac{n}{n_1}\mathcal{D}(\theta,\{X_i\}_{i=1}^n)-\frac{n_2}{n_1}-\sqrt{\frac{1440e\pi}{1-e^{-1}}}\sqrt{\frac{p+1}{n_1}}-\sqrt{\frac{\log (1/\delta )}{2n_1}} \\
\label{eq:dfact2} &\geq& \mathcal{D}(\theta,\{Y_i\}_{i=1}^{n_1})-\frac{n_2}{n_1}-\sqrt{\frac{1440e\pi}{1-e^{-1}}}\sqrt{\frac{p+1}{n_1}}-\sqrt{\frac{\log (1/\delta )}{2n_1}} \\
\label{eq:vc2} &\geq& \mathcal{D}(\theta,P_{\theta})-\frac{n_2}{n_1}-2\sqrt{\frac{1440e\pi}{1-e^{-1}}}\sqrt{\frac{p+1}{n_1}}-\sqrt{\frac{2\log (1/\delta )}{n_1}} \\
\label{eq:half} &=& \frac{1}{2}-\frac{n_2}{n_1}-2\sqrt{\frac{1440e\pi}{1-e^{-1}}}\sqrt{\frac{p+1}{n_1}}-\sqrt{\frac{2\log (1/\delta )}{n_1}}.
\end{eqnarray}
The inequalities (\ref{eq:vc1}) and (\ref{eq:vc2}) are by (\ref{eqn:fact1}). The inequalities (\ref{eq:dfact1}) and (\ref{eq:dfact2}) are due to the property of depth function that
$$n_1\mathcal{D}(\eta,\{Y_i\}_{i=1}^{n_1})\geq n\mathcal{D}(\eta,\{X_i\}_{i=1}^{n})-n_2\geq n_1\mathcal{D}(\eta,\{Y_i\}_{i=1}^{n_1})-n_2,$$
for any $\eta\in\mathbb{R}^p$. The inequality (\ref{eq:deftukey}) is by the definition of $\hat{\theta}$. Finally, the equality (\ref{eq:half}) is because $P_{\theta}=N(\theta,I_p)=N(0,I_p)$, so that
\begin{equation}
\mathcal{D}(\eta ,P_{\theta})=1-\Phi (\left\Vert \eta \right\Vert ),
\label{eqn:fact2}
\end{equation}
for any $\eta\in\mathbb{R}^p$, where $\Phi(\cdot)$ is the cumulative distribution function of $N(0,1)$.
Combining (\ref{eq:half}) with (\ref{eqn:fact2}) and (\ref{eq:ratio}), we have
$$\Phi(\norm{\hat{\theta}})\leq\frac{1}{2}+\frac{\epsilon}{1-\epsilon}+40\sqrt{\frac{6e\pi}{1-e^{-1}}}\sqrt{\frac{p+1}{n}}+\frac{7}{2}\sqrt{\frac{\log (1/\delta )}{n}},$$
with probability at least $1-2\delta$.
Note that $\Phi(\norm{\hat{\theta}})-1/2=\int_0^{\norm{\hat{\theta}}}(2\pi)^{-1/2}e^{-t^2/2}dt$ and $e^{-t^2/2}$ is bounded away from $0$ in the neighborhood of $t=0$. Thus, under the assumption that $\epsilon<1/5$ and $\sqrt{p/n}+\sqrt{\log(1/\delta)/n}$ are sufficiently small, we obtain the bound $\norm{\hat{\theta}}\leq C\left(\sqrt{p/n}+\epsilon+\sqrt{\log(1/\delta)/n}\right)$ with probability at least $1-2\delta$, where $C$ is an absolute constant.
\end{proof}

\begin{proof}[Proof of Proposition \ref{pro:location cm}]
Given the conclusion of Lemma \ref{lem:qiantao}, it is sufficient to consider the case $\epsilon\leq 1/4$.
Recall the constant $c'>0$ in Lemma \ref{lem:ratio}.
When $\epsilon^2\leq (\frac{64\log(12)}{(c')^2}\vee 1) n^{-1}$, the classical minimax lower bound implies
$$\sup_{\theta,Q}\mathbb{P}_{(\epsilon,\theta,Q)}\left\{\norm{\hat{\theta}-\theta}^2\geq C'p/n\right\}>c,$$
for some small constants $0<c<1/3$ and $C'>0$, by considering the case $Q=P_{\theta}$.
Since $1/n\gtrsim (1/n)\vee\epsilon^2$, we have
$$\sup_{\theta,Q}\mathbb{P}_{(\epsilon,\theta,Q)}\left\{\norm{\hat{\theta}-\theta}^2\geq Cp(n^{-1}\vee\epsilon^2)\right\}>c,$$ where $C=C'(1\wedge\frac{(c')^2}{64\log(12)})$ Hence, it is sufficient to consider the case $\epsilon^2> (\frac{64\log(12)}{(c')^2}\vee 1) n^{-1}$.

Let us consider the distribution $\mathbb{P}=\mathbb{P}_{(\epsilon,0,Q)}$ with $Q\{Z=(1,1,...,1)^T\}=1$. Decompose the observations into $\{X_i\}_{i=1}^n=\{Y_i\}_{i=1}^{n_1}\cup\{Z_i\}_{i=1}^{n_2}$ as in Lemma \ref{lem:ratio}. Then for each $j\in[p]$, define the event
$$
E_j=\left\{\sup_{\eta}\left|\frac{1}{n_1}\sum_{i=1}^{n_1}\mathbb{I}\{Y_{ij}\leq \eta\}-\Phi(\eta)\right|\leq\sqrt{\frac{\log(12)}{n}}\right\},
$$
and
$$E=\left\{c'\epsilon<\frac{n_2}{n_1}<\frac{1}{2}\right\},$$
with $c'>0$ specified in Lemma \ref{lem:ratio}.
We claim that
\begin{equation}
\mathbb{P}\left\{\sum_{j=1}^p\mathbb{I}_{E_j}\geq c_1p, E\right\}\geq \frac{1}{3},\label{eq:2claim}
\end{equation}
for some small constant $c_1>0$.
We will establish the conclusion of Proposition \ref{pro:location cm} by assuming (\ref{eq:2claim}) holds.
The inequality (\ref{eq:2claim}) will be proved in the end. Let us first show that $E\cap E_j\subset \{\hat{\theta}_j^2\geq c_2\epsilon^2\}$, where the absolute constant $0<c_2<1$ depends on $c'$ only. For each $\eta<1$, we have
\begin{eqnarray*}
&&n\mathcal{D}(\eta,\{X_{ij}\}_{i=1}^n) \\
&=&\left(\sum_{i=1}^{n_1}\mathbb{I}\{Y_{ij}>\eta\}+\sum_{i=1}^{n_2}\mathbb{I}\{Z_{ij}>\eta\}\right)\wedge\left(\sum_{i=1}^{n_1}\mathbb{I}\{Y_{ij}\leq\eta\}+\sum_{i=1}^{n_2}\mathbb{I}\{Z_{ij}\leq\eta\}\right) \\
&=& \left(\sum_{i=1}^{n_1}\mathbb{I}\{Y_{ij}>\eta\}+n_2\right)\wedge\left(\sum_{i=1}^{n_1}\mathbb{I}\{Y_{ij}\leq\eta\}\right).
\end{eqnarray*}
Therefore, the event $E_j$ implies
\begin{equation}
\sup_{\eta< 1}\left|\frac{n}{n_1}\mathcal{D}(\eta,\{X_{ij}\}_{i=1}^n)-\left\{(1-\Phi(\eta)+n_2/n_1)\wedge\Phi(\eta)\right\}\right|\leq \sqrt{\frac{\log(12)}{n}}.\label{eq:zhaoge}
\end{equation}
When $\hat{\theta}_j< 1$, we have
\begin{eqnarray}
\nonumber \Phi(\hat{\theta}_j) &\geq& (1-\Phi(\hat{\theta}_j)+n_2/n_1)\wedge\Phi(\hat{\theta}_j) \\
\label{eq:haha1}&\geq& \frac{n}{n_1}\mathcal{D}(\hat{\theta}_j,\{X_{ij}\}_{i=1}^n)-\sqrt{\frac{\log(12)}{n}} \\
\label{eq:defcm}&\geq& \frac{n}{n_1}\mathcal{D}(\theta_j^*,\{X_{ij}\}_{i=1}^n)-\sqrt{\frac{\log(12)}{n}} \\
\label{eq:haha2}&\geq&  (1-\Phi(\theta_j^*)+n_2/n_1)\wedge\Phi(\theta_j^*)-2\sqrt{\frac{\log(12)}{n}} \\
\nonumber&=& \frac{1}{2}+\frac{n_2}{2n_1}-2\sqrt{\frac{\log(12)}{n}},
\end{eqnarray}
where $\theta^*$ is defined by the equation $(1-\Phi(\theta_j^*)+n_2/n_1)\wedge\Phi(\theta_j^*)=\frac{1}{2}+\frac{n_2}{2n_1}$, which is guaranteed to have a solution when $n_2/n_1<1/2$ under the event $E$.  The inequalities (\ref{eq:haha1}) and (\ref{eq:haha2}) are due to (\ref{eq:zhaoge}) and the inequality (\ref{eq:defcm}) is by the definition of $\hat{\theta}$. By  $\epsilon^2> \frac{64\log(12)}{(c')^2} n^{-1}$ and the event $E$, we have
$$\Phi(\hat{\theta}_j)\geq\frac{1}{2}+\frac{1}{4}c'\epsilon,$$
which implies $\hat{\theta}_j^2\geq c_2\epsilon^2$. When $\hat{\theta_j}<1$ does not hold, we have $\hat{\theta}_j^2\geq 1\geq \epsilon^2$. This establishes $E\cap E_j\subset\{\hat{\theta}_j^2\geq c_2\epsilon^2\}$, noting $c_2<1$. Hence, when $\epsilon^2> \frac{64\log(12)}{(c')^2} n^{-1}$, we can pick small enough constant $C>$ such that
\begin{eqnarray*}
&& \sup_{\theta,Q}\mathbb{P}_{(\epsilon,\theta,Q)}\left\{\norm{\hat{\theta}-\theta}^2\geq Cp(\epsilon^2\vee n^{-1})\right\} \\
&\geq& \sup_{\theta,Q}\mathbb{P}_{(\epsilon,\theta,Q)}\left\{\norm{\hat{\theta}-\theta}^2\geq c_1c_2p\epsilon^2\right\} \\
 &\geq& \mathbb{P}\left\{\sum_{j=1}^p\hat{\theta}_j^2\geq c_1c_2p\epsilon^2\right\} \geq \mathbb{P}\left\{\sum_{j=1}^p\mathbb{I}_{E_j}\geq c_1p, E\right\} \geq \frac{1}{3},
\end{eqnarray*}
where recall $\mathbb{P}=\mathbb{P}_{(\epsilon,0,Q)}$ in the above argument. This gives the desired conclusion, noting $c<1/3$.

Finally, let us prove (\ref{eq:2claim}) to close proof. First, we have $\mathbb{P}(E)\geq 2/5$ by Lemma \ref{lem:ratio} and the assumptions $\epsilon<1/4$ and $\epsilon^2>n^{-1}$. Moreover,
\begin{eqnarray*}
\mathbb{P}(E_j) &\geq& \mathbb{P}\left\{\sup_{\eta}\left|\frac{1}{n_1}\sum_{i=1}^{n_1}\mathbb{I}\{Y_{ij}\leq \eta\}-\Phi(\eta)\right|\leq \sqrt{\frac{\log(12)}{n}}\Big| E\right\}  \mathbb{P}(E) \\
&\geq& \frac{2}{5}\min_{m\geq n/2}\mathbb{P}\left\{\sup_{\eta}\left|\frac{1}{m}\sum_{i=1}^{m}\mathbb{I}\{Y_{ij}\leq \eta\}-\Phi(\eta)\right|\leq \sqrt{\frac{\log(12)}{n}}\right\} \\
&\geq& \frac{1}{3},
\end{eqnarray*}
where the last inequality follows from DKW inequality \cite{massart1990tight}. Therefore, by Hoeffding's inequality,
$$\mathbb{P}\left\{\sum_{j=1}^p\mathbb{I}_{E_j}\geq c_1p\right\}\geq 0.99,$$
for some small $c_1$ when $p$ is sufficiently large. Hence,
 $$\mathbb{P}\left\{\sum_{j=1}^p\mathbb{I}_{E_j}\geq c_1p, E\right\}\geq 1-\mathbb{P}\left\{\sum_{j=1}^p\mathbb{I}_{E_j}< c_1p\right\}-\mathbb{P}(E^c)>\frac{1}{3},$$
and the proof is complete.
\end{proof}

\renewcommand{\theequation}{D.\arabic{equation}}
\setcounter{equation}{0}

\section{Additional Proofs in Section 4}\label{sec:matrix elliptical}

Note that the proofs in Section \ref{sec:matrix} all depend on Theorem \ref%
{thm:general}. Similarly, all results in Section \ref{sec:elliptical} are
consequences of the following result analogous to Theorem \ref{thm:general}.

\begin{thm}
\label{thm:general elliptical} 
For some index subsets $S_1,\ldots,S_m\subset [p]$ with $\max_i|S_i|\leq s$, consider the estimator $\hat{\Gamma}$ defined in (\ref{eq:hatgamma}) with $\mathcal{U}=\cup_{i=1}^{m}\mathcal{V}_{S_i}$. Assume $\epsilon <\tau/3$, $(1+s)/n$ is sufficiently small and the distribution $P_{\Gamma}=EC_p(0,\Gamma,\eta)$ satisfies (\ref{eq:assell}).
Then for any $\delta\in(0,1/2)$ such that $n^{-1}\log (m/\delta )$ is sufficiently
small, we have
$$\sup_{u\in\mathcal{U}}\left|u^{T}(\hat{\Gamma}-\Gamma) u\right|\leq C\kappa^{1/2}\left(\epsilon+\sqrt{\frac{1+s+\log(m/\delta)}{n}}\right),$$
with $\mathbb{P}_{(\epsilon,\Gamma,Q)}$-probability at least $1-2\delta $
uniformly over all $Q$ and $\Gamma \in \mathcal{F}(M)\cap\mathcal{F}$,
where $C>0$ is some absolute constant.
\end{thm}
\begin{proof}
The proof is similar to that of Theorem \ref{thm:general}. We focus on
the difference and omit the overlapping content. In particular, the inequality (\ref{eq:verydeep}) can be derived by using the same argument under the elliptical distribution $P_{\Gamma}=EC_p(0,\Gamma,\eta)$. That is,
\begin{equation}
\mathcal{D}_u(\hat{\Gamma},P_{\Gamma}) \geq \frac{1}{2}-\frac{\epsilon}{1-\epsilon}-C_1\sqrt{\frac{1+s+\log(m/\delta)}{n}},\text{ for all }u\in\mathcal{U}, \label{eq:verydeep2}
\end{equation}
with probability at least $1-2\delta$. By Proposition \ref{prop:deepell} and the definition of $G(t)$ in (\ref{eq:G(t)}), we have
$$\frac{1}{2}-\mathcal{D}_u(\hat{\Gamma},P_{\Gamma})=\left|G(1)-G\left({\frac{u^T\hat{\Gamma}u}{u^T\Gamma u}}\right)\right|.$$
Combining this fact with (\ref{eq:verydeep2}), we have
$$\sup_{u\in\mathcal{U}}\left|G(1)-G\left({\frac{u^T\hat{\Gamma}u}{u^T\Gamma u}}\right)\right|\leq \frac{\epsilon}{1-\epsilon}+C_1\sqrt{\frac{1+s+\log(m/\delta)}{n}},$$
with probability at least $1-2\delta$. As long as $\epsilon<\tau/3$ and $\sqrt{\frac{1+s+\log(m/\delta)}{n}}$ is sufficiently small, we have $\frac{\epsilon}{1-\epsilon}+C_1\sqrt{\frac{1+s+\log(m/\delta)}{n}}<\tau$. By the assumption (\ref{eq:assell}), we must have $\left|1-{\frac{u^T\hat{\Gamma}u}{u^T\Gamma u}}\right|\leq\alpha$, which further implies
$$\sup_{u\in\mathcal{U}}\left|1-{\frac{u^T\hat{\Gamma}u}{u^T\Gamma u}}\right|\leq C_2\kappa^{1/2}\left(\epsilon+\sqrt{\frac{1+s+\log(m/\delta)}{n}}\right),$$
with probability at least $1-2\delta$. By the fact that $\Gamma\in\mathcal{F}(M)$, we have
$$\sup_{u\in\mathcal{U}}\left|u^T(\hat{\Gamma}-\Gamma)u\right|\leq C_3\kappa^{1/2}\left(\epsilon+\sqrt{\frac{1+s+\log(m/\delta)}{n}}\right),$$
with probability at least $1-2\delta$, which completes the proof.
\end{proof}

\begin{proof}[Proofs of Theorems \ref{thm:e1}-\ref{thm:eca}]
These results follow Theorem \ref{thm:general elliptical} and the arguments in the proofs of Theorem \ref{thm:matrix upper}, Theorem \ref{thm:band}, Theorem \ref{thm:bandable}, Theorem \ref{thm:sparse} and Theorem \ref{thm:pca}.
\end{proof}

\renewcommand{\theequation}{E.\arabic{equation}}
\setcounter{equation}{0}
\section{Additional Proofs in Sections 2, 3 and 5 on Lower Bounds}\label{sec:prooflower}

\begin{proof}[Proof of Theorem \ref{thm:lower}]
	When $\mathcal{M}(0)\geq\omega(\epsilon,\Theta)$, we have $\mathcal{M}(0)=\mathcal{M}(0)\vee\omega(\epsilon,\Theta)$. Thus,
	\begin{eqnarray*}
		&&\inf_{\hat{\theta}}\sup_{\theta\in\Theta}\sup_Q\mathbb{P}_{(\epsilon,\theta,Q)}\left\{L(\hat{\theta},\theta)\geq\mathcal{M}(0)\vee\omega(\epsilon,\Theta)\right\}\\
		&=&\inf_{\hat{\theta}}\sup_{\theta\in\Theta}\sup_Q\mathbb{P}_{(\epsilon,\theta,Q)}\left\{L(\hat{\theta},\theta)\geq\mathcal{M}(0)\right\}\geq c.
	\end{eqnarray*}
	It is sufficient to prove when $\mathcal{M}(0)<\omega(\epsilon,\Theta)$, we have
	\begin{equation}
	\inf_{\hat{\theta}}\sup_{\theta\in\Theta}\sup_Q\mathbb{P}_{(\epsilon,\theta,Q)}\left\{L(\hat{\theta},\theta)\geq\omega(\epsilon,\Theta)\right\}\geq c. \label{eq:lowertoprove}
	\end{equation}
	Let us pick $\theta_1,\theta_2$ that are solution of the following program
	$$\max_{\theta_1,\theta_2\in\Theta}L(\theta_1,\theta_2)\quad\text{s.t. }\TV(P_{\theta_1},P_{\theta_2})\leq\epsilon/(1-\epsilon).$$
	Then, there exists $\epsilon'\leq \epsilon$ such that
	$$L(\theta_1,\theta_2)=\omega(\epsilon,\Theta)\quad\text{and}\quad \TV(P_{\theta_1},P_{\theta_2})=\frac{\epsilon'}{1-\epsilon'}.$$
	For these $\theta_1,\theta_2\in\Theta$, let us define density functions
	$$p_{\theta_1}=\frac{dP_{\theta_1}}{d(P_{\theta_1}+P_{\theta_2})},\quad p_{\theta_2}=\frac{dP_{\theta_2}}{d(P_{\theta_1}+P_{\theta_2})}.$$
	Define $Q_1$ and $Q_2$ by their density functions
	$$\frac{dQ_1}{d(P_{\theta_1}+P_{\theta_2})}=\frac{(p_{\theta_2}-p_{\theta_1})\mathbb{I}\{p_{\theta_2}\geq p_{\theta_1}\}}{\TV(P_{\theta_1},P_{\theta_2})},$$
	$$\quad \frac{dQ_2}{d(P_{\theta_1}+P_{\theta_2})}=\frac{(p_{\theta_1}-p_{\theta_2})\mathbb{I}\{p_{\theta_1}\geq p_{\theta_2}\}}{\TV(P_{\theta_1},P_{\theta_2})}.$$
	Let us first check that $Q_1$ and $Q_2$ are probability measures. Since
	$$\int (p_{\theta_2}-p_{\theta_1})\mathbb{I}\{p_{\theta_2}\geq p_{\theta_1}\}=1-\int p_{\theta_1}\wedge p_{\theta_2}=\int(p_{\theta_1}-p_{\theta_2})\mathbb{I}\{p_{\theta_1}\geq p_{\theta_2}\},$$
	and
	$$\int (p_{\theta_2}-p_{\theta_1})\mathbb{I}\{p_{\theta_2}\geq p_{\theta_1}\}+\int(p_{\theta_1}-p_{\theta_2})\mathbb{I}\{p_{\theta_1}\geq p_{\theta_2}\}=2\TV(P_{\theta_1},P_{\theta_2}),$$
	we have
	$$\int (p_{\theta_2}-p_{\theta_1})\mathbb{I}\{p_{\theta_2}\geq p_{\theta_1}\}=\int(p_{\theta_1}-p_{\theta_2})\mathbb{I}\{p_{\theta_1}\geq p_{\theta_2}\}=\TV(P_{\theta_1},P_{\theta_2}),$$
	which implies
	$$\int \frac{dQ_1}{d(P_{\theta_1}+P_{\theta_2})} d(P_{\theta_1}+P_{\theta_2})=\int \frac{dQ_2}{d(P_{\theta_1}+P_{\theta_2})} d(P_{\theta_1}+P_{\theta_2})=1.$$
	Thus, $Q_1$ and $Q_2$ are well-defined probability measures. The least favorable pair in the parameter space is
	$$\mathbb{P}_1=(1-\epsilon')P_{\theta_1}+\epsilon' Q_1,\quad \mathbb{P}_2=(1-\epsilon')P_{\theta_2}+\epsilon' Q_2.$$
	By Lemma \ref{lem:qiantao},
	$$\mathbb{P}_1,\mathbb{P}_2\in\{(1-\epsilon')P_{\theta}+\epsilon'Q:\theta\in\Theta,Q\}\subset \{(1-\epsilon)P_{\theta}+\epsilon Q:\theta\in\Theta,Q\}.$$
	Direct calculation gives
	\begin{eqnarray*}
		\frac{d\mathbb{P}_1}{d(P_{\theta_1}+P_{\theta_2})} &=& (1-\epsilon')p_{\theta_1}+\epsilon'\frac{(p_{\theta_2}-p_{\theta_1})\mathbb{I}\{p_{\theta_2}\geq p_{\theta_1}\}}{\epsilon'/(1-\epsilon')} \\
		&=& (1-\epsilon')\left(p_{\theta_1}+(p_{\theta_2}-p_{\theta_1})\mathbb{I}\{p_{\theta_2}\geq p_{\theta_1}\}\right) \\
		&=& (1-\epsilon')\left(p_{\theta_2}+(p_{\theta_1}-p_{\theta_2})\mathbb{I}\{p_{\theta_1}\geq p_{\theta_2}\}\right) \\
		&=& (1-\epsilon')p_{\theta_2}+\epsilon'\frac{(p_{\theta_1}-p_{\theta_2})\mathbb{I}\{p_{\theta_1}\geq p_{\theta_2}\}}{\epsilon'/(1-\epsilon')} \\
		&=& \frac{d\mathbb{P}_2}{d(P_{\theta_1}+P_{\theta_2})}.
	\end{eqnarray*}
	Hence, $\mathbb{P}_1=\mathbb{P}_2$, which implies the corresponding $\theta_1$ and $\theta_2$ are not identifiable from the model, and their distance under $L(\theta_1,\theta_2)$ can be as far as $\omega(\epsilon,\Theta)$. A standard application of Le Cam's two point testing method \citep{yu1997assouad} leads to (\ref{eq:lowertoprove}) and the proof is complete.
\end{proof}

\begin{proof}[Proof of Theorem \ref{thm:location lower}]
	In this case $L(\theta_1,\theta_2)=\norm{\theta_1-\theta_2}^2$. Therefore,
	\begin{eqnarray*}
		&&\omega(\epsilon,\Theta) \\
		&=& \sup\left\{\norm{\theta_1-\theta_2}^2: \TV(N(\theta_1,I_p),N(\theta_2,I_p))=\epsilon/(1-\epsilon); \theta_1,\theta_2\in\Theta\right\} \\
		&\geq& \sup\left\{\norm{\theta_1-\theta_2}^2: \norm{\theta_1-\theta_2}^2/4\leq \epsilon^2; \theta_1,\theta_2\in\Theta\right\} \\
		&=& 4\epsilon^2.
	\end{eqnarray*}
	Moreover the rate $\mathcal{M}(0)\asymp \frac{p}{n}$ is classical (see, for example, \cite{ma2013volume}). Hence, we have $\mathcal{M}(\epsilon)\asymp (p/n)\vee\epsilon^2$ by Theorem \ref{thm:lower}.
\end{proof}

\begin{proof}[Proofs of Theorem \ref{thm:l1}, Theorem \ref{thm:l2} and Theorem \ref{thm:l3}]
	Without loss of generality, we assume $M>1+\epsilon$. Consider $\Sigma_1=I_p$ and $\Sigma_2=I_p+\epsilon E_{11}$, where $E_{11}$ is a matrix with $1$ in the $(1,1)$-entry and $0$ elsewhere. Note that both $\Sigma_1$ and $\Sigma_2$ are in all matrix classes considered in Section \ref{sec:matrix}. Then,
	$$\TV(N(0,\Sigma_1),N(0,\Sigma_2))^2\leq \frac{1}{2}D(N(0,\Sigma_1)||N(0,\Sigma_2))\leq \frac{1}{8}\fnorm{\Sigma_1-\Sigma_2}^2=\frac{\epsilon^2}{8},$$
	and $L(\Sigma_1,\Sigma_2)=\opnorm{\Sigma_1-\Sigma_2}^2=\epsilon^2$. Therefore, $\omega(\epsilon,\Theta)\geq\epsilon^2$. For the space $\mathcal{F}(M)$, we have $\mathcal{M}(0)\asymp\frac{p}{n}$ according to Theorem 6 of \cite{ma2013volume}. For $\mathcal{F}_k(M)$ and $\mathcal{F}_{\alpha}(M,M_0,M_{\min})$, we have $\mathcal{M}(0)\asymp \frac{k+\log p}{n}$ and $\mathcal{M}(0)\asymp n^{-\frac{2\alpha}{2\alpha+1}}+\frac{\log p}{n}$, respectively, which are implied by Theorem 3 of \cite{cai2010optimal}. Finally, for $\mathcal{F}_s(M)$, $\mathcal{M}(0)\asymp \frac{s\log(ep/s)}{n}$ by Theorem 4 of \cite{cai13}. By Theorem \ref{thm:lower}, we obtain the desired lower bound.
\end{proof}

\begin{proof}[Proof of Theorem \ref{thm:l4}]
	Consider $\Sigma_1=\lambda\theta_1\theta_1^T+I_p$ and $\Sigma_2=\lambda\theta_2\theta_2^T+I_p$. It is obvious that $\Sigma_1,\Sigma_2\in\mathcal{F}_{s,\lambda}(M,1)$. For $r\geq 2$, we may consider some $V\in O(p,r-1)$ with $\supp(V)\subset \{3,4,...,p\}$. Then, let $\Sigma_1=\lambda\theta_1\theta_1^T+\lambda VV^T+I_p$ and $\Sigma_2=\lambda\theta_2\theta_2^T+\lambda VV^T+I_p$. For both cases, we have $\Sigma_1,\Sigma_2\in\mathcal{F}_{s,\lambda}(M,r)$. Since
	$\TV(N(0,\Sigma_1),N(0,\Sigma_2))^2\leq \frac{\lambda^2}{8}\fnorm{\theta_1\theta_1^T-\theta_2\theta_2^T}^2$ and $L(\Sigma_1,\Sigma_2)=\fnorm{\theta_1\theta_1^T-\theta_2\theta_2^T}^2$, we have
	$$\omega(\epsilon,\Theta)\geq\sup\left\{\fnorm{\theta_1\theta_1^T-\theta_2\theta_2^T}^2:  \frac{\lambda^2}{8}\fnorm{\theta_1\theta_1^T-\theta_2\theta_2^T}^2\leq \epsilon^2\right\}\gtrsim \frac{\epsilon^2}{\lambda^2}\wedge c,$$
	for some constant $c>0$. The reason we need $c$ in the above inequality is because $\fnorm{\theta_1\theta_1^T-\theta_2\theta_2^T}^2$ is a bounded loss. By Theorem 3 of \cite{cai2013sparse}, $\mathcal{M}(0)\gtrsim \frac{s\log(ep/s)}{n\lambda^2}$. By Theorem \ref{thm:lower}, we obtain the desired lower bound.
\end{proof}

\renewcommand{\theequation}{F.\arabic{equation}}
\setcounter{equation}{0}
\section{Additional Proofs in Section 6}

\begin{proof}[Proof of Theorem \ref{thm:niubi}]
Let us shorthand $\mathbb{P}_{(\epsilon,\Theta,Q)}$, $\{X_i\}_{i=1}^n$ and $\{Y_i\}_{i=1}^{n_1}$ by $\mathbb{P}$, $X$ and $Y$. For any estimator $\hat{\theta}(\cdot)$, we have
\begin{eqnarray}
\nonumber && \mathbb{P}\left\{L(\hat{\theta}(X),\theta)>\frac{1}{2}A^{-1}\delta\right\} \\
\label{eq:81}&\geq& \mathbb{P}\left\{L(\hat{\theta}(X),\hat{\theta}(Y))>\delta, L(\hat{\theta}(Y),\theta)\leq \frac{1}{2}A^{-1}\delta\right\} \\
\label{eq:82}&\geq& \mathbb{P}\left\{L(\hat{\theta}(X),\hat{\theta}(Y))>\delta\right\}-\mathbb{P}\left\{L(\hat{\theta}(Y),\theta)>\frac{1}{2}A^{-1}\delta\right\},
\end{eqnarray}
where the inequality (\ref{eq:81}) is due to (\ref{eq:triangle}) and the inequality (\ref{eq:82}) is union bound.
The identity $\epsilon=\epsilon(\hat{\theta},\Theta,\delta)$ means that $$\sup_{\theta\in\Theta}\sup_Q\mathbb{P}\left\{L(\hat{\theta}(X),\hat{\theta}(Y))>\delta\right\}>c.$$ 
Hence, by (\ref{eq:82}), we have
\begin{equation}
\sup_{\theta\in\Theta}\sup_Q\mathbb{P}\left\{L(\hat{\theta}(X),\theta)>\frac{1}{2}A^{-1}\delta\right\}+\sup_{\theta\in\Theta}\sup_Q\mathbb{P}\left\{L(\hat{\theta}(Y),\theta)>\frac{1}{2}A^{-1}\delta\right\}\geq c.\label{eq:diao}
\end{equation}
Let us upper bound $\sup_{\theta\in\Theta}\sup_Q\mathbb{P}\left\{L(\hat{\theta}(Y),\theta)>\frac{1}{2}A^{-1}\delta\right\}$ by
\begin{eqnarray}
\nonumber&& \sup_{\theta\in\Theta}\sup_Q\mathbb{P}\left\{L(\hat{\theta}(Y),\theta)>\frac{1}{2}A^{-1}\delta\right\} \\
\label{eq:83}&=& \sup_{\theta\in\Theta}\mathbb{E}P_{\theta}^{n_1}\left\{L(\hat{\theta},\theta)>\frac{1}{2}A^{-1}\delta\right\} \\
\nonumber&\leq& \sup_{\theta\in\Theta}\max_{n_1\geq n/3}P_{\theta}^{n_1}\left\{L(\hat{\theta},\theta)>\frac{1}{2}A^{-1}\delta\right\} + \mathbb{P}\left\{n_1<\frac{n}{3}\right\} \\
\label{eq:84}&\leq& \max_{n_1\geq n/3}\sup_{\theta\in\Theta}P_{\theta}^{n_1}\left\{L(\hat{\theta},\theta)>\frac{1}{2}A^{-1}\delta\right\} + \exp\left(-\frac{n}{18}\right) \\
\label{eq:85}&\leq& \sup_{\theta\in\Theta}P_{\theta}^n\left\{L(\hat{\theta},\theta)>\frac{1}{2}c_1A^{-1}\delta\right\}+ \exp\left(-\frac{n}{18}\right) \\
\label{eq:86}&\leq& \sup_{\theta\in\Theta}\sup_{Q}\mathbb{P}\left\{L(\hat{\theta}(X),\theta)>\frac{1}{2}c_1A^{-1}\delta\right\}+ \exp\left(-\frac{n}{18}\right).
\end{eqnarray}
In the equality (\ref{eq:83}), the expectation operator $\mathbb{E}$ is associated with the probability $n_1\sim\text{Binomial}(n,1-\epsilon)$. The inequality (\ref{eq:84}) is by Hoeffding's inequality and the assumption $\epsilon<1/2$. The inequality (\ref{eq:85}) is by the assumption (\ref{eq:notcrazy}). Finally, the inequality (\ref{eq:86}) is due to the relation $\{P_{\theta}:\theta\in\Theta\}\subset\{(1-\epsilon)P_{\theta}+\epsilon Q:\theta\in\Theta,Q\}$. Combining the above argument with (\ref{eq:diao}) and the inequality
\begin{eqnarray*}
&&\sup_{\theta\in\Theta}\sup_Q\mathbb{P}\left\{L(\hat{\theta}(X),\theta)>\frac{1}{2}c_1A^{-1}\delta\right\}\\
&\geq&\sup_{\theta\in\Theta}\sup_Q\mathbb{P}\left\{L(\hat{\theta}(X),\theta)> \frac{1}{2}A^{-1}\delta\right\},
\end{eqnarray*} 
we get
$$2 \sup_{\theta\in\Theta}\sup_{Q}\mathbb{P}\left\{L(\hat{\theta}(X),\theta)>\frac{1}{2}c_1A^{-1}\delta\right\}\geq c-\exp\left(-\frac{n}{18}\right),$$
which leads to the desired conclusion for a sufficiently large $n$.
\end{proof}

\renewcommand{\theequation}{G.\arabic{equation}}
\setcounter{equation}{0}

\section{Proofs of Propositions and Lemmas}
\begin{proof}[Proof of Proposition \ref{prop:truth}]
For any $u\in S^{p-1}$, we have
$\mathcal{D}_u(\beta\Sigma,P_{\Sigma})=P_{\Sigma}\left(|u^TX|^2\leq \beta u^T\Sigma u\right)\wedge P_{\Sigma}\left(|u^TX|^2 \geq \beta u^T\Sigma u\right)$, which equals
$\left(2\Phi(\sqrt{\beta})-1\right)\wedge\left(2-2\Phi(\sqrt{\beta})\right)$ or one because that either $u^TX/\sqrt{u^T\Sigma u}\sim N(0,1)$ or $\mathbb{P}\{u^TX=0\}=1$ with $u^T\Sigma u=0$. Since $\Phi(\sqrt{\beta})=3/4$, we have $\left(2\Phi(\sqrt{\beta})-1\right)\wedge\left(2-2\Phi(\sqrt{\beta})\right)=1/2$. Note that we assume $\Sigma$ is not a zero matrix throughout the paper, there is at least one $u\in S^{p-1}$ such that $u^TX/\sqrt{u^T\Sigma u}\sim N(0,1)$. Thus we finally obtain that $\mathcal{D}_{\mathcal{U}}(\beta\Sigma,P_{\Sigma})=\inf_{u\in\mathcal{U}}\mathcal{D}_u(\beta\Sigma,P_{\Sigma})=1/2$.
\end{proof}

\begin{proof}[Proof of Proposition \ref{prop:unique}]
The existence is guaranteed by the continuity.
Suppose there are two canonical representation, then equivalently, $G(t)=\frac{1}{2}$ will have another solution besides $t=1$. However, $G(t)=G(1)$ for some $t\neq 1$ contradicts (\ref{eq:assell}). This completes the proof.
\end{proof}

\begin{proof}[Proof of Proposition \ref{prop:deepell}]
For any $u\in S^{p-1}$ such that $u^T\Gamma u \neq 0$, we have $\mathcal{D}_u(\Gamma,P_{\Gamma})=P_{\Gamma}\left(|u^TX|^2\leq  u^T\Gamma u\right)\wedge P_{\Gamma}\left(|u^TX|^2\geq  u^T\Gamma u\right)$, which equals $G(1)\wedge (1-G(1^-))$ by the definition of $G(t)$. According to the definition of canonical representation, $G(1)=1/2$ and thus the proof is complete.
\end{proof}

\begin{proof}[Proof of Proposition \ref{prop:example}]
For $X\sim EC_p(0,\Gamma,\eta)$, its characteristic function must be in the form $\mathbb{E}\exp(\sqrt{-1}t^TX)=\phi(t^T\Gamma t)$ with some univariate function $\phi(\cdot)$ called characteristic generator. The characteristic generator $\phi(\cdot)$ is completely determined by the distribution of $\eta$ \citep{fang1990symmetric}. For multivariate Gaussian, $\phi(v)=\exp(-v/(2\beta))$. For multivariate Laplace, $\phi(v)=1/(1+\beta v/2)$. For multivariate $t$, $\phi(v)=\frac{(\beta vd)^{d/4}}{2^{d/2-1}\Gamma (d/2)}K_{d/2}(\sqrt{\beta vd})$, where $K_m(\cdot)$ is the modified Bessel of the second kind. Note that for all the examples considered, $\phi(\cdot)$ does not depend on $\Gamma$ or the dimension $p$, which means the distribution of $\eta$ does not depend on $p$, either. Since the distribution of $\eta$ fully determines the function $G(\cdot)$ defined by (\ref{eq:G(t)}), the equation $G(1)=1/2$ is satisfied for some constant $\beta$ independent of $p$. Finally, the condition (\ref{eq:assell}) is satisfied for some constants $\tau,\alpha,\kappa$ independent of $p$ because of the smoothness of the derivative of $G(\cdot)$.
\end{proof}

\begin{proof}[Proof of Proposition \ref{prop:exampleU}]
Let $\phi(\cdot)$ be the characteristic generator of $X_1$. Note that we have $\mathbb{E}\exp(\sqrt{-1}t^T(X_1-X_2))=[\phi(t^T\Gamma t)]^2$. Since the continuity of $\phi(\cdot)$ implies the continuity of $[\phi(\cdot)]^2$, the same argument in the proof of Proposition \ref{prop:example} leads to the desired conclusion.
\end{proof}

\begin{proof}[Proof of Lemma \ref{lem:ratio}]
Note that $n_2\sim\text{Binomial}(n,\epsilon)$. By Hoeffding's inequality, $\mathbb{P}(n_2>n\epsilon+t)\leq\exp(-2t^2/n)$ for any $t>0$. Thus, $n_2\leq n\epsilon+\sqrt{\frac{n}{2}\log(1/\delta)}$ with probability at least $1-\delta$. This implies $n_1\geq n(1-\epsilon)-\sqrt{\frac{n}{2}\log(1/\delta)}$ and therefore,
$$\frac{n_2}{n_1}\leq \frac{\epsilon+\sqrt{\frac{1}{2n}\log(1/\delta)}}{1-\epsilon-\sqrt{\frac{1}{2n}\log(1/\delta)}},$$
with probability at least $1-\delta$. Under the assumption that $\epsilon<1/5$ and $\sqrt{\frac{1}{2n}\log(1/\delta)}<1/5$, we have $n_2/n_1\leq \epsilon/(1-\epsilon)+\frac{25}{12}\sqrt{\frac{1}{2n}\log(1/\delta)}$ with probability at least $1-\delta$. This proves (\ref{eq:ratio}). A symmetric argument leads to
$$\frac{n_2}{n_1}\geq \frac{\epsilon-\sqrt{\frac{1}{2n}\log(1/\delta)}}{1-\epsilon+\sqrt{\frac{1}{2n}\log(1/\delta)}},$$
with probability at least $1-\delta$. For $\delta=1/2$ and $\epsilon^2>1/n$, we have $n_2/n_1\geq c\epsilon$ with probability at least $1/2$. Thus, the proof is complete.
\end{proof}

\begin{proof}[Proof of Lemma \ref{lem:qiantao}]
For any $\theta\in\Theta$ and $Q$, define
$$Q'=\frac{\epsilon_2-\epsilon_1}{\epsilon_2}Q+\frac{\epsilon_1}{\epsilon_2}P_{\theta}.$$
It is easy to see that $Q'$ is a probability measure, and it satisfies
$$(1-\epsilon_1)P_{\theta}+\epsilon_1Q=(1-\epsilon_2)P_{\theta}+\epsilon_2Q'.$$
This immediately gives the desired conclusion.
\end{proof}

\begin{proof}[Proof of Lemma \ref{lem:DKW}]

The proof follows from standard empirical processes properties \cite{dudley1978central,pollard2012convergence} for VC classes \cite{vapnik1971uniform}. In order to bound the main term $\sup_{u\in \mathcal{V}_{S},t\in \mathbb{R}%
}\left\vert \mathbb{P}(IH_{u,t})-\mathbb{P}_{n}(IH_{u,t})\right\vert $, we bound $$%
\sup_{u\in \mathcal{V}_{S},t\in \mathbb{R}}\left\vert \mathbb{P}(IH_{u,t})-\mathbb{P}%
_{n}(IH_{u,t})\right\vert -\mathbb{E}\sup_{u\in \mathcal{V}_{S},t\in \mathbb{%
		R}}\left\vert \mathbb{P}(IH_{u,t})-\mathbb{P}_{n}(IH_{u,t})\right\vert, $$ and $\mathbb{%
	E}\sup_{u\in \mathcal{V}_{S},t\in \mathbb{R}}\left\vert \mathbb{P}(IH_{u,t})-\mathbb{P%
}_{n}(IH_{u,t})\right\vert $ separately. 

We first apply McDiarmid's bounded difference inequality \cite{mcdiarmid1989method}, \cite[Theorem~2.2]{devroye2012combinatorial} to obtain that with
probability at least $1-e^{-nt^{2}/2}$,
\begin{equation}
\sup_{u\in \mathcal{V}_{S},t\in \mathbb{R}}\left\vert \mathbb{P}(IH_{u,t})-\mathbb{P}%
_{n}(IH_{u,t})\right\vert -\mathbb{E}\sup_{u\in \mathcal{V}_{S},t\in \mathbb{%
		R}}\left\vert \mathbb{P}(IH_{u,t})-\mathbb{P}_{n}(IH_{u,t})\right\vert <t,
\label{eq:Mcdiarmid bound}
\end{equation}%
for any $t>0$. To bound the term $\mathbb{E}\sup_{u\in \mathcal{V}_{S},t\in 
	\mathbb{R}}\left\vert \mathbb{P}(IH_{u,t})-\mathbb{P}_{n}(IH_{u,t})\right\vert $, we
need some notation on covering number. Collect all $IH_{u,t}$ in a class of
sets $\mathcal{A}=\{IH_{u,t}:u\in \mathcal{V}_{S},t\in \mathbb{R}\}$. Define
a set of binary vectors by $\mathcal{A}(x_{1}^{n})=\{(b)\in \mathbb{R}%
^{n}:\exists IH_{u,t}\in \mathcal{A}$ s.t. $b_{i}=\mathbb{I}\{x_{i}\in
IH_{u,t}\}$ for all $i\in \lbrack n]\}$. We say a set $\mathcal{B}_{r}$ of
binary vectors is a cover of $\mathcal{A}(x_{1}^{n})$ with radius $r>0$, if
for any $b\in \mathcal{A}(x_{1}^{n})$ there exists some $b_{0}\in \mathcal{B}%
_{r}$ such that the normalized Hamming distance $\rho (b,b_{0}):=\sqrt{%
	\sum_{i=1}^{n}\mathbb{I}\{b_{i}\neq b_{0i}\}/n}\leq r$. Finally we define
covering number $\mathcal{N}(r,\mathcal{A}(x_{1}^{n}))$ as the cardinality
of the smallest cover of $\mathcal{A}(x_{1}^{n})$ with radius $r$. The
following result is a simple version of Dudley's metric entropy bound \cite{dudley1978central}, \cite[Theorem~3.2]{devroye2012combinatorial},
\begin{equation}
\mathbb{E}\sup_{u\in \mathcal{V}_{S},t\in \mathbb{R}}\left\vert \mathbb{P}(IH_{u,t})-%
\mathbb{P}_{n}(IH_{u,t})\right\vert \leq \frac{24}{\sqrt{n}}%
\max_{x_{1},\ldots ,x_{n}}\int_{0}^{1}\sqrt{\log 2\mathcal{N}(r,\mathcal{A}%
	(x_{1}^{n}))}dr. \label{eq:Dudley integral}
\end{equation}
It suffices to provide an upper bound of the covering number $\mathcal{N}(r,%
\mathcal{A}(x_{1}^{n}))$. To this end, we claim that the class of sets $%
\mathcal{A}$ has VC dimension no more than $3+2|S|$. The following result in
\cite{dudley1978central} (see also \cite[Theorem~4.3]{devroye2012combinatorial}, and \cite{haussler1995sphere} for a refined constant) then relates the VC dimension of $\mathcal{A}$ to the covering numbers $%
\mathcal{N}(r,\mathcal{A}(x_{1}^{n}))$, 
\begin{equation}
\mathcal{N}(r,\mathcal{A}(x_{1}^{n}))\leq \left( 4e/r^{2}\right)
^{(3+2|S|)/(1-e^{-1})}.  \label{eq:Dudley VC}
\end{equation}%
We thus are able to combine (\ref{eq:Dudley integral}) and (\ref{eq:Dudley
	VC}) to obtain a bound with explicit constants for the term $\mathbb{E}\sup_{u\in \mathcal{V}_{S},t\in \mathbb{%
		R}}\left\vert \mathbb{P}(IH_{u,t})-\mathbb{P}_{n}(IH_{u,t})\right\vert $,%
\begin{equation}
\mathbb{E}\sup_{u\in \mathcal{V}_{S},t\in \mathbb{R}}\left\vert \mathbb{P}(IH_{u,t})-%
\mathbb{P}_{n}(IH_{u,t})\right\vert \leq \sqrt{\frac{1440e\pi }{1-e^{-1}}}%
\sqrt{\frac{3+2|S|}{n}},  \label{eq:Expectation bound}
\end{equation}%
where we have used the transformation $r=\sqrt{5e}e^{-u/2}$ and gamma
integral $\int_{0}^{\infty }\sqrt{u}e^{-u}du=\sqrt{\pi /4}$. In the end, we
combine (\ref{eq:Mcdiarmid bound}) with $t=\sqrt{\log (1/\delta )/(2n)}$ and
(\ref{eq:Expectation bound}) to obtain the desired result.

It remains to show that $\mathcal{A}$ has VC dimension no more than $3+2|S|$%
. Define a half space $H_{u,t}=\{y:u^{T}y<t\}$ associated with some unit
vector $u$ and some $t\in \mathbb{R}$. We note that $\mathcal{A}\subset 
\mathcal{A}_{0}:=\{A:A=B\cap C$ for some $B\in \mathcal{H}_{S}$ and $C\in 
\mathcal{H}_{S}\}$, where the class of sets $\mathcal{H}_{S}=\{H_{u,t}:u\in 
\mathcal{V}_{S},t\in \mathbb{R}\}=\{\{y:g(y)\geq 0\}:g\in \mathcal{G}\}$
consists of certain half spaces in $\mathbb{R}^{p}$ and $\mathcal{G}$ is a $%
|S|+1$ dimensional linear space. It follows from \cite{steele1978existence} (see, also \cite{dudley1978central}, and \cite[Lemma~4.2]{devroye2012combinatorial}) that $\mathcal{H}_{S}$ has VC dimension no more than $|S|+1$ and thus $%
\mathcal{A}$ has VC dimension no more than $2|S|+3$ by \cite{dudley1978central}, which
completes the proof.
\end{proof}

\begin{proof}[Proof of Proposition \ref{thm:matrix upper_app}]
For any symmetric matrix $S$, we have
$$|u^TSu|-|v^TSu|\leq |(u-v)^TS(u+v)|\leq\|u-v\|\opnorm{S}\|u+v\|.$$ Note that $\bar{\mathcal{U}}$ is taken to be the $(1/2)$-net of $S^{p-1}$. We first show that if both $u$ and $v$ are unit vectors and $\|u-v\|\leq\frac{1}{2}$, then we have $\|u-v\|\|u+v\|\leq \sqrt{15}/4$. Indeed, without loss of generality, we assume $u=(1,0,\ldots,0)^T$ and $v=(v_1,\ldots,v_p)^T$ with $\sum_{i=1}^pv_i^2=1$. Then the aim is to optimize $\|u-v\|^2\|u+v\|^2=4(1-v_1^2)$ given the constraint that $\|u-v\|^2=2-2v_1\leq\frac{1}{4}$. The solution is simply that $v_1=\frac{7}{8}$, which implies that the maximum value of $\|u-v\|^2\|u+v\|^2$ is $\frac{15}{16}$. Thus we have shown that $\|u-v\|\|u+v\|\leq \sqrt{15}/4$. Consequently, we have that $$\opnorm{\hat{\bar{\Sigma}}-\Sigma}\leq\max_{u\in\bar{\mathcal{U}}}|u^T(\hat{\Sigma}-\Sigma)u|+\sqrt{15}/4\opnorm{\hat{\bar{\Sigma}}-\Sigma},$$ which implies $\opnorm{\hat{\bar{\Sigma}}-\Sigma}\leq (1-\sqrt{15}/4)^{-1}\max_{u\in\bar{\mathcal{U}}}|u^T(\hat{\Sigma}-\Sigma)u|$.

It suffices to show that under the assumptions of Proposition \ref{thm:matrix upper_app}, there exists some absolute constant $C>0$ such that with probability at least $1-2\delta$ uniformly over all $Q$ and $\Sigma \in \mathcal{F}(M)$,
\begin{equation}
\max_{u\in\bar{\mathcal{U}}}|u^T(\hat{\Sigma}-\Sigma)u| \leq C\left(
\left(\frac{p}{n}\vee\epsilon^2\right)+\frac{\log(1/\delta)}{n}\right). \label{eq:general union bound}
\end{equation}
The proof of (\ref{eq:general union bound}) follows from a similar analysis of Theorem \ref{thm:general} with $\mathcal{U}$ replaced by the $\bar{\mathcal{U}}$, the $(1/2)$-net of $S^{p-1}$. By inspection of the proof of Theorem \ref{thm:general}, it suffices to show the following equation similar to (\ref{eq:dkwU}) holds with probability at least $1-\delta$ conditioning on the set of $(n_1,n_2)$ satisfying (\ref{eq:ratio}),
\begin{equation}
\sup_{\Gamma\in\mathcal{F}}\left|\mathcal{D}_{\bar{\mathcal{U}}}(\Gamma,P_{\Sigma})-\mathcal{D}_{\bar{\mathcal{U}}}(\Gamma,\{Y_i\}_{i=1}^{n_1})\right|\leq \sqrt{\frac{5\log p+\log(2/\delta)}{2n_1}}. \label{eq:general union bound2}
\end{equation} 
We note that 
\begin{eqnarray*}
	&& \sup_{\Gamma\in\mathcal{F}}\left|\mathcal{D}_{\bar{\mathcal{U}}}(\Gamma,P_{\Sigma})-\mathcal{D}_{\bar{\mathcal{U}}}(\Gamma,\{Y_i\}_{i=1}^{n_1})\right| \\
	&\leq& \sup_{\Gamma\in\mathcal{F}}\max_{u\in\bar{\mathcal{U}}}\left|\mathcal{D}_u(\Gamma,P_{\Sigma})-\mathcal{D}_u(\Gamma,\{Y_i\}_{i=1}^{n_1})\right| \\
	&\leq& \max_{u\in\bar{\mathcal{U}}}\sup_{t\in\mathbb{R}}\left|\frac{1}{n_1}\sum_{i=1}^{n_1}\mathbb{I}\{|u^TY_i|^2\leq t\}-P_{\Sigma}(|u^TY|^2\leq t)\right|.
\end{eqnarray*}
Applying the DKW inequality with tight constant in \cite{massart1990tight} and union bound ranging over all $|\bar{\mathcal{U}}|\leq 5^p$ directions in $\bar{\mathcal{U}}$ into above equation, we obtain (\ref{eq:general union bound2}), which completes the proof.
\end{proof}

\begin{proof}[Proof of Lemma \ref{lem:UDKW}]
Without loss of generality, assume $n$ is even. The case when $n$ is odd can be done via a slight modification of the same argument. Define
$$V_{u,t}(w_1,...,w_n)=\frac{2}{n}\sum_{i=1}^{n/2}\mathbb{I}\{|u^T(w_{2i-1}-w_i)|\leq t\}.$$
and
$$R(W_1,...,W_n)=\sup_{u\in\mathcal{V}_S,t\in \mathbb{R}}\left|\mathbb{U}_n(IH_{u,t})-\mathbb{P}(IH_{u,t})\right|.$$
In order to bound $R(W_1,...,W_n)$, we are going to bound $R(W_1,...,W_n)-\mathbb{E}R(W_1,...,W_n)$ and $\mathbb{E}R(W_1,...,W_n)$ separately.
For the first term, observe that for all $i \in [n]$,
$$\sup_{w_1,...,w_n,w_i'}|R(w_1,...,w_{i-1},w_i,w_{i+1},w_n)-R(w_1,...,w_{i-1},w_i',w_{i+1},w_n)|\leq\frac{1}{n}.$$
Then, by McDiarmid's bounded difference inequality \cite{mcdiarmid1989method}, we have
$$R(W_1,...,W_n)-\mathbb{E}R(W_1,...,W_n)\leq t,$$
with probability at least $1-e^{-2nt^2}$ for any $t>0$.
For the second term, we have
\begin{eqnarray*}
&& \mathbb{E}R(W_1,...,W_n) \\
&=& \mathbb{E}\sup_{u\in\mathcal{V}_S,t\in \mathbb{R}}\left|\mathbb{U}_n(IH_{u,t})-\mathbb{P}(IH_{u,t})\right| \\
&=& \mathbb{E}\sup_{u\in\mathcal{V}_S,t\in \mathbb{R}}\left|\frac{1}{n!}\sum_{\sigma}\left(V_{u,t}(W_{\sigma(1)},...,W_{\sigma(n)})-\mathbb{E}V_{u,t}(W_{\sigma(1)},...,W_{\sigma(n)})\right)\right| \\
&\leq& \frac{1}{n!}\sum_{\sigma}\mathbb{E}\sup_{u\in\mathcal{V}_S,t\in \mathbb{R}}\left|V_{u,t}(W_{\sigma(1)},...,W_{\sigma(n)})-\mathbb{E}V_{u,t}(W_{\sigma(1)},...,W_{\sigma(n)})\right| \\
&=& \mathbb{E}\sup_{u\in\mathcal{V}_S,t\in \mathbb{R}}\left|V_{u,t}(W_{\sigma(1)},...,W_{\sigma(n)})-\mathbb{E}V_{u,t}(W_{\sigma(1)},...,W_{\sigma(n)})\right|.
\end{eqnarray*}
The summation of $\sigma$ above is over all possible permutations. By replacing $n$ and each $X_i$ in the proof of Lemma \ref{lem:DKW} by $n/2$ and $W_{2i-1}-W_i$ respectively, we are able to bound the above expectation using (\ref{eq:Expectation bound}), i.e.,
\begin{equation*}
\mathbb{E}R(W_1,\ldots,W_n)  \leq \sqrt{\frac{1440e\pi }{1-e^{-1}}}\sqrt{\frac{3+2|S|}{n/2}}. 
\end{equation*}
Finally, combining the two bounds and picking $t=\sqrt{\frac{\log (1/\delta )}{2n}}$, we complete the proof.
\end{proof}

\renewcommand{\theequation}{H.\arabic{equation}}
\setcounter{equation}{0}
\section{Additional Proofs in Appendix A}

We first state an extension of the concentration inequality for the suprema of the empirical process in Lemma \ref{lem:DKW} to its corresponding U-process. Recall $IH_{u,t }=\{y:|u^{T}y|\leq t \}$, and for any $S\subset [p]$, $\mathcal{V}_{S}=\{u=(u_i)\in S^{p-1}: u_i=0\text{ if }i\notin S\}$. 
\begin{lemma}\label{lem:UDKW}
	For i.i.d. real-valued data $W_{1},\ldots,W_{n}\sim\mathbb{P}$, we have for any $S\subset [p]$,
	$$\sup_{u\in \mathcal{V}_{S},t \in \mathbb{R} }\left\vert \mathbb{P}(IH_{u,t })-\mathbb{U}%
	_{n}(IH_{u,t })\right\vert \leq \sqrt{\frac{1440e\pi}{1-e^{-1}}}\sqrt{\frac{3+2|S|}{n/2}}+%
	\sqrt{\frac{\log (1/\delta )}{2n}},$$ with probability at least $1-\delta$, where $\mathbb{U}_{n}(IH_{u,t })={n\choose 2}^{-1}\sum_{i<j}\mathbb{I}\{|u^T(W_i-W_j)| \leq t\}$ denotes the U-process of $\{W_{i}\}_{i=1}^{n}$ evaluated with the kernel function $\mathbb{I}\{|u^T(W_1-W_2)| \leq t\}$.
\end{lemma}

Now we are ready to state the proof of the two theorems.
\begin{proof}[Proof of Theorem \ref{thm:U} and Theorem \ref{thm:Ue}]
	It is sufficient to establish the results of Theorem \ref{thm:general} and Theorem \ref{thm:general elliptical} for the U-version matrix depth function. This can be done through the same arguments used in the proofs of Theorem \ref{thm:general} and Theorem \ref{thm:general elliptical} on the observation pairs. An analogous inequality to (\ref{eq:dkwU}) can be established using Lemma \ref{lem:UDKW}. We omit the details here.
\end{proof}


\end{document}